\documentclass[12pt]{amsart}
\usepackage{amsmath,amsthm,enumerate}
\usepackage{mathdots,bm,url,amscd}
\usepackage{amsfonts,amssymb,mathscinet}
\usepackage{a4wide}
\usepackage[all]{xy}
\usepackage[stable]{footmisc}

\begin{document}

\title{A conjecture on Whittaker--Fourier coefficients of cusp forms}

\author{Erez Lapid}
\address{Institute of Mathematics, The Hebrew University of Jerusalem, Jerusalem 91904, Israel and
Department of Mathematics, Weizmann Institute of Science, Rehovot 76100 Israel}
\email{erez.m.lapid@gmail.com}
\author{Zhengyu Mao}
\address{Department of Mathematics and Computer Science, Rutgers University, Newark, NJ 07102, USA}
\email{zmao@rutgers.edu}
\date{\today}
\thanks{Authors partially supported by U.S.--Israel Binational Science Foundation Grant \# 057/2008}
\thanks{Second named author partially supported by NSF grant DMS 1000636 and by a fellowship from the Simons Foundation}
\dedicatory{To the memory of Stephen Rallis}

\begin{abstract}
We formulate an analogue of the Ichino--Ikeda conjectures for the Whittaker--Fourier coefficients of
automorphic forms on quasi-split reductive groups. This sharpens the conjectures of Sakellaridis--Venkatesh
in the case at hand.
\end{abstract}

\maketitle

\newcommand{\A}{\mathbb{A}}                            
\newcommand{\C}{\mathbb{C}}                            
\newcommand{\Q}{\mathbb{Q}}                            
\newcommand{\R}{\mathbb{R}}                            
\newcommand{\Z}{\mathbb{Z}}                            
\newcommand{\AF}{\mathcal{A}}                          
\newcommand{\bs}{\backslash}
\newcommand{\Nm}{\operatorname{N}}
\newcommand{\BC}{\operatorname{BC}}
\newcommand{\Asai}{\operatorname{As}}
\newcommand{\Hom}{\operatorname{Hom}}
\newcommand{\vol}{\operatorname{vol}}
\newcommand{\cusp}{\operatorname{cusp}}
\newcommand{\disc}{\operatorname{disc}}
\newcommand{\temp}{\operatorname{temp}}
\newcommand{\dsum}{\oplus}
\newcommand{\ab}{\operatorname{ab}}
\newcommand{\OO}{\mathcal{O}}                         
\newcommand{\Gal}{\operatorname{Gal}}
\newcommand{\K}{\mathbf{K}}                             
\newcommand{\stint}{\int^{st}}                         
\newcommand{\Mpsi}[1]{\psi^{w_{#1}}_{N_{#1}}}
\newcommand{\rst}[1]{{#1}_\#}
\newcommand{\Res}{\operatorname{Res}}
\newcommand{\imgset}[1]{\Pi^{#1}(\Res_{E/F}\GL_m)}
\newcommand{\canArthur}{\mathcal{H}}
\newcommand{\mainconst}[1]{c_{#1}^{\psi_N}}
\newcommand{\loc}{\operatorname{loc}}
\newcommand{\SC}{\operatorname{sc}}
\newcommand{\Img}{\operatorname{Im}}
\newcommand{\simw}{\sim_w}
\newcommand{\dual}[1]{#1^{\vee}}
\newcommand{\alt}[1]{#1^{\wedge}}
\newcommand{\dW}{\dual{W}}           
\newcommand{\Mp}{\widetilde{\operatorname{Sp}}}
\newcommand{\Ad}{\operatorname{Ad}}
\newcommand{\GL}{\operatorname{GL}}
\newcommand{\SO}{\operatorname{SO}}
\newcommand{\Orth}{\operatorname{O}}
\newcommand{\Gspin}{\operatorname{Gspin}}
\newcommand{\U}{\operatorname{U}}
\newcommand{\GU}{\operatorname{GU}}
\newcommand{\Sp}{\operatorname{Sp}}
\newcommand{\SL}{\operatorname{SL}}
\newcommand{\mira}{\mathcal{P}}                     
\newcommand{\sym}{\operatorname{sym}}
\newcommand{\Ind}{\operatorname{Ind}}                   
\newcommand{\ind}{\operatorname{ind}}                   
\newcommand{\res}{\operatorname{res}}                   
\newcommand{\AI}{\operatorname{AI}}
\newcommand{\stdG}[1]{t_{#1}}
\newcommand{\Langlands}{\mathcal{L}}
\newcommand{\params}{\Psi}
\newcommand{\tparams}{\params^t}
\newcommand{\id}{\operatorname{id}}
\newcommand{\cent}{\mathcal{S}}
\newcommand{\csgr}{\mathcal{CSGR}}                     
\newcommand{\Irr}{\operatorname{Irr}}                  
\newcommand{\JF}{\mathbf{j}}                           
\newcommand{\rest}{\big|}                              
\newcommand{\Cusp}{\operatorname{Cusp}}                
\newcommand{\gen}{\operatorname{gen}}                  
\newcommand{\sqr}{\operatorname{sqr}}                  
\newcommand{\unr}{\operatorname{unr}}                  
\newcommand{\PGL}{\operatorname{PGL}}
\newcommand{\Ker}{\operatorname{Ker}}
\newcommand{\Coker}{\operatorname{Coker}}
\newcommand{\mc}{\operatorname{MC}}           
\newcommand{\modulus}{\delta}
\newcommand{\whit}{\mathcal{W}}                          
\newcommand{\Whit}{\mathbb{W}}                           
\newcommand{\tr}{\operatorname{tr}}                        
\newcommand{\inner}[2]{\left(#1,#2\right)}
\newcommand{\sprod}[2]{\left\langle#1,#2\right\rangle}
\newcommand{\abs}[1]{\left|{#1}\right|}
\newcommand{\norm}[1]{\lVert#1\rVert}
\newcommand{\sm}[4]{\left(\begin{smallmatrix}{#1}&{#2}\\{#3}&{#4}\end{smallmatrix}\right)}
\newcommand{\der}{\operatorname{der}}                     
\newcommand{\unitcirc}{\mathbb{S}^1}
\newcommand{\Spec}{\operatorname{Spec}}
\newcommand{\et}{\operatorname{\acute et}}
\newcommand{\flf}{\operatorname{fl}}
\newcommand{\Gad}{\operatorname{ad}}

\newtheorem{theorem}{Theorem}[section]
\newtheorem{lemma}[theorem]{Lemma}
\newtheorem{proposition}[theorem]{Proposition}
\newtheorem{remark}[theorem]{Remark}
\newtheorem{conjecture}[theorem]{Conjecture}
\newtheorem{definition}[theorem]{Definition}
\newtheorem{corollary}[theorem]{Corollary}
\newtheorem{example}[theorem]{Example}


\numberwithin{equation}{section}

\setcounter{tocdepth}{1}
\tableofcontents
\section{Introduction}
Let $G$ be a quasi-split reductive group over a number field $F$ and $\A$ the ring of adeles of $F$.
Let $B$ be a Borel subgroup of $G$ defined over $F$, $N$ the unipotent radical
of $B$ and fix a non-degenerate character $\psi_N$ of $N(\A)$, trivial on $N(F)$.
For a cusp form $\varphi$ of $G(F)\bs G(\A)$ we consider the Whittaker--Fourier coefficient\footnote{The Haar
measure is normalized so that $\vol(N(F)\bs N(\A))=1$}
\[
\whit(\varphi)=\whit^{\psi_N}(\varphi):=\int_{N(F)\bs N(\A)}\varphi(n)\psi_N(n)^{-1}\ dn.
\]
If $\pi$ is an irreducible cuspidal representation then $\whit$, if non-zero, gives a realization of
$\pi$ in the space of Whittaker functions on $G(\A)$, which by local multiplicity one
depends only on $\pi$ as an abstract representation.
It therefore provides a useful tool for understanding $\pi$, both computationally and
conceptually. It is natural to study the size of $\whit(\varphi)$.
For the general linear group, the theory of Rankin--Selberg integrals, developed in higher rank
by Jacquet, Piatetski-Shapiro and Shalika, expresses, among other things, the Petersson inner product in terms
of a canonical inner product on the local Whittaker model of $\pi$. (See e.g.~\cite{MR1816039}.)
The global factor which shows up in this expression is the residue at $s=1$ of $L(s,\pi\otimes\pi^\vee)$
(or alternatively,\footnote{We caution that some authors refer to the adjoint $L$-function
as the quotient of $L(s,\pi\otimes\pi^\vee)$ by $\zeta_F(s)$} the adjoint $L$-function of $\pi$)
where $\pi^\vee$ is the contragredient of $\pi$. Dually, $\abs{\whit(\varphi)}^2$ is related to
$(\res_{s=1}L(s,\pi,\Ad))^{-1}$ (assuming certain $L^2$-normalization).

One way to try to make this more precise, and to generalize it to other groups,
is to take the $\psi_N$-th Fourier coefficient
of a matrix coefficient of $\pi$ and to relate it to the product of Whittaker functions.
The integral
\[
\int_{N(\A)}(\pi(n)\varphi,\varphi^\vee)_{G(F)\bs G(\A)^1}\psi_N(n)^{-1}\ dn
\]
does not converge. In fact, even the local integrals
\begin{equation} \label{eq: Fourier MC}
I_v(\varphi,\varphi^\vee)=\int_{N(F_v)}(\pi_v(n_v)\varphi_v,\varphi_v^\vee)_v\psi_N(n_v)^{-1}\ dn_v
\end{equation}
(where $\varphi_v$, $\varphi_v^\vee$ are now taken from the space of $\pi_v$ and $\pi_v^\vee$ respectively
and $(\cdot,\cdot)_v$ is the canonical pairing)
do not converge unless $\pi_v$ is square-integrable.
However, it is possible to regularize $I_v$, and by the Casselman--Shalika formula, almost everywhere
we have $I_v(\varphi_v,\varphi_v^\vee)=\Delta_{G,v}(1)L(1,\pi_v,\Ad)^{-1}$ if
$\varphi_v$, $\varphi_v^\vee$ are unramified vectors with $(\varphi_v,\varphi_v^\vee)=1$
and $\Delta_{G,v}(1)$ is a certain $L$-factor depending on $G_v$ but not $\pi_v$.
By local multiplicity one there exists a constant $\mainconst{\pi}$ depending on $\pi$ such that
\begin{equation} \label{eq: gllc}
\whit^{\psi_N}(\varphi)\whit^{\psi_N^{-1}}(\varphi^\vee)
=(\mainconst{\pi}\vol(G(F)\bs G(\A)^1))^{-1}\frac{\Delta_G^S(s)}{L^S(s,\pi,\Ad)}\big|_{s=1}\prod_{v\in S}I_v(\varphi_v,\varphi_v^\vee)
\end{equation}
for all $\varphi=\otimes\varphi_v\in\pi$, $\varphi^\vee=\otimes_v\varphi_v^\vee\in\pi^\vee$
and all $S$ sufficiently large.
Implicit here is the existence and non-vanishing of $\frac{\Delta_G^S(s)}{L^S(s,\pi,\Ad)}\big|_{s=1}$,
(or equivalently, of $\lim_{s\rightarrow 1}(s-1)^lL^S(s,\pi,\Ad)$
where $l$ is the dimension of the split part of the center of $G$).

The Rankin--Selberg theory for $\GL_n$ alluded to above shows that
$\mainconst{\pi}=1$ for any irreducible cuspidal representations of $\GL_n$. (See \S\ref{sec: GLn}.
A similar result was proved independently by Sakellaridis--Venkatesh \cite{1203.0039}.)

It is desirable to extend this relation to other quasi-split groups.
The first problem is that $\mainconst{\pi}$ depends on the automorphic realization of $\pi$.
Therefore, in cases where there is no multiplicity one, it is not clear which $\pi$'s to take.

There are (at least) two ways to approach this problem.
One way is to use the notion of \emph{$\psi_N$-generic spectrum} as defined by Piatetski-Shapiro \cite{MR546599}.
It is the orthogonal complement of the $L^2$ automorphic forms with vanishing Whittaker functions.
We denote this space by $L^2_{\cusp,\psi_N}(G(F)\bs G(\A)^1)$.
This space is multiplicity free and it is a meaningful problem
to study $\mainconst{\pi}$ for the irreducible constituents of the $\psi_N$-generic spectrum.

Another, more speculative way is to admit Arthur's conjectures (for the discrete spectrum) in a strong form, namely
a canonical decomposition
\[
L_{\disc}^2(G(F)\bs G(\A)^1)=\mathop{\widehat\oplus}\limits_{\phi}\overline{\canArthur_\phi}
\]
according to elliptic Arthur's parameters.
This approach was taken by Sakellaridis--Venkatesh in \cite{1203.0039}.
A good indication for its validity is a recent result of V. Lafforgue who established, in the function field case,
(say in the split case) a canonical decomposition of $C_c^{\cusp}(G(F)\bs G(\A)/K\Xi,\overline{\Q_l})$ according to Langlands's parameters,
for suitable compact open subgroups $K\subset G(\A)$ and a central lattice $\Xi$ \cite{1209.5352}.

The difficulty with this approach in the number field case is that not only are Arthur's conjectures wide open,
it is not even clear how to uniquely characterize the spaces $\canArthur_\phi$
since they cannot be pinned down purely representation theoretically (at least by standard Hecke operators).\footnote{In \cite{1209.5352}
Lafforgue introduces (in the function field case) additional symmetries of geometric nature on the space $G(F)\bs G(\A)/K\Xi$}
Nevertheless, it turns out to be profitable to admit the existence of the spaces $\canArthur_\phi$, hypothetical as they may be.

For the group $G=\GL_n$ the spaces $\canArthur_\phi$ are always irreducible.
Moreover, $\canArthur_\phi$ is cuspidal if and only if it is generic and this happens if and only if $\phi$ is of Ramanujan type
(i.e., it has trivial $\SL_2$-type).
For other groups, the reducibility of $\canArthur_\phi$ is measured to a large extent by a certain finite group $\cent_\phi$
(and its local counterparts) attached to $\phi$ \cite{MR1021499}. This of course goes back to Labesse--Langlands (\cite{MR540902}, cf.~\cite{MR757954}).
For instance, if $G$ is split then the group $\cent_\phi$ is the quotient of the centralizer of the image of $\phi$
in the complex dual $\widehat G$ of $G$ by the center of $\widehat G$.
In particular, for $G=\GL_n$ we always have $\cent_\phi=1$.
In general we expect that if $\phi$ is of Ramanujan type then
$\canArthur_\phi\cap L^2_{\cusp,\psi_N}(G(F)\bs G(\A)^1)$ is irreducible. We denote the representation on this (hypothetical) space by $\pi^{\psi_N}(\phi)$.
If $\phi$ is not of Ramanujan type (as happens for instance if $\canArthur_\phi$ is not contained in the cuspidal spectrum)
then $\whit^{\psi_N}$ vanishes on $\canArthur_\phi$ -- see \cite{MR2784745}.
One is lead to make the following:

\begin{conjecture} \label{conj: intro}
For any elliptic Arthur's parameter $\phi$ of Ramanujan type we have $\mainconst{\pi^{\psi_N}(\phi)}=\abs{\cent_\phi}$.
\end{conjecture}

The conjecture is inspired by recent conjectures and results of Ichino--Ikeda \cite{MR2585578}
which sharpen the Gross--Prasad conjecture. (See \cite{Gan_symplecticlocal} for a recent extension
of these conjectures by Gan--Gross--Prasad.)
In fact, this goes back to early results of Waldspurger (\cite{MR783511, MR646366} -- see below).
More recently, Sakellaridis--Venkatesh formulated conjectures in the much broader scope of periods
over spherical subgroups (at least in the split case) \cite{1203.0039}.
Conjecture \ref{conj: intro} can be viewed as a strengthening of the conjectures of \cite{1203.0039} in the case at hand.

In \S\ref{sec: conjWhit} we will reduce this conjecture, under some natural compatibility assumptions on Arthur's conjecture,
to the case where $G$ is semisimple and simply connected.



For quasi-split classical groups one may formulate Conjecture \ref{conj: intro} more concretely thanks to the work of
Cogdell--Kim--Piatetski-Shapiro--Shahidi, Ginzburg--Rallis--Soudry and others.
To that end let us recall the descent method of Ginzburg--Rallis--Soudry \cite{MR2848523}.
Let $G$ be a quasi-split classical group and $\psi_N$ as before.
Starting from a set $\{\pi_1,\dots,\pi_k\}$ of (distinct) cuspidal representations of general linear groups $\GL_{n_i}$
of certain self-duality type depending on $G$
and with $n_1+\dots+n_k=m$ where $m$ is determined by $G$, one constructs
a $\psi_N$-generic cuspidal representation $\sigma=\sigma^{\psi_N}(\{\pi_1,\dots,\pi_k\})$ of $G(\A)$.
Moreover, $\sigma$ is multiplicity free and any irreducible constituent of $\sigma$ has
the isobaric sum $\pi_1\boxplus\dots\boxplus\pi_k$ as its functorial transfer to $\GL_m$ under the natural homomorphism of $L$-groups.
In fact, combined with the results of Cogdell--Piatetski-Shapiro--Shahidi \cite{MR2767514} which unify and extend earlier work of
Cogdell--Kim--Piatetski-Shapiro--Shahidi \cite{MR1863734, MR2075885} and Kim--Krishnamurthy \cite{MR2127169, MR2149370}, it is known that \emph{all}
$\psi_N$-generic cuspidal representations of (quasi-split) classical groups are covered by descent.
In particular, one can describe $L(1,\sigma,\Ad)$ in terms of known $L$-functions of $\GL_n$.
The representation $\sigma$ is known to be irreducible for $\Orth(2n+1)$ (or equivalently, for $\SO(2n+1)$).
It is expected to be irreducible in all cases except if one views $\SO(2n)$ (rather than $\Orth(2n)$) as a classical group
(as we're forced to if we want to stick to connected groups):
in that case $\sigma$ may decompose as $\tau\oplus\theta(\tau)$ where $\tau$ is irreducible and $\theta$
is an outer involution preserving $N$ and $\psi_N$.



Conjecture \ref{conj: intro} translates into the following:

\begin{conjecture} \label{conj: globalclassical}
Let $\pi$ be an irreducible constituent of $\sigma^{\psi_N}(\{\pi_1,\dots,\pi_k\})$. Then
\[
\mainconst{\pi}=\begin{cases}2^{k-2}&\text{if $G=\SO(2n)$ and }\theta(\pi)=\pi,\\2^{k-1}&\text{otherwise.}\end{cases}
\]
\end{conjecture}

We remark that following Arthur's work \cite{Artendclass} and its follow-up by Mok \cite{1206.0882}, one expects to have multiplicity one for all classical groups
(again, with the caveat that if $\SO(2n)$ is admitted as a classical group then multiplicity could be two).
Thus, in all cases except for $\SO(2n)$ we could have formulated the conjecture for any $\psi_N$-generic representation whose functorial transfer to $\GL_m$ is
$\pi_1\boxplus\dots\boxplus\pi_k$.
At any rate, Arthur's work is not a prerequisite for the formulation of Conjecture \ref{conj: globalclassical}.
In fact, one can easily modify Conjecture \ref{conj: globalclassical} for $\Gspin$ groups using the work
of Asgari-Shahidi \cite{MR2219256, 1101.3467} and Hundley-Sayag \cite{MR2505178, 1110.6788}.

We can also formulate an analogous conjecture for the metaplectic groups $\Mp_n$ -- the two-fold cover of the symplectic groups $\Sp_n$.
(One expects multiplicity one to hold in this case as well.)
While these groups are not algebraic, they behave in many respects like algebraic groups.
In particular, the descent method applies to them (and gives rise to irreducible representations).
For the metaplectic group, Conjecture \ref{conj: globalclassical} and the relation \eqref{eq: gllc} have to be modified as follows:
\begin{conjecture} \label{conj: metplectic global}
Assume that $\tilde\pi$ is the $\psi_N$-descent of $\{\pi_1,\dots,\pi_k\}$ to $\Mp_n$.
Let $\pi$ be the isobaric sum $\pi_1\boxplus\dots\boxplus\pi_k$. Then
\[
\whit^{\psi_N}(\varphi)\whit^{\psi_N^{-1}}(\varphi^\vee)=
(2^k\vol(\Sp_n(F)\bs\Sp_n(\A)))^{-1}\Delta_{\Sp_n}^S(1)\frac{L^S(\frac12,\pi)}{L^S(1,\pi,\sym^2)}\prod_{v\in S}I_v(\varphi_v,\varphi_v^\vee).
\]
\end{conjecture}
We note that in the case of $\Mp_n$, the image of the $\psi_N$-descent consists of the cuspidal $\psi_N$-generic spectrum whose
$\psi$-theta lift to $\SO(2n-1)$ vanishes where $\psi$ is determined by $\psi_N$. (See \cite[\S11]{MR2848523} for more details.)
In the case $n=1$, this excludes the so-called exceptional representations.

The case of the metaplectic two-fold cover of $\SL_2$ (i.e., $n=1$) goes back to the classical result of Waldspurger on the Fourier coefficients
of half-integral weight modular forms \cite{MR646366} which was later generalized by many authors
\cite{MR894322, MR1244668, MR629468, MR1233447, MR783554, MR1404335, MR2059949, MR2669637, 1308.2353}.
Waldspurger used the Shimura correspondence as his main tool.
While it is conceivable that in general, the theta correspondence will reduce the conjecture for the metaplectic group
to the case of $\SO(2n+1)$, this will not suffice by itself to prove the conjecture.

A different approach, which was taken by Jacquet \cite{MR983610} and completed by Baruch--Mao (for $n=1$) \cite{MR2322488} is via the relative trace formula.
An important step in generalizing this approach to higher rank was taken by Mao--Rallis \cite{MR2656089}.
We mention in passing an exciting new result by Wei Zhang \cite{WeiZhang} on the Gan--Gross--Prasad conjecture for unitary groups,
which uses an analogous relative trace formula conceived by Jacquet--Rallis \cite{MR2767518}.


Let us describe the contents of the paper.

In \S\ref{sec: FCMC} we define the local integrals \eqref{eq: Fourier MC} in the $p$-adic case
as stable integrals, in the sense that the integral over a sufficiently large compact open subgroup $U$ of $N$
is independent of $U$. This is closely related to the situation in \cite{MR581582}.
The integrals \eqref{eq: Fourier MC} are compatible with the Jacquet integral and parabolic induction.
In particular, we can compute them in the unramified case using the Casselman--Shalika formula.
In the Archimedean case we give an ad hoc definition for \eqref{eq: Fourier MC}, using the results of \cite[Ch. 15]{MR1170566}.

This is used in \S\ref{sec: conjWhit} to introduce $\mainconst{\pi}$ -- the main object of interest of this paper.
We formulate Conjecture \ref{conj: intro} and show that it is compatible with restriction to subgroups containing the derived group
as well as with projection by a central induced torus.
Next we show that $\mainconst{\pi}=1$ in the case of the general linear group using Rankin--Selberg integrals (\S\ref{sec: GLn}).
Consequently, Conjecture \ref{conj: intro} holds for both the general and the special linear group.
We consider classical groups and the metaplectic group more closely in \S\ref{sec: classical groups} where we formulate Conjectures \ref{conj: globalclassical} and \ref{conj: metplectic global}.
Finally, in \S\ref{sec: examples} we explicate and prove certain low rank cases of Conjecture \ref{conj: globalclassical}.
These cases boil down to a relation (for small $n$) between a certain group of self-twists of a representation of $\pi$ of $\GL_n(\A)$ and the isobaric decomposition
of the functorial transfer under a certain representation of $\GL_n(\C)$.

In a sequel to this paper, we will reduce conjectures \ref{conj: globalclassical} and \ref{conj: metplectic global}
to a local conjectural identity which will be proved for the metaplectic group in the $p$-adic case.

\subsection{Acknowledgement}
It is a pleasure to acknowledge the contribution of several mathematicians to this paper.
First, we thank Joseph Bernstein and Herv\'e Jacquet for discussions leading to
Propositions \ref{prop: Bernstein} and \ref{prop: Jacquetprop} respectively.
Conversations with Yiannis Sakellaridis and Akshay Venkatesh were very beneficial towards the formulation
of Conjecture \ref{conj: intro}.
We thank David Soudry for explaining to us many fine points about the descent method on numerous occasions.
We thank Jean-Pierre Labesse for correspondence which lead to Appendix \ref{sec: appendix}.
We thank Michail Borovoi, Philippe Gille and Diana Shelstad for their help with Galois cohomology,
Vincent Lafforgue for expounding on the material of \cite{1209.5352},
Kaoru Hiraga for clarifying some aspects of the book \cite{MR2918491}
and Freydoon Shahidi for answering a question about \cite{MR2784745}.
We would also like to thank Wee Teck Gan, Atsushi Ichino and Omer Offen for helpful discussions and suggestions.
We thank the anonymous referee for his careful reading of an earlier version of the paper and for making constructive comments.
We thank Joachim Schwermer and the Erwin Schr\"odinger Institute in Vienna for their support and for providing excellent working conditions
for collaboration.

We cannot close this introduction without recalling our dear mentor Steve Rallis.
Both authors were privileged to benefit from his guidance as post-docs at the Ohio State University.
His advice and passion for mathematics made a deep impact on us.
The subject matter of this paper was close to Steve's heart.
We hope that it is suitable to dedicate this work to his memory.


\section{Fourier coefficients of matrix coefficients} \label{sec: FCMC}

Let $\bf G$ be a quasi-split reductive group over a local field $F$ of characteristic $0$.
If $\bf X$ is a smooth variety over $F$ and $S$ is an $F$-algebra, we use $X(S)$ to denote the $S$-points of $\bf X$,
or simply $X$ to denote its $F$-points.
We write $C^\infty(X)$ for the space of smooth functions on $X$.
(In the $p$-adic case, this means the locally constant functions on $X$.)
We also write $C_c^\infty(X)$ for the space of compactly supported smooth functions on $X$.


Let $\bf A$ be a maximal $F$-split torus of $\bf G$, ${\bf T}=C_{\bf G}({\bf A})$ (a maximal torus of $\bf G$, since $\bf G$ is quasi-split) and
$\bf B=\bf T\ltimes \bf N$ a Borel subgroup containing $\bf A$ (defined over $F$).
We denote by $\Phi$ the set of roots of $\bf A$, and by $\Phi_+$ (resp.~$\Delta_0$) the subset of
positive indivisible (resp.~simple) roots with respect to $\bf B$.
For any $\alpha\in\Phi_+$ let $\bf N_\alpha$ be the subgroup of $\bf N$ whose Lie algebra is the direct sum of the
weight spaces corresponding to roots of $\bf T$ (over the algebraic closure of $F$) whose restriction to $\bf A$
is a multiple of $\alpha$.
Let $W=\operatorname{Norm}_{G}(T)/T$ be the Weyl group of $G$ and $w_0$ the longest element of $W$.

We fix a non-degenerate (continuous) character $\psi_N:N\rightarrow\C^*$, that is $\psi_N\big|_{N_\alpha}\not\equiv1$
for every $\alpha\in\Delta_0$.
For any subgroup $N'$ of $N$ we denote the restriction of $\psi_N$ to $N'$ by $\psi_{N'}$.

By a representation of $G$, we will always mean a smooth representation $(\pi,V)$ in the $p$-adic case
(with the discrete topology on $V$)
and a smooth Fr\'echet representation $(\pi,V)$ of moderate growth in the archimedean case.
If $(\pi,V)$ is a representation of finite length, we write $(\pi^\vee,V^\vee)$ for the contragredient representation.
Let $(\cdot,\cdot)=(\cdot,\cdot)_\pi$ be the canonical pairing on $V\times V^\vee$.
For any pair $v\in V$, $v^\vee\in V^\vee$ we define the matrix coefficient $\mc_{v,v^\vee}(g)=(\pi(g)v,v^\vee)_{\pi}$.

\label{sec: irrnot}
We denote by $\Irr G$ the set of equivalence classes of irreducible representations of $G$.
Recall that $\pi\in\Irr G$ is called \emph{square-integrable} if its cental character $\omega_\pi$ is unitary and any matrix coefficient
lies in $L^2(Z\bs G,\omega_\pi)$ where $\bf Z$ is the center of $\bf G$.
(The notation stands for the space of functions on $G$ which are $(Z,\omega_\pi)$-equivariant and which are square-integrable
modulo $Z$.)
We say that $\pi\in\Irr G$ is essentially square-integrable if some twist of $\pi$ by a (not necessarily unitary) character of $G$ is square-integrable.
We denote by $\Irr_{\sqr}G$ the class of essentially square-integrable irreducible representations.
In the $p$-adic case we will also write $\Irr_{\cusp}G$ for the set of supercuspidal representations in $\Irr G$.
If $\pi\in\Irr_{\sqr}G$ then any matrix coefficient of $\pi$ belongs to the Harish-Chandra Schwartz space of the derived group
of $G$ (\cite[Corollaire III.1.2]{MR1989693} -- $p$-adic case, \cite[Theorem 15.2.4]{MR1170566} -- archimedean case),
and in particular, it is integrable over $N$
(\cite[Proposition II.4.5]{MR1989693}, \cite[Thereom 7.2.1]{MR929683}).
This is not true for a general $\pi\in\Irr G$. The goal of this section is to make sense of the integral
\[
\int_N\mc_{v,v^\vee}(n)\psi_N(n)^{-1}\ dn
\]
when it does not converge absolutely.

Let $\JF_{\psi_N}(\pi)=\JF_{\psi_N}^G(\pi)$ be the twisted Jacquet module of $\pi$, namely, the quotient of $\pi$
by the closure of the span of $\pi(n)v-\psi_N(n)v$, $u\in N$, $v\in V_\pi$.
In the $p$-adic case $\pi\mapsto\JF_{\psi_N}(\pi)$ is an exact functor.
We say that $\pi$ is $\psi_N$-generic if $\JF_{\psi_N}(\pi)$ is nontrivial,
in which case it is one dimensional (if $\pi\in\Irr G$).
We denote by $\Irr_{\gen,\psi_N}G$ the set of equivalence classes of irreducible representations
which are $\psi_N$-generic.

\subsection{}
We start with the $p$-adic case.
Until further notice $F$ will be a $p$-adic field with ring of integers $\OO$.
For any group $\bf H$ over $F$ denote by $\csgr(H)$ the set of compact open subgroups of $H$.
Suppose that $U$ is a unipotent group over $F$ with a fixed Haar measure $du$.
Recall that the group generated by a relatively compact subset of $U$ is relatively compact.
In particular, the set $\csgr(U)$ is directed.

\label{sec: stint}
\begin{definition} \label{def: stable integral}
Let $f$ be a smooth function on $U$.
We say that $f$ has a \emph{stable integral} over $U$ if there exists $U_1\in\csgr(U)$ such that for any
$U_2\in\csgr(U)$ containing $U_1$ we have
\begin{equation} \label{eq: comval}
\int_{U_2}f(u)\ du=\int_{U_1}f(u)\ du.
\end{equation}
In this case we write $\stint_U f(u)\ du$ for the common value \eqref{eq: comval} and say that
$\stint_U f(u)\ du$ stabilizes at $U_1$.
In other words, $\stint_U f(u)\ du$ is the limit of the net $(\int_{U_1}f(u)\ du)_{U_1\in\csgr(U)}$
with respect to the discrete topology of $\C$.
\end{definition}

\begin{remark} \label{rem: basicstable}
\begin{enumerate}
\item Clearly, if $f\in C_c^\infty(U)$ then $f$ has a stable integral.
\item More generally, let $R$ (resp., $L$) be the right (resp., left) regular representation of $U$ on $C^\infty(U)$.
We extend $R$ and $L$ to representations of the algebra of finite Borel measures of $U$ with compact support.
Suppose that there exist $U_1,U_2\in\csgr(U)$ such that $R(e_{U_1})L(e_{U_2})f\in C_c^\infty(U)$ where $e_{U_i}$
is the Haar measure on $U_i$ with volume $1$, $i=1,2$.
Then $f\in C^\infty(U)$ has a stable integral over $U$ and
\[
\stint_U f(u)\ du=\int_U[R(e_{U_1})L(e_{U_2})f](u)\ du.
\]
In this case, we will say that $f$ is \emph{compactly supported after averaging}.
\item It is not true in general that if $f\in L^1(U)\cap C^\infty(U)$ then $f$ has a stable integral.
However, if it does, then the stable integral is equal to $\int_Uf(u)\ du$.
\item If $f$ has a stable integral over $U$, then any right or left translate of
$f$ by an element of $U$ has a stable integral (with the same value).
\item \label{part: autstable} Similarly, if $\alpha$ is an automorphism of $U$ and $\stint_Uf(u)\ du$ is defined then
$\stint_Uf(\alpha(u))\ du$ is defined and equals to $m_\alpha^{-1}\stint_Uf(u)\ du$ where $m_\alpha$
is the module of $\alpha$.
\end{enumerate}
\end{remark}

\begin{proposition} \label{prop: stablemc}
Let $(\pi,V)\in\Irr G$.
Then for any $v\in V$, $v^\vee\in V^\vee$ the function $\psi_N^{-1}\cdot\mc_{v,v^\vee}\big|_N$ is compactly supported after averaging and hence
has a stable integral over $N$.
Moreover, if $K_0\in\csgr(G)$ and $v\in V^{K_0}$, $v^\vee\in (V^\vee)^{K_0}$ then
\[
(v,v^\vee)_\pi^{\psi_N}:=\stint_{N(F)}(\pi(n)v,v^\vee)_\pi\psi_N(n)^{-1}\ dn
\]
stabilizes at $U_1\in\csgr(N)$ depending only on $K_0$.
The bilinear form $(v,v^\vee)_\pi^{\psi_N}$
is $(N,\psi_N)$-equivariant in $v$ and $(N,\psi_N^{-1})$-equivariant in $v^\vee$.
Thus, $(v,v^\vee)_\pi^{\psi_N}\equiv0$ unless $\pi\in\Irr_{\gen,\psi_N}G$, in which case
$(\cdot,\cdot)_\pi^{\psi_N}$ descends to a non-degenerate pairing (denoted the same way)
between the one-dimensional spaces $\JF_{\psi_N}(\pi)$ and $\JF_{\psi_N^{-1}}(\pi^\vee)$.
\end{proposition}

We will prove the proposition below.
Note that if $\pi\in\Irr_{\sqr}$ then
\[
(v,v^\vee)_\pi^{\psi_N}=\int_{N(F)}(\pi(n)v,v^\vee)\psi_N(n)^{-1}\ dn.
\]

\begin{remark} \label{rem: anydual}
Suppose that we are given $\pi, \hat\pi\in\Irr G$ with a non-degenerate pairing $(\cdot,\cdot)$ between them.
Then, by identifying $\hat\pi$ with $\pi^\vee$ we can make sense of $(\cdot,\cdot)^{\psi_N}$ in this context.
\end{remark}

\begin{remark}
For a different approach to define $(\cdot,\cdot)^{\psi_N}$ (at least in the tempered case), see \cite{1203.0039}.
\end{remark}

\subsection{}
In order to prove Proposition \ref{prop: stablemc} we will first need an auxiliary result which is based on \cite{MR581582}.
Let $P=M\ltimes U$ be a standard parabolic subgroup of $G$ with its standard Levi decomposition.
Let $P'=M'\ltimes U'$ be the standard parabolic subgroup of $G$ which is conjugate to the parabolic subgroup opposite to $P$.
Denote by $W^M$ the Weyl group of $M$ and by $w_0^M$ the longest element in $W^M$.
We identify $W^M\bs W$ with the set of left $W^M$-reduced elements of $W$.
Denote by $w_M=w_0^Mw_0$ the longest element of $W^M\bs W$, so that $w_M^{-1}Mw_M=M'$.
If $\sigma$ is a representation of $M$ then we write $\Ind_P^G\sigma=\Ind\sigma$ for the (normalized) parabolic induction.

Recall the Bruhat decomposition
\[
G=\cup_{w\in W^M\bs W}Pw N.
\]
Also recall the Bruhat order on $W^M\bs W$ defined by
$w_1\le w_2$ whenever $Pw_1N$ is contained in the closure of $Pw_2N$.

We denote by $(\Ind_P^G\sigma)^\circ$ the $P'$-invariant space of sections in $\varphi$ which are
supported on the big cell $Pw_MP'=Pw_MN=Pw_MU'$.
Note that for any $\varphi\in (\Ind_P^G\sigma)^\circ$ the function $\varphi(w_M\cdot)$ on $U'$
is compactly supported.

\begin{lemma} \label{lem: suppbc}
Suppose that $\sigma$ is a representation of $M$ and $\pi=\Ind_P^G\sigma$.
Then for any $\varphi\in\pi$ there exists $N_1\in\csgr(N)$
such that $\varphi_{N_1,\psi_{N_1}}:=\pi(\psi_{N_1}^{-1}e_{N_1})\varphi\in(\Ind\sigma)^\circ$.
Moreover, let $K_0\in\csgr(G)$ and assume that $\varphi\in\pi^{K_0}$. Then we can choose $N_1$
above depending only on $K_0$, and the support of $\varphi_{N_1,\psi_{N_1}}$ on $w_MU'$ is bounded
in terms $K_0$ only.
\end{lemma}

\begin{proof}
This is proved exactly as in \cite[Lemma 2.2]{MR581582}.

We show by induction on $\ell(w)$ that for any $w\in W^M\bs W$ there exists
$N_1\in\csgr(N)$ such that $\varphi_{N_1,\psi_{N_1}}^{}$ vanishes on $\cup_{w'<w}Pw'N$.
For $w=w_M$ we will obtain the lemma.

The base of the induction ($w=e$) is the empty statement.
Note that if $\varphi_{N_1,\psi_{N_1}}^{}$ vanishes on $Pw'N$ then the same holds
for any $N_2\in\csgr(N)$ containing $N_1$.
Therefore, for the induction step it will be enough to show the following statement for any $w\ne w_M$.
\begin{multline*}
\text{If $\varphi\rest_{\cup_{w'<w}Pw'N}\equiv0$ then we can choose $N_1\in\csgr(N)$
such that $\varphi_{N_1,\psi_{N_1}}^{}\rest_{PwN}\equiv0$.}
\end{multline*}
Since $w\ne w_M$, there exists $\alpha\in\Delta_0$ such that $w\alpha$ is a root of $A$ in the Lie algebra of $U$.
Therefore $N_\alpha\subset N\cap w^{-1}Uw$ and $\psi_{N\cap w^{-1}Uw}\not\equiv1$.
Let $N_2\in\csgr(N)$ be sufficiently large so that $\psi_{N_2\cap w^{-1}Uw}\not\equiv1$.
Then
\begin{align*}
\varphi_{N_2,\psi_{N_2}}^{}(w)&=\int_{N_2}\varphi(wn)\psi_{N_2}(n)^{-1}\ dn\\&=
\int_{N_2\cap w^{-1}Uw\bs N_2}\int_{N_2\cap w^{-1}Uw}\varphi(wn'n)\psi_{N_2}(n')^{-1}\psi_{N_2}(n)^{-1}\ dn'\ dn\\&=
\int_{N_2\cap w^{-1}Uw\bs N_2}\varphi(wn)\psi_{N_2}(n)^{-1}\int_{N_2\cap w^{-1}Uw}\psi_{N_2}(n')^{-1}\ dn'\ dn=0.
\end{align*}
It follows that $\varphi_{N_1,\psi_{N_1}}^{}(wu)=(R(u)\varphi)_{uN_1u^{-1},\psi_{uN_1u^{-1}}}^{}(w)=0$
for any $N_1\in\csgr(N)$ and $u\in N$ such that
\begin{equation} \label{eq: vancond}
\psi_{uN_1u^{-1}\cap w^{-1}Uw}\not\equiv1.
\end{equation}
Clearly, the condition \eqref{eq: vancond} is right $N_1$-invariant in $u$.
It is also left $N_w$-invariant where $N_w=N\cap w^{-1}Nw$ since $N_w$ normalizes $w^{-1}Uw$.
By assumption, the support of $\varphi(w\cdot)$ on $N$ is compact modulo $N_w$.
Choose a compact subset $\Omega\subset N$ such that the above support is contained in $N_w\Omega$.
Choose $N_1\in\csgr(N)$ containing $\cup_{u\in\Omega}u^{-1}N_2u$.
Thus, \eqref{eq: vancond} holds for $u\in\Omega$. Hence it holds for $u\in N_w\Omega N_1$.
Thus $\varphi_{N_1,\psi_{N_1}}^{}$ vanishes on $wN_w\Omega N_1$.
On the other hand, $\varphi_{N_1,\psi_{N_1}}^{}$ vanishes on $wN\setminus wN_w\Omega N_1$
by the support condition on $\varphi$. We conclude that $\varphi_{N_1,\psi_{N_1}}^{}$
vanishes on $wN$ and hence on $PwN$, as required.

For the second statement we just need to observe that in the argument above, we can choose $\Omega$ to depend only on $K_0$.
\end{proof}

\begin{proof}[Proof of Proposition \ref{prop: stablemc}]
By Jacquet's subrepresentation theorem, it is enough to consider
the case where $\pi=\Ind_P^G\sigma$ (not necessarily irreducible) and $\sigma$ is a
supercuspidal representation of $M$ (not necessarily unitary).
We identify $\pi^\vee$ with $\Ind_P^G\sigma^\vee$ via the pairing
\begin{equation} \label{eq: indpairing}
(\varphi,\varphi^\vee)_\pi=\int_{P\bs G}(\varphi(g),\varphi^\vee(g))_\sigma\,dg.
\end{equation}
We will take the `measure' on $P\bs G$ by fixing a Haar measure on $U'$ and defining
\begin{equation} \label{eq: measPG}
\int_{P\bs G}f(g)\ dg=\int_{U'}f(w_Mu)\ du
\end{equation}
for any continuous function $f$ on $G$ satisfying
$f(pg)=\modulus_P(p)f(g)$ for any $p\in P$, $g\in G$ where $\modulus_P$ is the modulus function of $P$.
Let $\varphi\in\pi$, $\varphi^\vee\in\pi^\vee$.
By the previous lemma there exists $N_1\in\csgr(N)$ such that $\varphi_{N_1,\psi_{N_1}}\in(\Ind\sigma)^\circ$.
Similarly, there exists $N_2\in\csgr(N)$ such that $\varphi^\vee_{N_2,\psi_{N_2}^{-1}}\in(\Ind\sigma^\vee)^\circ$.
Note that
$R(\psi_{N_1}^{-1}e_{N_1})L(\psi_{N_2}e_{N_2})\mc_{\varphi,\varphi^\vee}=\mc_{\varphi_{N_1,\psi_{N_1}}^{},\varphi^\vee_{N_2,\psi_{N_2}^{-1}}}$.
Therefore, upon replacing $\varphi$ and $\varphi^\vee$ by
$\varphi_{N_1,\psi_{N_1}}$ and $\varphi^\vee_{N_2,\psi_{N_2}^{-1}}$ respectively we may assume
that $\varphi\in(\Ind\sigma)^\circ$ and $\varphi^\vee\in(\Ind\sigma^\vee)^\circ$ and we will show that
$\mc_{\varphi,\varphi^\vee}$ is compactly supported on $N$ in this case.
By \eqref{eq: indpairing} and \eqref{eq: measPG} we have
\begin{equation} \label{eq: MCu}
(\pi(u)\varphi,\varphi^\vee)_\pi=\int_{U'}(\varphi(w_Mu_1u),\varphi^\vee(w_Mu_1))_\sigma\ du_1
\end{equation}
for any $u\in N$.
By the property of $\varphi^\vee$ the integral over $u_1$
can be taken over a compact subset. Write $u_1u=u_2u_3$ where
$u_2\in N\cap M'$ and $u_3\in U'$. Then
\[
\varphi(w_Mu_1u)=\sigma(w_Mu_2w_M^{-1})\varphi(w_Mu_3).
\]
Thus, in the integral on the right-hand side of \eqref{eq: MCu}, $u_3$ is confined to a compact set. Moreover, since the matrix
coefficients of $\sigma$ are compactly supported modulo the center of $M$,
$u_2$ is confined to a compact set as well. Hence, the same is true for $u$ as claimed.

For the statement about the dependence on $K_0$ it suffices to use the corresponding statement in the previous
lemma and the fact that there are only finitely many supercuspidal representations (up to twisting
by unramified character) for a given level.

We still have to show the non-vanishing of the pairing in the generic case.
This will be done in Proposition \ref{prop: nontrivpsi_N} below.
\end{proof}

\begin{remark}
The dependence of $U_1$ on $K_0$ in Proposition \ref{prop: stablemc} is not made explicit in the proof above.
\end{remark}

\subsection{The Jacquet integral}
Next we consider the Jacquet integral.
Let $P=M\ltimes U$ be a standard parabolic subgroup of $G$. Set $N_M=N\cap M$ and $N_{M'}=N\cap M'$.
Recall that if $\pi\in\Irr_{\gen,\psi_N}G$ is a subquotient of $\Ind_P^G\sigma$ where $\sigma\in\Irr M$
then $\sigma\in\Irr_{\gen,\Mpsi{M}}M$ where $\Mpsi{M}$
is the character on $N_M$ given by $\Mpsi{M}(n)=\psi_N(w_M^{-1}nw_M)$.
For any $\varphi\in\Ind\sigma$ choose $N_1\in\csgr(N)$ such that $\varphi_{N_1,\psi_{N_1}}^{}\in(\Ind\sigma)^\circ$.
Then the integral
\[
\varphi\mapsto\int_{U'}\varphi_{N_1,\psi_{N_1}}^{}(w_Mu)\psi_{U'}(u)^{-1}\ du
\]
converges and its projection to $\JF_{\Mpsi{M}}^M(\sigma)$ does not depend on the choice of $N_1$.
Thus we get a map
\[
J_\sigma^{\psi_N}:=\JF_{\psi_N}(\Ind_P^G\sigma)\rightarrow\JF_{\Mpsi{M}}^M(\sigma)
\]
which is in fact an isomorphism of vector spaces.
Dually, from any $\Mpsi{M}$-Whittaker functional on $\sigma$ we construct
a $\psi_N$-Whittaker functional on $\Ind\sigma$. Of course, this construction coincides with the usual one
given by analytic continuation.

By abuse of notation, we often view $J_\sigma^{\psi_N}$ as a map defined on $\Ind_P^G\sigma$ through the canonical projection
$\Ind_P^G\sigma\rightarrow\JF_{\psi_N}(\Ind_P^G\sigma)$.

\begin{proposition} \label{prop: jacquetdescent}
Suppose that $\sigma\in\Irr_{\sqr,\gen,\Mpsi{M}}M$ and let $\pi=\Ind_P^G\sigma$.
Identify $\pi^\vee$ with $\Ind_P^G\sigma^\vee$ as before. Then
\begin{equation} \label{eq: redtosqrint}
(\varphi,\varphi^\vee)_\pi^{\psi_N}=
(J_\sigma^{\psi_N}(\varphi),J_{\sigma^\vee}^{\psi_N^{-1}}(\varphi^\vee))_\sigma^{\Mpsi{M}}
\end{equation}
for any $\varphi\in\pi$, $\varphi^\vee\in\pi^\vee$.
\end{proposition}

\begin{proof}
As before, by Lemma \ref{lem: suppbc} we can assume that $\varphi\in(\Ind\sigma)^\circ$ and $\varphi^\vee\in(\Ind\sigma^\vee)^\circ$.
In this case, by \eqref{eq: MCu}, the left-hand side of \eqref{eq: redtosqrint} is equal to
\[
\int_{N_{M'}}\int_{U'}\int_{U'}(\varphi(w_Mu_1u_2u_3),\varphi^\vee(w_Mu_1))_\sigma\psi_N(u_2u_3)^{-1}\ du_1\ du_2\ du_3.
\]
By a change of variable we get
\begin{align*}
&\int_{N_{M'}}\int_{U'}\int_{U'}(\varphi(w_Mu_1u_3),\varphi^\vee(w_Mu_2))_\sigma\psi_N(u_1u_2^{-1}u_3)^{-1}\ du_1\ du_2\ du_3\\=&
\int_{N_{M'}}\int_{U'}\int_{U'}(\varphi(w_Mu_3u_1),\varphi^\vee(w_Mu_2))_\sigma\psi_N(u_1u_2^{-1}u_3)^{-1}\ du_1\ du_2\ du_3\\=&
\int_{N_M}\int_{U'}\int_{U'}(\sigma(u_3)\varphi(w_Mu_1),\varphi^\vee(w_Mu_2))_\sigma\psi_N(u_1u_2^{-1}w_M^{-1}u_3w_M)^{-1}\ du_1\ du_2\ du_3.
\end{align*}
Note that the integrals over $u_1$ and $u_2$ are effectively sums over finite sets which are
independent of $u_3$ because of our assumption on $\varphi$ and $\varphi^\vee$,
and for any $u_1$ and $u_2$, the integral over $u_3$ is absolutely convergent since $\sigma\in\Irr_{\sqr}M$.
Thus the triple integral is absolutely convergent, which justifies the previous steps.
We obtain
\[
\int_{N_M}(\sigma(u_3)J_\sigma^{\psi_N}(\varphi),J_{\sigma^\vee}^{\psi_N^{-1}}(\varphi^\vee))_\sigma\Mpsi{M}(u_3)^{-1}\ du_3
\]
which is the right-hand side of \eqref{eq: redtosqrint}, as required.
\end{proof}

\begin{remark}
In fact, using induction in stages and the transitivity of the Jacquet integral,
the proposition holds for any $\sigma\in\Irr_{\gen,\Mpsi{M}}M$,
not necessarily essentially square integrable.
\end{remark}

We can now complete the remaining part of Proposition \ref{prop: stablemc}.

\begin{proposition} \label{prop: nontrivpsi_N}
Suppose that $\pi\in\Irr_{\gen,\psi_N}G$. Then the bilinear form $(\cdot,\cdot)_\pi^{\psi_N}$ is non-trivial.
\end{proposition}

\begin{proof}
If $\pi$ is supercuspidal, this follows from \cite[Lemma 1.1]{MR729755} (which is stated for $\GL_n$, but proved in general).
Alternatively, it follows from the proof of \cite[Lemma 3]{MR1816039}.
In the general case, realize $\pi$ as a quotient of $\Ind_P^G\sigma$ where $\sigma$ is supercuspidal.
By Proposition \ref{eq: redtosqrint} and the supercuspidal case, the bilinear form $(\cdot,\cdot)_{\Ind\sigma}^{\psi_N}$ is non-trivial.
On the other hand $(\cdot,\cdot)_{\Ind\sigma}^{\psi_N}$ factors through $\pi\times\Ind\sigma^\vee$ since $\pi$ is the only $\psi_N$-generic subquotient of
$\Ind\sigma$. Moreover, the embedding $\JF_{\psi_N^{-1}}(\pi^\vee)\rightarrow\JF_{\psi_N^{-1}}(\Ind\sigma^\vee)$ is an isomorphism
since $\dim\JF_{\psi_N^{-1}}(\Ind\sigma^\vee)=1$.
Hence, the restriction $(\cdot,\cdot)_{\Ind\sigma}^{\psi_N}$ to $\Ind\sigma\times\pi^\vee$ is non-zero.
The proposition follows.
\end{proof}

We can also derive the following consequence.
We thank Joseph Bernstein for this observation.
\begin{proposition} \label{prop: Bernstein}
For any left and right $G$-smooth function $f$ on $G$ the integral
\[
\stint_N f(n)\psi_N(n)^{-1}\ dn
\]
is well defined. Moreover, for any $K_0\in\csgr(G)$ there exists $N_1\in\csgr(N)$ such that
\[
\stint_N f(n)\psi_N(n)^{-1}\ dn=\int_{N_1}f(n)\psi_N(n)^{-1}\ dn
\]
for any bi-$K_0$-invariant function $f$ on $G$.
\end{proposition}

\begin{proof}
We first prove it for $f\in C_c^\infty(G)$. In this case we use Plancherel inversion to write
\[
f(x)=\int_{\Irr_{\temp}G}\tr(\pi(x)\pi(f))\ d\mu_{\operatorname{pl}}(\pi)
\]
where $d\mu_{\operatorname{pl}}$ is the Plancherel measure on $\Irr_{\temp}G$ the set of tempered representations of $G$.
If $f$ is bi-$K_0$-invariant then only those $\pi$ such that $V_\pi^{K_0}\ne0$ contribute, namely only a finite number
of compact tori.
Integrating over a big compact open subgroup of $N$ and interchanging the integral
we see that the integral stabilizes depending only on $K_0$ by the uniformity part of Proposition
\ref{prop: stablemc},
since the trace is a finite sum of matrix coefficients corresponding to vectors in $V_\pi^{K_0}$.
Of course, this also shows that there exists $N_1\in\csgr(N)$ depending only on $K_0$ such that
\[
\int_{N_2}1_{K_0\gamma K_0}(n)\psi_N(n)^{-1}\ dn=0
\]
for all $N_2\supset N_1$ and $\gamma$ outside the compact set $K_0N_1K_0$.
The proposition follows.
\end{proof}

We also have the following closely related consequence which was kindly explained to us by Jacquet.
\begin{proposition} \label{prop: Jacquetprop}
Given $\varphi\in\ind_N^G\psi_N$, $\varphi^\vee\in\ind_N^G\psi_N^{-1}$, the matrix coefficient
\[
(\pi(g)\varphi,\varphi^\vee)_{N\bs G}
\]
is compactly supported in $g\in G$.\footnote{Here $\ind$ denotes compact induction}
\end{proposition}

\begin{proof}
For any $f\in C_c^\infty(G)$ we have
\begin{multline*}
\int_G(\pi(g)\varphi,\varphi^\vee)_{N\bs G}f(g)\ dg=\int_G\int_{N\bs G}\varphi(xg)\varphi^\vee(x)f(g)\ dx\ dg
\\=\int_G\int_{N\bs G}\varphi(g)\varphi^\vee(x)f(x^{-1}g)\ dx\ dg=
\int_{N\bs G}\int_{N\bs G}\varphi(g_1)\varphi^\vee(g_2)\int_Nf(g_2^{-1}ng_1)\psi_N(n)\ dn.
\end{multline*}
Suppose that $\varphi$ and $\varphi^\vee$ are right-$K_0$-invariant for some $K_0\in\csgr(G)$ and take $f$ to be bi-$K_0$-invariant.
Then by the previous proposition
\[
\varphi(g_1)\varphi^\vee(g_2)\int_Nf(g_2^{-1}ng_1)\psi_N(n)\ dn=0
\]
if $f$ is supported outside a compact set depending only on $K_0$ and the support of $\varphi$, $\varphi^\vee$.
The proposition follows.
\end{proof}

\begin{remark} \label{rem: archJacquet}
It will be interesting to know whether the analogue of Proposition \ref{prop: Jacquetprop} holds in the archimedean case.
Namely, suppose that $\varphi:G\rightarrow\C$ (resp., $\varphi^\vee$) is smooth, left $(N,\psi)$ (resp., $(N,\psi^{-1})$) equivariant
and rapidly decreasing on $AK$ together with all its derivatives. Is the matrix coefficient
\[
(\pi(g)\varphi,\varphi^\vee)_{N\bs G}
\]
necessarily a Schwartz function on $G$?
\end{remark}

\subsection{Unramified computation}
Next we consider the unramified case.
Suppose that $\bf G$ splits over an unramified extension of $F$.
Let $K$ be a hyperspecial maximal compact subgroup of $G$ as in \cite[Corollary 4.6]{MR1474159}.
Let $\Delta_G(s)$ be the $L$-factor of the dual $M^\vee$ to the motive $M$ introduced by Gross in [ibid.]
(cf.~\cite[p.~79]{MR0230728}). It depends only on the $F$-isogeny class of $\bf G$.
For instance, $\Delta_T(s)$ is the $L$-factor of the Artin representation of $\Gal(\bar F/F)$ on $X^*(\bf T)\otimes\C$
where $X^*(\bf T)$ is the lattice of algebraic characters of $\bf T$ (defined over the algebraic closure $\bar F$ of $F$).
On the other hand, if $\bf G$ is split over $F$ of rank $r$ then
\[
\Delta_G(s)=\prod_{i=1}^r\zeta_F(s+d_i-1)
\]
where $d_1,\dots,d_r$ are exponents of $\bf G$ and $\zeta_F(s)=(1-q_F^{-s})^{-1}$ is Tate's local factor corresponding to the trivial
character of $F^*$ (where of course $q_F$ is the cardinality of the residue field of $F$).


We denote by $\Irr_{\unr}G$ the set of unramified irreducible representations of $G$.


\begin{proposition} \label{prop: unram}
Suppose that $\pi\in\Irr_{\unr} G$ and $\psi_N$ is unramified in the sense of \cite[\S3]{MR581582}.
Let $v_0$ and $v_0^\vee$ be unramified vectors in $\pi$ and $\pi^\vee$ respectively.
Then
\begin{equation} \label{eq: unramcomp}
(v_0,v_0^\vee)_\pi^{\psi_N}=\vol(N\cap K)\frac{(v_0,v_0^\vee)_\pi\Delta_G(1)}{L(1,\pi,\Ad)}.
\end{equation}
\end{proposition}

\begin{proof}
Since the space of unramified vectors is one dimensional, it is enough to check the claim
for a specific choice of non-zero vectors $v_0$, $v_0^\vee$.
Write $\pi$ as a subrepresentation of $\Ind_B^G\chi$ where $\chi$ is an unramified character of $T$.
(In fact, $\pi$ is a direct summand of $\Ind_B^G\chi$ but we will not use this fact.)
Let $\varphi_0$ (resp.~$\varphi_0^\vee$) be the unramified vector in $\Ind_B^G\chi$
(resp.~$\Ind_B^G\chi^{-1}$) such that $\varphi_0(e)=\varphi_0^\vee(e)=1$.
Note that the validity of \eqref{eq: unramcomp} is independent of any choice of Haar measure.
We endow $G$ and $T$ with the `canonical' Haar measures described in the discussion preceding \cite[Proposition 4.7]{MR1474159}.
We endow $N$ with the Haar measure such that $\vol(N\cap K)=1$. This gives rise to a Haar measure on $B=T\ltimes N$,
which is compatible with the relation \eqref{eq: measPG}.
It follows from \cite[Proposition 4.7]{MR1474159} (cf.~\cite{MR0213362, MR581580}) that
\[
(\varphi_0,\varphi_0^\vee)_{\Ind\chi}=\frac{\vol_G(K)}{\vol_B(K\cap B)}=\frac{\vol_G(K)}{\vol_T(K\cap T)}=\frac{\Delta_T(1)}{\Delta_G(1)}.
\]
Therefore we need to show that (for the above choice of measures)
\begin{equation} \label{eq: fromCS}
(\varphi_0,\varphi_0^\vee)_{\Ind\chi}^{\psi_N}=\frac{\Delta_T(1)}{L(1,\pi,\Ad)}.
\end{equation}
However, by \eqref{eq: redtosqrint} we have
\[
(\varphi_0,\varphi_0^\vee)_{\Ind\chi}^{\psi_N}=J_\chi^{\psi_N}(\varphi_0)J_{\chi^{-1}}^{\psi_N^{-1}}(\varphi_0^\vee)
\]
and \eqref{eq: fromCS} follows from the Casselman--Shalika formula \cite{MR581582}.
To explain this, we introduce some more notation. Let $\widehat G$ be the complex dual group of $\bf G$ with Borel subgroup
$\widehat B=\widehat T\ltimes\widehat N$. The Galois group $\Gal(\bar F/F)$ acts on $\widehat G$ and preserves
$\widehat B$ and $\widehat T$. Let $X_*(\widehat T)$ be the lattice of algebraic co-characters of $\widehat T$.
Let $\lambda\in X^*({\bf T})\otimes\C=X_*(\widehat T)\otimes\C$ be any element whose image under the canonical map
$X^*({\bf T})\otimes\C\rightarrow\Hom(T,\C^*)$ is $\chi$ and let $\hat t_\chi\in\widehat T$ be the image of $\lambda$ under the canonical map
$X_*(\widehat T)\otimes\C\rightarrow\widehat T(\C)$.
Let $\hat{\mathfrak{n}}$ be the Lie algebra of $\widehat N$ and for each $\alpha\in\Phi_+$
let $\hat{\mathfrak{n}}_\alpha$ be the direct sum of the weight spaces of $\widehat T$ in $\hat{\mathfrak{n}}$ corresponding to the roots of $\bf T$
(over the algebraic closure of $F$) whose restriction to $\bf A$ is a multiple of $\alpha$. Thus, $\hat{\mathfrak{n}}=\dsum_{\alpha\in\Phi_+}
\hat{\mathfrak{n}}_\alpha$. Suppose that $\bf G$ splits over an unramified extension $E/F$ and let $\sigma$ be the Frobenius
element of $\Gal(E/F)$. By the Casselman--Shalika formula\footnote{This is written a little differently (but equivalently) in \cite{MR581582}.
To compare the two expressions -- cf.~\cite[\S3]{MR581580}} we have
\[
J_\chi^{\psi_N}(\varphi_0)=\prod_{\alpha\in\Phi_+}
\det(1-q_F^{-1}\sigma\Ad\hat t_\chi\big|_{\hat{\mathfrak{n}}_\alpha}).
\]
Therefore (cf.~\cite[\S4]{MR581580}),
\[
(\varphi_0,\varphi_0^\vee)_{\Ind\chi}^{\psi_N}=\det(1-q_F^{-1}\sigma\Ad\hat t_\chi\big|_{\hat{\mathfrak{n}}})
\det(1-q_F^{-1}\sigma\Ad\hat t_\chi^{-1}|_{\hat{\mathfrak{n}}})=
\frac{\Delta_T(1)}{L(1,\pi,\Ad)}.
\]
The proposition follows.
\end{proof}

\subsection{Archimedean case} \label{sec: arch}
Now consider the archimedean case.
Let $\pi\in\Irr_{\gen,\psi_N}G$.
We would like to define the pairing $(\cdot,\cdot)_\pi^{\psi_N}$.
Suppose first that $\pi$ is tempered.
We write $\pi$ as a direct summand of $\Ind_P^G\sigma$ where
$P=M\ltimes U$ is a standard parabolic subgroup and $\sigma\in\Irr_{\sqr}M$.
We have $\sigma\in\Irr_{\gen,\Mpsi{M}}M$.
Identify $\pi^\vee$ with a direct summand of $\Ind_P^G\sigma^\vee$ as before.
The Jacquet integral $J^{\psi_N}_\sigma$ still makes sense in this context.
Namely, there is a unique isomorphism of vector spaces
\[
\JF_{\psi_N}(\Ind_P^G\sigma)\rightarrow\JF_{\Mpsi{M}}^M(\sigma)
\]
such that the resulting map
\[
J_\sigma^{\psi_N}:\Ind_P^G\sigma\rightarrow\JF_{\Mpsi{M}}^M(\sigma)
\]
extends the map
\[
\varphi\mapsto\int_{U'}\varphi(w_Mu)\psi_{U'}(u)^{-1}\ du, \ \ \varphi\in(\Ind\sigma)^\circ
\]
(or rather its composition with the natural projection $\sigma\rightarrow\JF_{\Mpsi{M}}^M(\sigma)$)
where $(\Ind\sigma)^\circ$ is defined as in the $p$-adic case \cite[Theorem 15.4.1]{MR1170566}.
Thus we will \emph{define} the pairing $(v,v^\vee)_\pi^{\psi_N}$ to be right-hand side of \eqref{eq: redtosqrint}.
Once again, this will descend to a pairing between
$\JF_{\psi_N}(\pi)$ and $\JF_{\psi_N^{-1}}(\pi^\vee)$.

The pairing $(v,v^\vee)_\pi^{\psi_N}$ does not depend on the choice of $\sigma$.
Indeed, suppose that $\pi$ is a direct summand of $\Ind_{P'}^G\sigma'$ where $\sigma'\in\Irr_{\sqr}M'$
for standard parabolic subgroup $P'=M'\ltimes U'$ of $G$.
Then there exists $w\in W$ such that $w$ is right $W_M$-reduced, $wMw^{-1}=M'$ and $\sigma'=w\sigma$.
We identify $\JF_{\Mpsi{M}}^M(\sigma)$ with $\JF_{\Mpsi{M'}}^{M'}(\sigma')$ through $w$.
We can use the results of Shahidi \cite{MR1070599} to define normalized intertwining operators
\[
R(\sigma,w):\Ind_P\sigma\rightarrow\Ind_{P'}\sigma',
\]
and similarly for $\sigma^\vee$, 
such that
\[
J_{\sigma'}^{\psi_N}\circ R(\sigma,w)=J_\sigma^{\psi_N},\,\,
J_{\sigma'^\vee}^{\psi_N^{-1}}\circ R(\sigma^\vee,w)=J_{\sigma^\vee}^{\psi_N^{-1}}
\]
and $(R(\sigma,w)\varphi,R(\sigma^\vee,w)\varphi^\vee)=(\varphi,\varphi^\vee)$
for any $\varphi\in\pi$, $\varphi^\vee\in\pi^\vee$.
More precisely, $R(\sigma,w)$ is given by normalizing the standard intertwining operator by the local coefficients
of \cite{MR1070599}.

Alternatively, following \cite{1203.0039}, we could define $(v,v^\vee)_\pi^{\psi_N}$ (in the tempered case) as follows.
(See \cite{LMao4}.) Let $N^\circ$ be the derived group of $N$.
The integral $\int_{N^\circ}(\pi(n\cdot)v,v^\vee)\ dn$ converges and defines an $L^2$ function on $N^\circ\bs N$.
Its Fourier transform is regular on the open set of non-degenerate characters of $N$.
Its value at $\psi_N^{-1}$ is by definition $(v,v^\vee)_\pi^{\psi_N}$.

In the general case, we use the Langlands classification to write
$\pi$ as a subrepresentation of $\Ind_P^G\sigma$ where $P=M\ltimes U$ is a standard parabolic and
$\sigma\in\Irr_{\gen,\Mpsi{M}}M$ is essentially tempered. (In fact, $\Ind_P^G\sigma$
is irreducible.)
As before, we identify $\pi^\vee$ with a quotient of $\Ind_P^G\sigma^\vee$.
Once again, the Jacquet integral $J^{\psi_N}_\sigma$ gives rise to an isomorphism of vector spaces
\[
J_\sigma^{\psi_N}:\JF_{\psi_N}(\Ind_P^G\sigma)\rightarrow\JF_{\Mpsi{M}}^M(\sigma)
\]
\cite[Theorem 15.6.7]{MR1170566}.
We thus define the pairing $(v,v^\vee)_\pi^{\psi_N}$ to be the right-hand side of \eqref{eq: redtosqrint}.
It descends to a pairing between $\JF_{\psi_N}(\pi)$ and $\JF_{\psi_N^{-1}}(\pi^\vee)$.

By abuse of notation, if $f=\mc_{v,v^\vee}$ we will formally write
\[
\stint_Nf(n)\psi_N(n)^{-1}\ dn:=(v,v^\vee)_\pi^{\psi_N}.
\]

\begin{remark}
It will be interesting to find a purely function-theoretic way to define
\[
\stint_Nf(n)\psi_N(n)^{-1}\ dn
\]
as we did in the $p$-adic case.
\end{remark}

If $S$ is a finite set of places, $\pi_S\in\Irr_{\gen,\psi_{N(F_S)}}G(F_S)$,
$u=\otimes_{v\in S}u_v\in\pi_S$, $u^\vee=\otimes_{v\in S}u_v^\vee\in\pi_S^\vee$
we write
\[
(u,u^\vee)_{\pi_S}^{\psi_N}:=\prod_{v\in S}(u_v,u_v^\vee)_{\pi_v}^{\psi_N}.
\]
We extend it to a bilinear form on $\pi_S\times\pi_S^\vee$ by linearity.
Of course, as before the definition depends on a choice of a Haar measure on $N(F_S)$.


\subsection{Metaplectic group} \label{sec: metalocal}
The results of this section have analogues for the metaplectic group $\Mp_n$, the two-fold cover of the rank $n$ symplectic group
$\Sp_n$. (Any representation of $\Mp_n$ will be implicitly assumed to be genuine.)
The maximal unipotent subgroup $N$ of $\Sp_n$ embeds uniquely in $\Mp_n$ and we fix a non-degenerate character $\psi_N$ of $N$.
Uniqueness of Whittaker model in this context was proved by Szpruch in the $p$-adic case \cite{MR2366363}.
In the archimedean case \cite[Ch.~15]{MR1170566} is still applicable.
Propositions \ref{prop: stablemc}, \ref{prop: jacquetdescent} and \ref{prop: nontrivpsi_N} and their proofs hold with minimal changes.
In particular, we can define $(v,v^\vee)_\pi^{\psi_N}$ at least in the $p$-adic case.
In the archimedean case we will define $(v,v^\vee)_\pi^{\psi_N}$ as in \S\ref{sec: arch}
and \emph{assume} that this is unambiguous. (This can probably be checked using the results of \cite{Sz}.)

Assume that $q$ is odd.
Let $\pi$ be an unramified representation of $\Mp_n$.
Then we have 
\begin{equation} \label{eq: metunram}
(v_0,v_0^\vee)_\pi^{\psi_N}=\vol(N\cap K)\frac{(v_0,v_0^\vee)L_{\psi}(\frac12,\pi)\Delta_{\Sp_n}(1)}{L(1,\pi,\Ad)}.
\end{equation}
Here $\psi$ is a character of $F$ depending on $\psi_N$ and the factor $L_{\psi}(\frac12,\pi)$ in the numerator is the
Shimura unramified local factor corresponding to $\pi$ and $\psi$.
It is equal to $L(\frac12,\tau)$ when $\tau$ is the $\psi$-lift of $\pi$ to $\GL_{2n}$.
(Cf. \cite{MR1722953} for precise definitions;
recall that changing $\psi_N$ results in twisting $\tau$ by a quadratic character.)
The factor $L(1,\pi,\Ad)$ in the denominator is defined to be $L(1,\tau,\sym^2)$.
(Alternatively, we could have also defined it directly in terms of the parameters of $\pi$.
In particular, it does not depend on the choice of $\psi_N$.)
The proof of \eqref{eq: metunram} is the same as \eqref{eq: unramcomp}, except that instead of the Casselman--Shalika formula we use
its metaplectic analogue due to Bump--Friedberg--Hoffstein \cite{MR1115113}.

\section{Conjecture about Whittaker coefficients} \label{sec: conjWhit}

\subsection{} \label{sec: cpidef}
Now let us turn to the global case.
That is, $\bf G$ will be a quasi-split group over a number field $F$, $\bf A$ a fixed
maximal $F$-split torus of $\bf G$, ${\bf T}=C_{\bf G}(\bf A)$ and $\bf B=\bf T\ltimes \bf N$ a Borel subgroup (defined over $F$).
Let $\A$ be the ring of adeles of $F$ and let $\psi_N$ be a non-degenerate character of $N(\A)$
which is trivial on $N(F)$. Denote by $\abs{\cdot}_{\A^*}$ the idele norm
$\abs{\cdot}_{\A^*}:\A^*\rightarrow\R_{>0}$.

As usual, let
\[
G(\A)^1=\cap_{\chi}\Ker\abs{\chi}_{\A^*}
\]
where $\chi$ ranges over the lattice of (one-dimensional)
$F$-rational characters of $\bf G$ and we extend $\chi$ to a homomorphism
$\chi:G(\A)\rightarrow\A^*$.
Thus $G(\A)^1$ is normal in $G(\A)$ and the quotient $G(\A)/G(\A)^1$ is isomorphic to
$\R^l$ where $l$ is the rank of the split part of the center of $G$.
We have $\vol(G(F)\bs G(\A)^1)<\infty$.

Denote by $L^2_{\cusp}(G(F)\bs G(\A)^1)$ the cuspidal part of $L^2(G(F)\bs G(\A)^1)$.
We write $\Cusp G$ for the set of equivalence classes of irreducible cuspidal representations of $G(\A)$.


We take the Tamagawa measure on $N(\A)$ so that $\vol(N(F)\bs N(\A))=1$.
Let $\whit^{\psi_N}$ be the $\psi_N$-th Whittaker--Fourier coefficient of a function $\varphi$ on $G(F)\bs G(\A)$, i.e.
\[
\whit^{\psi_N}(g,\varphi)=\int_{N(F)\bs N(\A)}\varphi(ng)\psi_N(n)^{-1}\ dn.
\]
We often write $\whit^{\psi_N}(\varphi)=\whit^{\psi_N}(e,\varphi)$.

Let $\pi$ be an irreducible automorphic cuspidal representation of $G(\A)$ realized in $L^2_{\cusp}(G(F)\bs G(\A))$.
We assume that $\pi$ is $\psi_N$-generic, i.e. that $\whit^{\psi_N}$ does not vanish identically on the space of $\pi$.
We realize the contragredient $\dual\pi$ automorphically as $\{\overline{\varphi}:\varphi\in\pi\}$.
In other words, the pairing $(\cdot,\cdot)_\pi$ is given by
\begin{equation} \label{eq: inner product}
\int_{G(F)\bs G(\A)^1}\varphi(g)\dual\varphi(g)\ dg, \ \ \varphi\in\pi, \dual\varphi\in\dual\pi.
\end{equation}
Note that $\dual\pi$ is $\psi_N^{-1}$-generic.
We will assume that the following property is satisfied.
\begin{equation} \label{eq: Ad propr}
\text{The partial $L$-function $L^S(s,\pi,\Ad)$ has a pole of order $l$ at $s=1$.}
\end{equation}
This is expected to hold in general. It is known for $\GL_n$ (and therefore, for $\SL_n$) and for classical groups
(see \S\ref{sec: classical groups} below).

Let $S$ be a finite set of places containing the archimedean places and such that $\bf G$ and $\psi_N$ are unramified outside $S$.
Let $\K=\prod K_v$ be a maximal compact subgroup of $G(\A)$ which is special at all (finite) $v$ and hyperspecial for all
$v\notin S$. Also set ${\K}^S={\K}\cap G(\A^S)$.
We take the Haar measure on $N(F_S)$ such that $\vol({\K}^S\cap N(\A^S))=1$ with respect to the measure on $N(\A^S)$
which is compatible with the decomposition $N(\A)=N(F_S)\times N(\A^S)$.

If $\pi$ is unramified outside $S$ then $\pi_S:=\pi^{{\K}^S}$, the ${\K}^S$-fixed vectors of $\pi$,
is an irreducible representation of $G(F_S)$. We can identify $\dual{(\pi_S)}$ with $(\dual\pi)_S$
using the pairing $(\cdot,\cdot)_\pi$, i.e., we take $(\cdot,\cdot)_{\pi_S}=(\cdot,\cdot)_\pi$.
Then by local uniqueness of Whittaker model and the non-vanishing of $(\cdot,\cdot)_{\pi_S}^{\psi_N}$
there exists a non-zero constant $\mainconst{\pi}$ such that for any such $S$
and for any $\varphi\in\pi_S=\pi^{{\K}^S}$, $\varphi^\vee\in\pi_S^\vee=(\pi^\vee)^{{\K}^S}$ we have
\begin{equation} \label{def: cpi}
\whit^{\psi_N}(\varphi)\whit^{\psi_N^{-1}}(\varphi^\vee)=
(\mainconst{\pi}\vol(G(F)\bs G(\A)^1))^{-1}\lim_{s\rightarrow 1}\frac{\Delta_G^S(s)}{L^S(s,\pi,\Ad)}(\varphi,\varphi^\vee)_{\pi_S}^{\psi_N}
\end{equation}
where now $\Delta_G^S(s)$ is the partial $L$-function of the dual $M^\vee$ of the motive $M$ of \cite[\S1]{MR1474159}.
For instance, if $\bf G$ is split over $F$ then $\Delta_G^S(s)=\prod_{i=1}^r\zeta_F^S(s+d_i-1)$ where $d_1,\dots,d_r$ are the exponents of $\bf G$
and $\zeta_F^S(s)=\prod_{v\notin S}\zeta_{F_v}(s)$, $\Re s>1$ is the partial Dedekind zeta function.
In general, $\Delta_G^S(s)$ has a pole of order $l$ at $s=1$.
By the unramified computation \eqref{eq: unramcomp} $\mainconst{\pi}$ does not depend on the choice of $S$.
It also does not depend on the choice of Haar measure on $G(\A)$. However, it depends on the automorphic realization of $\pi$,
not just on $\pi$ as an abstract representation, unless of course $\pi$ has multiplicity one in the cuspidal spectrum.

Note that
\begin{equation} \label{eq: twistconst}
\mainconst{\pi\otimes\omega}=\mainconst{\pi}
\end{equation}
for any character $\omega$ of $G(F)\bs G(\A)$.

\begin{remark}
In principle, we could have considered the discrete, rather than cuspidal spectrum.
However, if we admit Arthur's conjectures (see below) then $\whit^{\psi_N}$
vanishes on (the smooth part) of the residual spectrum of $G$ (namely, on the orthogonal
complement of the cuspidal spectrum in the discrete spectrum of $L^2(G(F)\bs G(\A)^1)$).
For the group $G=\GL_m$, the vanishing of $\whit^{\psi_N}$ on the residual spectrum
follows (unconditionally) from the description of the latter by M\oe glin--Waldspurger \cite{MR1026752}.
\end{remark}

A similar relation holds for $\Mp_n$.
As usual, $\Mp_n(\A)$ is the two-fold cover of $\Sp_n(\A)$ which splits over $\Sp_n(F)$.
Let $\bf N$ be the standard maximal unipotent subgroup of $\Sp_n$ and $\psi_N$ a non-degenerate character of $N(\A)$
(viewed as a subgroup of $\Mp_n(\A)$), trivial on $N(F)$.
Let $\psi$ be the corresponding character of $F\bs\A$ as in the local case (see \S\ref{sec: metalocal}).
The pairing on $\Sp_n(F)\bs\Mp_n(\A)$ of two genuine functions $\varphi_1$, $\varphi_2$
is defined by $\int_{\Sp_n(F)\bs\Sp_n(\A)}\varphi_1(g)\varphi_2(g)\ dg$.
Let $\tilde\pi$ be an irreducible genuine cuspidal automorphic representation of $\Mp_n$.
For simplicity assume that the $\psi$-theta lift of $\tilde\pi$ to $\SO(2n-1)$ vanishes.
In this case it is a consequence of the descent method of Ginzburg--Rallis--Soudry \cite{MR2848523} that $L^S(1,\tilde\pi,\Ad)$ is defined.
(See \S\ref{sec: classical groups} below.) By \eqref{eq: metunram}
\begin{equation} \label{eq: metacpi}
\whit^{\psi_N}(\varphi)\whit^{\psi_N^{-1}}(\varphi^\vee)=
(\mainconst{\tilde\pi}\vol(\Sp_n(F)\bs\Mp_n(\A)))^{-1} L_{\psi}^S(\frac12,\tilde\pi)
\frac{\Delta_{\Sp_n}^S(s)}{L^S(1,\tilde\pi,\Ad)}(\varphi,\varphi^\vee)_{\tilde\pi_S}^{\psi_N}
\end{equation}
where $\mainconst{\tilde\pi}$ is independent of $S$ or the Haar measure on $\Sp_n(\A)$.
Here we take the Haar measure on $\Mp_n(\A)$ so that $\vol(\Sp_n(F)\bs\Mp_n(\A))=2\vol(\Sp_n(F)\bs\Sp_n(\A))$.
Note that $\Delta_{\Sp_n}^S(s)=\prod_{i=1}^n\zeta_F^S(2i)$.

The main question that we shall study in this paper is what is the value of $\mainconst{\pi}$
both in the algebraic and the metaplectic case.

\subsection{}
At this point we will assume Arthur's conjectures \cite{MR1021499}.
Actually, here we only care about the discrete spectrum but we will need a slightly stronger form of the conjectures
which is not strictly speaking made explicit in [ibid.].
(See \cite[Conjecture 2A]{MR2331344}, \cite{1203.0039}, \cite{MR2784745} and \cite{MR2320317} for follow-up conjectures.
Our formulation takes these supplements into account, but it is phrased somewhat differently.)
First, we admit the existence of the Langlands group $\Langlands_F$, a locally compact group
whose irreducible $n$-dimensional representations classify cuspidal representations
of $\GL_n(\A)$ \cite{MR546619}. Let $W_F$ be the Weil group of $F$. The group $\Langlands_F$ comes equipped with a surjective homomorphism
\begin{equation} \label{eq: langweil}
\Langlands_F\rightarrow W_F
\end{equation}
whose kernel is a perfect group, as well as with embeddings of the local Weil groups $W_{F_v}\hookrightarrow\Langlands_F$ for any place $v$.
Let $^LG=\widehat G\rtimes W_F$ be the $L$-group of $G$, where $\widehat G$ is the complex dual group of $\bf G$ and $W_F$ acting through
the action of $\Gamma:=\Gal(E/F)$ where $E/F$ is a finite Galois extension over which $\bf G$ splits.

Recall that
\begin{equation} \label{eq: zhatg}
Z(\widehat G)=\Ker[\widehat G\rightarrow \widehat{G_{\SC}}]
\end{equation}
where $\bf G_{\SC}$ is the simply connected cover of the derived group of $\bf G$
(with a natural map $\bf G_{\SC}\rightarrow\bf G$).
We denote by $Z(\widehat G)_u$ the maximal compact subgroup of $Z(\widehat G)$.

By the restriction-inflation sequence the map \eqref{eq: langweil} gives rise to isomorphisms
\[
H^1(\Langlands_F,Z(\widehat G))=H^1(W_F,Z(\widehat G)).
\]
(All cocycles are understood to be continuous; $H^1$ is defined with respect to continuous cocycles.)
Define
\begin{gather*}
\ker^1(\Langlands_F,Z(\widehat G))=\Ker[H^1(\Langlands_F,Z(\widehat G))\rightarrow\prod_v H^1(W_{F_v},Z(\widehat G))],\\
H^1_{\loc}(\Langlands_F,Z(\widehat G))=H^1(\Langlands_F,Z(\widehat G))/\ker^1(\Langlands_F,Z(\widehat G)).
\end{gather*}
Once again, we have
\begin{gather*}
\ker^1(\Langlands_F,Z(\widehat G))=\ker^1(W_F,Z(\widehat G)):=\Ker[H^1(W_F,Z(\widehat G))\rightarrow\prod_v H^1(W_{F_v},Z(\widehat G))],\\
H^1_{\loc}(\Langlands_F,Z(\widehat G))=H^1_{\loc}(W_F,Z(\widehat G)):=H^1(W_F,Z(\widehat G))/\ker^1(W_F,Z(\widehat G)).
\end{gather*}
In particular, $\ker^1(\Langlands_F,Z(\widehat G))$ is finite.
We also write $H^1_{\loc}(\Langlands_F,Z(\widehat G)_u)$ for the image of $H^1(\Langlands_F,Z(\widehat G)_u)$ in
$H^1_{\loc}(\Langlands_F,Z(\widehat G))$. Once again, this coincides with $H^1_{\loc}(W_F,Z(\widehat G)_u)$, defined analogously.
By Lemma \ref{lem: characters}, the group $H^1_{\loc}(W_F,Z(\widehat G))$ (resp., $H^1_{\loc}(W_F,Z(\widehat G)_u)$) is isomorphic to the group of characters
(resp., unitary characters) of $G(F)\bs G(\A)$.

We consider the set $\params(G)$ of elliptic Arthur's parameters.
The elements of $\params(G)$ are equivalence classes of homomorphisms $\phi:\Langlands_F\times\SL_2(\C)\rightarrow\,^LG$
satisfying the following properties:
\begin{enumerate}
\item The composition of $\phi\rest_{\Langlands_F}$ with the canonical map $^LG\rightarrow W_F$ is the map \eqref{eq: langweil}.
\item The projection of the image of $\phi\rest_{\Langlands_F}$ to $\widehat G$ is bounded.
\item The restriction of $\phi$ to $\SL_2(\C)$ (the so-called $\SL_2$-type of $\phi$) is an algebraic homomorphism to $\widehat G$.
\item \label{part: elliptic} The centralizer $C_{\widehat G}(\phi)$ of the image of $\phi$ in $\widehat G$ is finite modulo
$Z(\widehat G)^\Gamma=Z(\widehat G)\cap C_{\widehat G}(\phi)$.
\end{enumerate}
Two such homomorphisms $\phi_1$, $\phi_2$ are equivalent if there exist $s\in\widehat G$ and a $1$-cocycle $z$ of
$\Langlands_F$ in $Z(\widehat G)$ whose class in $H^1(\Langlands_F,Z(\widehat G))$ lies in $\ker^1(\Langlands_F,Z(\widehat G))$
such that $s\phi_1(x,y)s^{-1}=z(x)\phi_2(x,y)$ for all $x\in\Langlands_F$, $y\in\SL_2(\C)$.

Given $\phi\in\params(G)$
let $D_\phi$ be the subgroup of $\widehat G$ consisting of elements $s\in\widehat G$ such that\footnote{Here $[x,y]$ is the commutator
$xyx^{-1}y^{-1}$} $[s,\Img(\phi)]\subset Z(\widehat G)$.
(In particular, $s$ commutes with the image of $\SL_2(\C)$.)
Then $D_\phi$ contains $Z(\widehat G)$ and $C_{\widehat G}(\phi)$ and the map attaching to $s\in D_\phi$ the $1$-cocycle $x\mapsto[s,\phi(x)]$ of $\Langlands_F$ in $Z(\widehat G)$
induces a homomorphism $D_\phi/Z(\widehat G)\rightarrow H^1(\Langlands_F,Z(\widehat G))$ whose kernel can be identified with
$C_{\widehat G}(\phi)/Z(\widehat G)^\Gamma$.
We denote by $S_\phi$ (resp., $\cent_\phi$) the inverse image of $\ker^1(\Langlands_F,Z(\widehat G))$ in $D_\phi$ (resp., $D_\phi/Z(\widehat G)$).
Note that the finiteness of $C_{\widehat G}(\phi)/Z(\widehat G)^\Gamma$ is equivalent to the finiteness of $\cent_\phi$.
If $\Gamma$ is cyclic then in fact $\ker^1(\Langlands_F,Z(\widehat G))=1$ and $\cent_\phi=C_{\widehat G}(\phi)/Z(\widehat G)^\Gamma$.

We will mostly deal with $\phi$'s which are of Ramanujan type, i.e., with trivial $\SL_2$-type.
We denote by $\tparams(G)\subset\params(G)$ the set of parameters of Ramanujan type.
We can think of them as homomorphisms $\phi:\Langlands_F\rightarrow\,^LG$
(Langlands's parameters) up the equivalence relation above (which is simply conjugation if $G$ splits over a cyclic extension)
-- cf.~\cite[\S10]{MR757954}.

The main assertion is the existence of a canonical orthogonal decomposition\footnote{Here $\overline{\canArthur_\phi}$ denotes the $L^2$-closure}
\begin{equation} \label{eq: Arthur decomposition}
L_{\disc}^2(G(F)\bs G(\A)^1)=\mathop{\widehat\oplus}\limits_{\phi\in\params(G)}\overline{\canArthur_\phi}
\end{equation}
into subspaces that are invariant under the adjoint action of $G_{\Gad}(F)$, where the unramified components of the irreducible constituents of $\canArthur_\phi$
are determined by $\phi$ (or more precisely, by the Langlands parameter $w\mapsto\phi(w,\sm{\abs{w}^{\frac12}}00{\abs{w}^{-\frac12}})$
associated to $\phi$).
Of course, this condition by itself \emph{does not} determine $\canArthur_\phi$ uniquely.

Multiplying a homomorphism $\phi:\Langlands_F\times\SL_2(\C)\rightarrow\,^LG$
by a $1$-cocycle of $\Langlands_F$ in $Z(\widehat G)_u$ gives rise to an action, denoted $\alpha\cdot\phi$,
of $H^1_{\loc}(\Langlands_F,Z(\widehat G)_u)$ on $\params(G)$ and $\tparams(G)$.
We will give ourselves that if $\omega$ is the unitary character of $G(F)\bs G(\A)$
corresponding to $\alpha\in H^1_{\loc}(\Langlands_F,Z(\widehat G)_u)$ then
\begin{equation} \label{eq: twistcompatible}
\canArthur_{\alpha\cdot\phi}=\canArthur_\phi\otimes\omega:=\{\varphi\omega:\varphi\in\canArthur_\phi\}.
\end{equation}

If $\phi$ is not of Ramanujan type then $\whit^{\psi_N}$ vanishes on $\canArthur_\phi$
(for local reasons -- see \cite{MR2784745}).
Suppose that $\phi$ is of Ramanujan type. Then every constituent of $\canArthur_\phi$ is tempered
almost everywhere (in fact, conjecturally everywhere \cite{MR2784745}), hence cuspidal \cite{MR733320}.
We will \emph{assume} that the orthogonal complement $\canArthur_\phi^{\psi_N}$ of the space
\[
\{\varphi\in\canArthur_\phi:\whit^{\psi_N}(\cdot,\varphi)\equiv0\}
\]
in $\canArthur_\phi$ is irreducible (and in particular non-zero) and
we will denote by $\pi^{\psi_N}(\phi)$ the irreducible automorphic cuspidal
representation of $G(\A)$ on $\canArthur_\phi^{\psi_N}$. \label{sec: pipsi}


We remark that in the function field case V. Lafforgue has recently obtained (assuming for simplicity that $\ker^1(W_F,Z(\widehat G))=1$)
a canonical decomposition analogous to \eqref{eq: Arthur decomposition} for the cuspidal spectrum,
parameterized by Langlands's parameters (which in this case amount to suitable homomorphisms $\Gal(\bar F/F)\rightarrow\,^LG$
up to conjugation by $\widehat G$) \cite{1209.5352}.
Every parameter which contributes should arise (conjecturally) from an Arthur parameter (which is uniquely determined).
However, it is not clear at this stage whether one can attach a canonical irreducible automorphic cuspidal representation which plays the role of $\pi^{\psi_N}(\phi)$.

\begin{remark}
It will be interesting to give an alternative description, independent of Arthur's conjectures,
of the space $\oplus_{\phi\in\tparams(G)}\canArthur_\phi^{\psi_N}$.
Of course, for $\GL_m$ this space coincides with the cuspidal spectrum.
\end{remark}

Assuming the above setup we can now formulate the conjecture.
\begin{conjecture} \label{conj: main}
For any $\phi\in\tparams(G)$ we have $\mainconst{\pi^{\psi_N}(\phi)}=\abs{\cent_\phi}$.
\end{conjecture}

\begin{remark} \label{rem: indepsi}
Let $\bf T$ be a maximal torus of $\bf G$ normalizing $\bf N$ and let ${\bf T}_{\Gad}={\bf T}/{\bf Z}({\bf G})$.
Then for any $\phi\in\tparams(G)$ and $t\in T_{\Gad}(F)$
we have $\pi^{\psi_N\circ\Ad(t)}(\phi)=\{\varphi\circ\Ad(t^{-1}):\varphi\in\pi^{\psi_N}(\phi)\}$.
Therefore, Conjecture \ref{conj: main} is independent of the choice of $\psi_N$, since $T_{\Gad}(F)$ acts transitively
on the set of non-degenerate characters.

Also, if $\alpha\in H^1_{\loc}(\Langlands_F,Z(\widehat G)_u)$ and $\omega$ is the corresponding unitary character of $G(F)\bs G(\A)$
then by \eqref{eq: twistcompatible}
\begin{equation} \label{eq: gentwist}
\pi^{\psi_N}(\alpha\cdot\phi)=\pi^{\psi_N}(\phi)\otimes\omega.
\end{equation}
This is of course consistent with Conjecture \ref{conj: main} (using \eqref{eq: twistconst})
since $\cent_{\alpha\cdot\phi}=\cent_\phi$.
\end{remark}


Regardless of Arthur's conjectures one can consider, following Piatetski-Shapiro,
the orthogonal complement $L^2_{\cusp,\psi_N}(G(F)\bs G(\A)^1)$ in $L^2_{\cusp}(G(F)\bs G(\A)^1)$ of the subspace
\[
\{\varphi\in L^2_{\cusp}(G(F)\bs G(\A)^1):\whit^{\psi_N}(\cdot,\varphi)\equiv0\text{ almost everywhere}\}.
\]
By local uniqueness of Whittaker model, the space $L^2_{\cusp,\psi_N}(G(F)\bs G(\A)^1)$ is multiplicity free (cf.~\cite{MR546599}).
Hence, we could have tried to formulate our conjectures for the irreducible constituents of $L^2_{\cusp,\psi_N}(G(F)\bs G(\A)^1)$
instead of the hypothetical spaces $\pi^{\psi_N}(\phi)$.
(Note that if $\pi$ is such a constituent then $\dual\pi$ is realized in $L^2_{\cusp,\psi_N^{-1}}(G(F)\bs G(\A)^1)$.)
It is not clear whether in general one can expect a nice formula.
However, in certain cases we can hope to get a handle on the space $L^2_{\cusp,\psi_N}(G(F)\bs G(\A)^1)$ and the constants
$\mainconst{\pi}$ for its constituents. (See the discussion in \S\ref{sec: classical groups} below about classical groups.)

\subsection{} \label{sec: restder}
Consider the case of a connected reductive group $\bf\widetilde G$ defined and quasi-split over $F$ and a connected algebraic subgroup
$\bf G$ of $\bf\widetilde G$ defined over $F$ containing the derived group $\bf\widetilde G^{\der}$ of $\bf\widetilde G$.
This case was considered by Hiraga--Saito in \cite{MR2918491} following Labesse--Langlands \cite{MR540902}.
Let $\tilde\pi$ be an irreducible cuspidal automorphic representation of $\widetilde G(\A)$.
Let $X(\tilde\pi)$ be the group of characters $\omega$ of $\widetilde G(\A)$ which are trivial on $\widetilde G(F)G(\A)$ such that
$\tilde\pi\otimes\omega=\tilde\pi$ (as physical spaces). It is a finite group (\cite[Lemma 4.11]{MR2918491}) which may a priori depend
on the automorphic realization of $\tilde\pi$.
Let $\bf T$ be the torus $\bf\widetilde G/\bf G$.
Note that $\Delta_{\widetilde G}^S(s)=\Delta_T^S(s)\Delta_G^S(s)$.

The following result is essentially proved in \cite{MR2918491}. For convenience we include some details.

\begin{lemma} \label{lem: restG}
Suppose that $\tilde\pi$ is an irreducible cuspidal $\psi_N$-generic representation of $\widetilde G(\A)$ realized on $V_{\tilde\pi}$.
Assume that the space of $\tilde\pi\otimes\omega$ is orthogonal to that of $\tilde\pi$ for any $\omega\notin X(\tilde\pi)$.
(This condition is of course automatically satisfied if the cuspidal multiplicity of $\tilde\pi$ is one.)
Let $V'_{\tilde\pi}=\{\varphi\rest_{G(\A)}:\varphi\in V_{\tilde\pi}\}$. Then
\begin{enumerate}
\item $V'_{\tilde\pi}$ is the direct sum of distinct irreducible cuspidal representations of $G(\A)$.
\item There is a unique $\psi_N$-generic irreducible constituent of $V'_{\tilde\pi}$.
\item If $\pi$ is the $\psi_N$-generic irreducible constituent of $V'_{\tilde\pi}$ then
$\mainconst{\pi}=\abs{X(\tilde\pi)}\mainconst{\tilde\pi}$.
\end{enumerate}
\end{lemma}

\begin{proof}
The first property follows from the fact that for any place $v$, the restriction of $\tilde\pi_v$ to $G(F_v)$ is a direct sum of distinct irreducible
representations (since $\pi_v$ is generic -- see [ibid., Ch. 3]).

It is clear that $\whit^{\psi_N}$ does not vanish on some irreducible constituent of $V'_{\tilde\pi}$.
On the other hand, $\tilde\pi_v$ admits a unique irreducible constituent which is $\psi_{N(F_v)}$-generic.
The second property follows.

Let $\varphi_i$, $i=1,2$ be in the space of $\pi$ and let $\tilde\varphi_i$ be in the space of $\tilde\pi$ such that
$\varphi_i=\tilde\varphi_i\rest_{G(\A)}$.
Replacing $\tilde\varphi_i$ by $\sum_{\omega\in X(\tilde\pi)}\tilde\varphi_i\omega$, we may assume that $\tilde\varphi_i$
is supported in $\cap_{\omega\in X(\tilde\pi)}\Ker\omega$. By the Poisson summation formula we have
\[
\vol(G(F)\bs G(\A)^1)^{-1}(\varphi_1,\varphi_2)_{G(F)\bs G(\A)^1}=\vol(\widetilde G(F)\bs\widetilde G(\A)^1)^{-1}
\sum_\omega(\tilde\varphi_1\omega,\tilde\varphi_2)_{\widetilde G(F)\bs\widetilde G(\A)^1}
\]
where $\omega$ ranges over the characters of the compact abelian group $\widetilde G(F)G(\A)^1\bs\widetilde G(\A)^1$.
By the condition on $\tilde\pi$, only $\omega\in X(\tilde\pi)$ give a (possibly) non-zero contribution. By the conditions
on $\tilde\varphi_i$ we therefore get
\[
\vol(G(F)\bs G(\A)^1)^{-1}(\varphi_1,\varphi_2)_{G(F)\bs G(\A)^1}=\vol(\widetilde G(F)\bs\widetilde G(\A)^1)^{-1}
\abs{X(\tilde\pi)}(\tilde\varphi_1,\tilde\varphi_2)_{\widetilde G(F)\bs\widetilde G(\A)^1}.
\]
The relation between $\mainconst{\pi}$ and $\mainconst{\tilde\pi}$ follows.
\end{proof}

Let us derive an analogous result at the level of parameters.
The embedding $\bf G\subset\bf\widetilde G$ gives rise to a homomorphism $^L\widetilde G\rightarrow\,^LG$ and hence to a map
\begin{equation} \label{eq: prmstG}
\params(\widetilde G)\rightarrow\params(G).
\end{equation}
Note that if $\tilde\phi\mapsto\phi$ under this map then the $\SL_2$ type of $\tilde\phi$ is determined by that of $\phi$.

Let $\tilde\phi\in\params(\widetilde G)$ and define
\[
X(\tilde\phi)=\{\alpha\in\Ker [H^1_{\loc}(\Langlands_F,Z(\widehat{\widetilde G})_u)\rightarrow H^1_{\loc}(\Langlands_F,Z(\widehat G)_u)]:
\alpha\cdot\tilde\phi\sim\tilde\phi\}.
\]
Let $\phi\in\params(G)$ be the image of $\tilde\phi$ under the map \eqref{eq: prmstG}.

\begin{lemma} \label{lem: paramside}
We have a natural short exact sequence
\[
1\rightarrow\cent_{\tilde\phi}\xrightarrow{\iota}\cent_\phi\xrightarrow{\kappa} X(\tilde\phi)\rightarrow1.
\]
Hence, $\abs{\cent_\phi}=\abs{\cent_{\tilde\phi}}\abs{X(\tilde\phi)}$.
\end{lemma}

\begin{proof}
We have a short exact sequence
\[
1\rightarrow\widehat T\rightarrow\widehat{\widetilde G}\xrightarrow{p}\widehat G\rightarrow1.
\]
Since $\bf\widetilde G^{\der}$ is semisimple, $\bf G^{\der}=\bf\widetilde G^{\der}$ and therefore $\bf G_{\SC}=\bf\widetilde G_{\SC}$.
By \eqref{eq: zhatg}, $\widehat T\subset Z(\widehat{\widetilde G})$ and we get a short exact sequence
\[
1\rightarrow\widehat T\rightarrow Z(\widehat{\widetilde G})\rightarrow Z(\widehat G)\rightarrow1.
\]
In other words,
\begin{equation} \label{eq: invimgZ}
\text{$Z(\widehat{\widetilde G})$ is the inverse image under $p$ of $Z(\widehat G)$.}
\end{equation}

The projection $p$ induces a map $\cent_{\tilde\phi}\xrightarrow{\iota}\cent_\phi$.
We claim that $\iota$ is injective. Indeed, suppose that $\tilde s\in S_{\tilde\phi}$
has trivial image in $\cent_\phi$ under $\iota$.
Then $p(s)\in Z(\widehat G)$ and by \eqref{eq: invimgZ} this implies that $s\in Z(G)$.

Next, we define the map $\cent_\phi\xrightarrow{\kappa} X(\tilde\phi)$.
Let $s\in S_\phi$, i.e. $s\in\widehat G$ is such that $x\mapsto [s,\phi(x)]\in Z(\widehat G)$
defines a locally trivial cocycle of $\Langlands_F$ in $Z(\widehat G)$. Let $\tilde s\in\widehat{\widetilde G}$ be any lift of $s$.
Then again by \eqref{eq: invimgZ}, $x\mapsto [\tilde s,\tilde\phi(x)]\in Z(\widehat{\widetilde G})$ so that it defines
an element $\kappa'(s)$ in $H^1(\Langlands_F,Z(\widehat{\widetilde G})_u)$ whose image in $H^1(\Langlands_F,Z(\widehat G)_u)$
is locally trivial. Clearly $\kappa'(s)$ does not depend on the choice of $\tilde s$ and it depends only on the image
$\bar s$ of $s$ in $\cent_\phi$. We define $\kappa(\bar s)$ to be the image of $\kappa'(s)$ in $H^1_{\loc}(\Langlands_F,Z(\widehat{\widetilde G})_u)$.
Since $[\tilde s,\tilde\phi(x)]\tilde\phi(x)=\tilde s\tilde\phi(x)\tilde s^{-1}$ we have $\kappa(\bar s)\in X(\tilde\phi)$.

It is also clear that $\kappa'(s)$ is locally trivial in $H^1(\Langlands_F,Z(\widehat{\widetilde G}))$ if and only if $\tilde s\in S_{\tilde\phi}$.
Finally we show that $\kappa$ is onto. Suppose that $\beta$ is a $1$-cocycle in $H^1(\Langlands_F,Z(\widehat{\widetilde G})_u)$
such that $\beta\cdot\tilde\phi\sim\tilde\phi$ and the image of $\beta$ in $H^1(\Langlands_F,Z(\widehat G)_u)$ is locally trivial.
Then there exists $\tilde s\in\widehat{\widetilde G}$ and a $1$-cocycle $\gamma$ of $\Langlands_F$ in $Z(\widehat{\widetilde G})$
whose image in $H^1(\Langlands_F,Z(\widehat{\widetilde G})_u)$ is locally trivial
such that $\tilde s\tilde\phi(x)\tilde s^{-1}=\beta(x)\gamma(x)\tilde\phi(x)$ for all $x\in\Langlands_F$.
If $s=p(\tilde s)$ we infer that $s\phi(x)s^{-1}=p(\beta(x)\gamma(x))\phi(x)$, so that $s\in S_\phi$
and $\kappa(\bar s)=\beta$.
\end{proof}

Let $\tilde\phi\in\params(\widetilde G)$ and $\phi\in\params(G)$ be as before.
It is natural to assume that $\canArthur_\phi$ is the image of the space $\canArthur_{\tilde\phi}$ under restriction of functions
to $G(\A)$. In view of \cite[Ch.~4]{MR2918491} it is also natural to assume that the map
$H^1(\Langlands_F,\widehat{\widetilde G})\rightarrow H^1(\Langlands_F,\widehat G)$
is onto, so that the map $\params(\widetilde G)\rightarrow\params(G)$ is onto.
For an analogous result for Weil groups see \cite{MR795713}.

\begin{corollary} \label{cor: subgroup}
Assume the above. Then Conjecture \ref{conj: main} holds for $\bf\widetilde G$ if and only if it holds for $\bf G$.
\end{corollary}

Indeed, it follows from our assumptions that if $\tilde\pi=\pi^{\psi_N}(\tilde\phi)$ then the $\psi_N$-generic constituent $\pi$ of the restriction
of $\tilde\pi$ to $G(\A)$ is $\pi^{\psi_N}(\phi)$. Moreover, it follows from Lemma \ref{lem: characters} and \eqref{eq: gentwist} that $X(\tilde\pi)=X(\tilde\phi)$.
By \eqref{eq: Arthur decomposition} and \eqref{eq: twistcompatible} the space of $\tilde\pi\otimes\omega$ is orthogonal to that of
$\tilde\pi$ unless $\omega\in X(\tilde\pi)$. Finally $L^S(s,\tilde\pi,\Ad)=\Delta_T^S(s)L^S(s,\pi,\Ad)$. The corollary follows.

\subsection{} \label{sec: projG}
Consider now the case where $\bf G=\bf\widetilde G/\bf T$ where $T$ is a central torus in $G$ which is \emph{induced}
i.e., it is the product of restriction of scalars over extensions of $F$ of split tori.
Thus we have a short exact sequence
\[
1\rightarrow\widehat G\rightarrow\widehat{\widetilde G}\rightarrow\widehat T\rightarrow1.
\]
In particular, $\widehat{\widetilde G}^{\der}\subset\widehat G$ and therefore
$\widehat{\widetilde G}=Z(\widehat{\widetilde G})\widehat G$. This also implies that the short exact sequence
\[
1\rightarrow Z(\widehat G)\rightarrow Z(\widehat{\widetilde G})\rightarrow\widehat T\rightarrow1
\]
is exact.

We also have $\Delta_{\widetilde G}^S(s)=\Delta_G^S(s)\Delta_T^S(s)$.
\begin{lemma} \label{lem: injindtorus}
The map
\begin{equation} \label{eq: kerinj}
H^1_{\loc}(\Langlands_F,Z(\widehat G))\rightarrow H^1_{\loc}(\Langlands_F,Z(\widehat{\widetilde G}))
\end{equation}
is injective.
\end{lemma}

\begin{proof}
By the assumption on $T$, for any subgroup $\Gamma'$ of $\Gamma$ the map $\widehat T\rightarrow\widehat T^{\Gamma'}$
given by $t\mapsto\prod_{\sigma\in\Gamma'}t^\sigma$ is surjective. It follows that for any surjective homomorphism
of $\Gamma'$-modules $S\rightarrow\widehat T$, the map $S^{\Gamma'}\rightarrow\widehat T^{\Gamma'}$ is also surjective.
In particular, $Z(\widehat{\widetilde G})^{\Gamma'}\rightarrow\widehat T^{\Gamma'}$ is onto and we conclude
that the map
\[
H^1(W_{F_v},Z(\widehat G))\rightarrow H^1(W_{F_v},Z(\widehat{\widetilde G}))
\]
is injective for all $v$. This immediately implies the injectivity of \eqref{eq: kerinj}.
\end{proof}

\begin{corollary}
Let $\phi\in\params(G)$ and let $\tilde\phi$ be the corresponding element in $\params(\widetilde G)$.
Then $\cent_{\tilde\phi}=\cent_{\phi}$.
\end{corollary}

\begin{proof}
Since $Z(\widehat G)=\widehat G\cap Z(\widehat{\widetilde G})$ we need to show that $S_{\tilde\phi}=Z(\widehat{\widetilde G})S_\phi$.
Suppose that $s\in S_{\tilde\phi}$. By changing $s$ by an element of $Z(\widehat{\widetilde G})$ we may assume that $s\in\widehat G$.
By assumption the $1$-cocycle $x\mapsto[s,\phi(x)]$ is locally trivial as an element of
$H^1(\Langlands_F,Z(\widehat{\widetilde G}))$. On the other hand $x\mapsto[s,\phi(x)]$ takes values
in $Z(\widehat{\widetilde G})\cap\widehat G=Z(\widehat G)$. It follows from the injectivity of \eqref{eq: kerinj}
that $s\in S_\phi$ as required.
\end{proof}

We will assume of course that if $\phi$ and $\tilde\phi$ are as above then $\canArthur_{\tilde\phi}$
consists of the pullback of functions in $\canArthur_\phi$ via $\widetilde G(\A)\rightarrow G(\A)$.
In particular, $\tilde\pi:=\pi^{\psi_N}(\tilde\phi)$ is the pullback via $\widetilde G(\A)\rightarrow G(\A)$ of $\tilde\pi:=\pi^{\psi_N}(\phi)$.
since $\widetilde G(\A)\rightarrow G(\A)$ is surjective by our assumption on $T$, the pull-back to $\widetilde G(\A)$ preserves the inner product.
Note that $L^S(s,\tilde\pi,\Ad)=\Delta_T^S(s)L^S(s,\pi,\Ad)$. Therefore $\mainconst{\tilde\pi}=\mainconst{\pi}$.

\begin{corollary} \label{cor: quotient}
Suppose that Conjecture \ref{conj: main} holds for $\bf\widetilde G$. Then it holds for $\bf G$.
\end{corollary}

By taking a $z$-extension (cf.~\cite{MR540901, MR683003}) we infer from Corollaries \ref{cor: subgroup} and \ref{cor: quotient}:

\begin{corollary}
Suppose that Conjecture \ref{conj: main} holds for quasi-split semisimple simply connected groups.
Then it holds for all quasi-split reductive groups.
\end{corollary}

Note that the reasoning is analogous to the argument reducing the computation of the Tamagawa number of a reductive group
to the semisimple simply connected case, i.e. to Weil's conjecture (cf.~\cite{MR631309}, \cite[\S5]{MR757954}).


\section{The $\GL_m$ case} \label{sec: GLn}
Let ${\bf G}={\bf G}_m$ be the group $\GL_m$ over a number field $F$.
Let $ N=N_m$ be the subgroup of upper unitriangular matrices in $G$.
Fix a non-degenerate character $\psi_N$ of $N(\A)$, trivial on $N(F)$.

\begin{theorem} \label{thm: GLn}
We have $\mainconst{\pi}=1$ for any cuspidal irreducible automorphic representation $\pi$ of $\GL_m$.
\end{theorem}

We will prove the theorem below. (See \cite[\S18]{1203.0039} for a similar result.)
The theorem essentially says that Conjecture \ref{conj: main} holds for $\GL_m$ (assuming the existence of Langlands's group)
since in this case $\cent_\phi=1$ for any parameter $\phi$.
Of course we will prove the theorem without assuming Arthur's conjectures.
From the discussion of \S \ref{sec: restder} and \ref{sec: projG} we get the following analogous results for the groups $\PGL_m$ and $\SL_m$.

\begin{corollary} \label{cor: PGLm}
We have $\mainconst{\pi}=1$ for any cuspidal irreducible representation $\pi$ of $\PGL_m$.
\end{corollary}

\begin{corollary} \label{cor: SLm}
Suppose that $\tilde\pi$ is a cuspidal irreducible automorphic representation of $\GL_m$ realized on $V_{\tilde\pi}$.
Let $\pi$ be the unique irreducible constituent of $\SL_m(\A)$ on
\[
V'_{\tilde\pi}=\{\varphi\rest_{\SL_m(\A)}:\varphi\in V_{\tilde\pi}\}
\]
on which $\whit^{\psi_N}$ is non-zero.
Then we have $\mainconst{\pi}=\abs{X(\tilde\pi)}$
where $X(\tilde\pi)$ is the group of Hecke characters $\omega$ such that $\tilde\pi\otimes\omega=\tilde\pi$.
\end{corollary}

We remark that every irreducible constituent of $V'_{\tilde\pi}$ is generic with respect to \emph{some}
non-degenerate character of $N(\A)$.
We also remark that in the case $G'=\SL_2$ we have multiplicity one for $G'$ (\cite{MR1792292}),
but this is not true for $m>2$ \cite{MR1303497} (hence, it is necessary to specify the automorphic realization
of $\pi$).



Let us now prove Theorem \ref{thm: GLn}.
Since we are free to choose $\psi_N$, we
fix a non-trivial continuous character $\psi=\otimes_v\psi_v$ of $\A$ which is trivial on $F$
and take $\psi_N$ to be the character $n\mapsto\psi(n_{1,2}+\dots+n_{m-1,m})$ of $N(\A)$ (trivial on $N(F)$).

We are also free to choose the Haar measure on $G_m(\A)$. It will be convenient to take the Tamagawa measure (for any $m$). \label{sec: Tamagawa}
First, we take the self-dual Haar measure on $\A$ with respect to $\psi$.
This measure does not depend on the choice of $\psi$: it is the Tamagawa measure and satisfies $\vol(F\bs\A)=1$.
On $F_v$ we take the self-dual measure with respect to $\psi_v$.
Thus, the measure on $\A$ is the product measure of the $F_v$'s.
We also take the product measure on any $F_v^k$.
Consider the top degree invariant differential form (gauge form) $\wedge_{i,j=1,\dots,m}dg_{i,j}/\det g^m$ on $G_m$.\footnote{We can ignore the ambiguity of the sign.}
By a standard construction (cf.~\cite{MR0217077}) this gauge form,
together with our choice of measure on $F_v$, give rise to a Haar measure $dg_v$ on $G(F_v)$ for all $v$.
If $v$ is $p$-adic and $\psi_v$ is unramified then the measure of $K_v=G_m(\OO_v)$ is $\Delta_{G_m,v}(1)^{-1}=\prod_{j=1}^m\zeta_{F_v}(j)^{-1}$.
The measure on $G_m(\A)$ is then defined to be
\[
(\res_{s=1}\Delta_{G_m}^S(s))^{-1}\cdot \prod_{v\in S}dg_v\prod_{v\notin S}\Delta_{G_m,v}(1)\,dg_v
\]
where $S$ is any finite set of places containing the archimedean ones.
The definition is independent of $S$.

Recall that $G(\A)^1=\{g\in G(\A):\abs{\det g}_{\A^*}=1\}$.
Let $A_G$ denote the central subgroup of $G(\A)$ consisting of the scalar matrices $\lambda I_m$ with $\lambda\in\R_{>0}$
where we embed $\R$ in $\A$ (and therefore $\R_{>0}$ in $\A^*$) via
$\R\hookrightarrow\A_{\Q}\hookrightarrow\A=\A_{\Q}\otimes_{\Q}F$ (where the second embedding is $x\mapsto x\otimes 1$).
We have $G(\A)=G(\A)^1\times A_G$.
We endow $A_G$ with the pull-back of the standard Haar measure
$\frac{dx}x$ on $\R_{>0}$ (where $dx$ is the standard Lebesgue measure) under
the isomorphism $\abs{\det}_{\A^*}:A_G\rightarrow\R_{>0}$.
Together with the Haar measure on $G(\A)$, this defines a Haar measure on $G(\A)^1$.
It is well known that $\vol(G(F)\bs G(\A)^1)=1$.



Let $\mira=\mira_m$ be the mirabolic subgroup of $G$ consisting of
matrices whose last row is $\xi_m=(0,\dots,0,1)\in F^m$.
We have $\mira\simeq\GL_{m-1}\ltimes F^{m-1}$.
We use this relation to endow $\mira(F_v)$ (and more generally $\mira_j(F_v)$) with local Tamagawa measures.
We have the following integration formula
\begin{equation} \label{eq: int form mira}
\int_{\mira_j\bs\GL_j}\phi(\xi_jg)\abs{\det g}\ dg=\int_{F^j}\phi(\eta)\ d\eta=\hat\phi(0)
\end{equation}
for any $\phi\in C_c(F^j)$.

Let $\pi\in\Cusp G$. By the theory of Rankin--Selberg integrals, for any $\varphi$ in the space of $\pi$
and $\varphi^\vee$ in the space of $\pi^\vee$ we have
\[
(\varphi,\varphi^\vee)_{G(F)\bs G(\A)^1}=\lim_{s\rightarrow 1}\frac{L^S(s,\pi\otimes\pi^\vee)}
{\Delta_G^S(s)}[\whit^{\psi_N}(\cdot,\varphi),\whit^{\psi_N^{-1}}(\cdot,\varphi^\vee)]_S
\]
where $S$ is a sufficiently large finite set of places and
\[
[W,W^\vee]_S=\int_{N_m(F_S)\bs\mira_m(F_S)}W(p)W^\vee(p)\ dp
\]
(cf.~\cite{MR2309989}).
The factor $\Delta_G^S(s)$ shows up because of our choice of measures on $G(\A)$ and $G(F_S)$.
We recall that $[\cdot,\cdot]_S$ defines an invariant pairing.
We also recall that $L^S(s,\pi\otimes\pi^\vee)=L^S(s,\pi,\Ad)$.
Theorem \ref{thm: GLn} therefore reduces to the following local identity.

\begin{lemma} \label{lem: whitrltnGLn}
Let $F$ be a local field of characteristic $0$ and let $\pi$ be an irreducible unitary generic representation
of $G=G(F)$, realized on its Whittaker model $\Whit^{\psi_N}(\pi)$.
Let $[\cdot,\cdot]$ be the pairing
\begin{equation} \label{eq: innerGL_m}
[W,W^\vee]=\int_{N\bs\mira}W(p)W^\vee(p)\ dp,\ \ W\in\Whit^{\psi_N}(\pi),\ W^\vee\in\Whit^{\psi_N^{-1}}(\pi^\vee).
\end{equation}
Then for any $W\in\Whit^{\psi_N}(\pi)$ and $W^\vee\in\Whit^{\psi_N^{-1}}(\pi^\vee)$ we have
\begin{equation} \label{eq: locidentwhit}
[W,W^\vee]^{\psi_N}=W(e)W^\vee(e).
\end{equation}
(For the meaning of the left-hand side cf.~Remark \ref{rem: anydual}.)
\end{lemma}

\begin{proof}
The local identity is proved in a way similar to the unfolding of the global Rankin--Selberg integral.
The main difference is that while in the unfolding process for the global case, certain terms vanish
by cuspidality, in the local case, which is continuous in nature, the analogous terms
do not contribute because they are of measure zero.

Note that the relation \eqref{eq: locidentwhit} is independent of the choice of Haar measure on $N$.
It will be convenient to identify $N$ (as a variety) with $F^{m\choose2}$ (in the obvious way)
and take the corresponding Haar measure.

Suppose first that $\pi=\Ind\sigma$.
Using the relation \eqref{eq: redtosqrint} and the result of \cite[Appendix A]{MR2930996},
the relations \eqref{eq: innerGL_m} for $\pi$ and $\sigma$ become equivalent.
Thus we reduce to the case where $\pi\in\Irr_{\sqr}G$, in which
\[
[W,W^\vee]^{\psi_N}=\int_N[\pi(n)W,W^\vee]\psi_N(n)^{-1}\ dn.
\]
(This is mostly useful in the archimedean case.)

Assume therefore that $\pi\in\Irr_{\sqr}G$.
Since we already know that \eqref{eq: locidentwhit} holds up to a constant,
it suffices to assume that the restrictions of $W$ and $W^\vee$ to $\mira$
are compactly supported modulo $N$ (cf.~\cite{MR0404534, MR2733072} -- in fact, since
we only consider the square-integrable case the archimedean case is elementary).

For any $j=1,\dots,m$ let $\mira_j$ (resp.~$N_j$) be the mirabolic subgroup of $\GL_j$
(resp.~the group of upper unitriangular matrices of $\GL_j$).
We embed $\GL_j$ (and its subgroups) in $\GL_m$ via $g\mapsto\sm g{}{}{I_{m-j}}$.
Let
\begin{align*}
I_j&=\int_{N_j}\big(\int_{N_{j-1}\bs\GL_{j-1}}W(gn)W^\vee(g)\abs{\det g}^{j-m}\ dg\big)\ \psi_{N_j}(n)^{-1}\ dn\\&=
\int_{N_j}\big(\int_{N_j\bs\mira_j}W(gn)W^\vee(g)\abs{\det g}^{j-m}\ dg\big)\ \psi_{N_j}(n)^{-1}\ dn.
\end{align*}
The sought-after identity \eqref{eq: locidentwhit} becomes $I_m=I_1$.
We claim that for any $j=1,\dots,m-1$, if $I_{j+1}$ converges as an iterated integral then so does $I_j$ and $I_j=I_{j+1}$.
Note that $N_{j+1}=N_j\ltimes C_j$ where $C_j\simeq F^j$ is the subgroup
of unipotent matrices in $N_{j+1}$ whose upper left corner of size $j\times j$ is the identity matrix.
Since $\modulus_{\mira_j}=\abs{\det\cdot}$ we can rewrite $I_{j+1}$ as
\begin{multline*}
\int_{N_j}\int_{C_j}\big(\int_{\mira_j\bs\GL_j}\int_{N_j\bs\mira_j}W(pgvu)W^\vee(pg)\abs{\det p}^{-1}\abs{\det pg}^{j+1-m}\ dp\ dg\big)\\
\ \psi_{C_j}(v)^{-1}\ dv\ \psi_{N_j}(u)^{-1}\ du.
\end{multline*}
In order to show the equality $I_{j+1}=I_j$ we will show that
\begin{multline} \label{eq: n2p1}
\int_{C_j}\big(\int_{\mira_j\bs\GL_j}\int_{N_j\bs\mira_j}W(pgvu)W^\vee(pg)\abs{\det p}^{-1}\abs{\det pg}^{j+1-m}\ dp\ dg\big)\
\psi_{C_j}(v)^{-1}\ dv\\=\int_{N_j\bs\mira_j}W(pu)W^\vee(p)\abs{\det p}^{j-m}\ dp.
\end{multline}
To that end, note that $\mira_j$ is the stabilizer of $\psi_{C_j}$ in $\GL_j$, so that we can write the left-hand side of \eqref{eq: n2p1} as
\begin{multline} \label{eq: mira}
\int_{C_j}\big(\int_{\mira_j\bs\GL_j}\int_{N_j\bs\mira_j}W(pgu)W^\vee(pg)\abs{\det p}^{-1}\abs{\det pg}^{j+1-m}\ dp
\ \psi_{C_j}(gvg^{-1})\ dg\big)\\ \ \psi_{C_j}(v)^{-1}\ dv.
\end{multline}
Let $f_u$ be the function on $F^j-\{0\}$ (row vectors) defined by
\[
f_u(\xi_jg)=\int_{N_j\bs\mira_j}W(pgu)W^\vee(pg)\abs{\det pg}^{j-m}\ dp, \ \ g\in\GL_j
\]
where $\xi_j=(0,\dots,0,1)\in F^j$.
This is well defined since $\GL_j$ acts transitively on $F^j-\{0\}$ and
the stabilizer of $\xi_j$ is $\mira_j$.
Moreover, we have $f_u\in C_c^\infty(F^j-\{0\})$ because of the condition on $W$, $W^\vee$.
Extending by $0$, we view $f_u$ as a function in $C_c^\infty(F^j)$.
We view elements of $C_j$ as column vectors of size $j$, and hence, the group $C_j$ itself
as the dual group of $F^j$. We write $\sprod{\cdot}{\cdot}$ for the corresponding pairing on $F^j\times C_j$.
Then \eqref{eq: mira} is
\[
\int_{C_j}\big(\int_{\mira_j\bs\GL_j}f_u(\xi_jg)\psi(\sprod{\xi_jg}v)\abs{\det g}\ dg\big)
\ \psi_{C_j}(v)^{-1}\ dv,
\]
which by \eqref{eq: int form mira} and Fourier inversion is equal to
\[
\int_{C_j}\big(\int_{F^j}f_u(\eta)\psi(\sprod\eta v)\ d\eta\big)\ \psi_{C_j}(v)^{-1}\ dv=
\int_{C_j}\hat f_u(v)\ \psi(\sprod{\xi_j}v)^{-1}\ dv=f_u(\xi_j).
\]
This proves \eqref{eq: n2p1} and completes the proof of the Lemma.
\end{proof}

\begin{remark} \label{rem: cmpcsupp}
In the $p$-adic case the proof above applies directly to any unitarizable $\pi\in\Irr_{\gen}G$ and there is no need to reduce
to the square integrable case.
Indeed, if the restrictions of $W$ and $W^\vee$ to $\mira$ are compactly supported modulo $N$
then $[\pi(\cdot)W,W^\vee]$ is compactly supported on $\mira$ (and in particular, on $N$).
One way to see this is to choose an arbitrary supercuspidal representation $\tau$ of $\GL_m$
and to realize the restrictions $W\rest_\mira$ (resp. $W^\vee\rest_\mira$)
as the restrictions of Whittaker functions $W_1\in\Whit^{\psi_N}(\tau)$ (resp. $W_1^\vee\in\Whit^{\psi_N^{-1}}(\tau^\vee)$).
Then $[\pi(\cdot)W,W^\vee]=[\tau(\cdot)W_1,W_1^\vee]$ on $\mira$.
As in Remark \ref{rem: archJacquet} it will be interesting to know whether an analogous fact holds in the archimedean case, namely
if the restrictions of $W$ and $W^\vee$ to $\mira$ are Schwartz functions modulo $N$,
is the matrix coefficient $[\pi(\cdot)W,W^\vee]$ necessarily a Schwartz function on $\mira$?
\end{remark}

\begin{remark}
For a general $\pi\in\Irr_{\gen}\GL_m$ (not necessarily unitarizable) one
can still define the pairing $[\cdot,\cdot]$ by
\[
[W,W^\vee]=\int_{N\bs\mira}W(p)W^\vee(p)\abs{\det p}^s\ dp\big|_{s=0}
\]
in the sense of analytic continuation (cf.~\cite[Appendix A]{MR2930996}).
Lemma \ref{lem: whitrltnGLn} and its proof remain valid.
\end{remark}

More generally for $i=1,\dots,m$ and any $W\in\Whit^{\psi_N}(\pi)$ and $W^\vee\in\Whit^{\psi_N^{-1}}(\pi^\vee)$, let
\begin{equation}\label{eq: defAi}
A_i(W,W^\vee)=\int_{N_{m-i}\bs\GL_{m-i}}W(g)W^\vee(g)\abs{\det g}^{1-i}\ dg.
\end{equation}
In particular, $A_1(W,W^\vee)=[W,W^\vee]$.
The proof above shows the following.

\begin{lemma}\label{L: Afe}
Let $U_i$ be the unipotent radical of the parabolic subgroup of $\GL_m$ of type $(m-i,1,\dots,1)$
endowed with the Haar measure induced by the identification (of varieties) $U_i\cong F^{{m\choose 2}-{m-i\choose2}}$.
Let $\pi\in\Irr_{\gen}$. Then
\[
A_i(W,W^\vee)=\int_{U_{i-1}}[\pi(u)W,W^\vee]\psi_{U_{i-1}}(u)^{-1}\ du
\]
for any $W\in\Whit_{\psi_N}(\pi)$ and $W^\vee\in\Whit^{\psi_N^{-1}}(\pi^\vee)$ such that
$W\rest_{\mira}\in C_c^\infty(N_m\bs\mira,\psi_N)$ and
$W^\vee\rest_{\mira}\in C_c^\infty(N_m\bs\mira,\psi_N^{-1})$ and the right-hand side converges.
(At least in the $p$-adic case, the last condition is satisfied automatically in view of Remark \ref{rem: cmpcsupp}.)
In particular, if $\pi\in\Irr_{\cusp}G$ (and $F$ is $p$-adic) then this holds for all $W\in\Whit^{\psi_N}(\pi)$ and $W^\vee\in\Whit^{\psi_N^{-1}}(\pi^\vee)$.
\end{lemma}

\begin{proof}
For $j=m-i+1,\dots,m$ let
\begin{align*}
I_j&=\int_{N_j\cap U_{i-1}}\big(\int_{N_{j-1}\bs\GL_{j-1}}W(gn)W^\vee(g)\abs{\det g}^{j-m}\ dg\big)\ \psi_{N_j}(n)^{-1}\ dn\\&=
\int_{N_j\cap U_{i-1}}\big(\int_{N_j\bs\mira_j}W(gn)W^\vee(g)\abs{\det g}^{j-m}\ dg\big)\ \psi_{N_j}(n)^{-1}\ dn.
\end{align*}
As in the proof of Lemma \ref{lem: whitrltnGLn} ones shows that
$I_j$ converges as an iterated integral and $I_j=I_{j+1}$, $j=m-i+1,\dots,m-1$.
The relation $I_{m-i+1}=I_m$ is the required identity.
\end{proof}

\begin{remark}
We refer to \cite{LMao4} for a generalization of Lemma \ref{L: Afe}
and to \cite[Lemme 3.7]{Wald2} for a closely related statement.
\end{remark}

\section{Classical and metaplectic groups} \label{sec: classical groups}
In the case of classical groups we can formulate a variant of Conjecture \ref{conj: main} without appealing to Arthur's conjectures.
Instead, we use the results of Cogdell--Piatetski-Shapiro--Shahidi (\cite{MR2767514}, building on earlier works
of Cogdell--Kim--Piatetski-Shapiro--Shahidi and Kim--Krishnamurthy \cite{MR1863734, MR2075885, MR2127169, MR2149370})
and Ginzburg--Rallis--Soudry (\cite{MR2848523} and previous works cited therein) describing the generic representations in terms of representations of $\GL_n$.
This variant is also applicable for the (non-algebraic) metaplectic group.

More precisely, let $\bf G$ be either $\SO(n)$ (the special orthogonal group of a split or a quasi-split quadratic space),
$\Sp_n$ (the symplectic group of a symplectic space of dimension $2n$),
$\Mp_n$ (the metaplectic double cover of $\Sp_n$) or $\U(n)$ (the quasi-split unitary group
of a hermitian space of dimension $n$).

In the even orthogonal case let $D\in F^*/(F^*)^2$ be the discriminant of the quadratic space
(with the sign convention so that $\bf G$ is split if and only if $D\in(F^*)^2$).
Let $\chi_D$ be the Hecke character $(D,\cdot)$ where $(\cdot,\cdot)$ denotes the quadratic Hilbert symbol.

In the unitary case we write $E$ for the quadratic extension of $F$ over which $\bf G$ splits
(i.e., over which the hermitian space is defined).
In all other cases we set $E=F$.

As usual, fix a maximal unipotent subgroup $\bf N$ of $\bf G$ and a non-degenerate character $\psi_N$ of $N(\A)$ trivial on $N(F)$.
Since we are free to choose $\psi_N$, in the even orthogonal case we will take $\psi_N$ of a special form
as in \cite{MR2848523} so that $(N,\psi_N)$ is preserved under conjugation by an element of order two in $\Orth(2n)$.
We will denote this outer involution by $\theta$. In all other cases except $\SO(2n)$ set $\theta=\id$
(and there is no restriction on $\psi_N$).

Consider the set $\imgset{G}$ whose elements are sets $\{\pi_1,\dots,\pi_k\}$ (mutually inequivalent representations)
where $\pi_i$ is a cuspidal irreducible representation of $\GL_{n_i}(\A_E)$, $i=1,\dots,k$ such that
\begin{enumerate}
\item $n_1+\dots+n_k=m=\begin{cases}2n&{\bf G}=\SO(2n+1), \Mp_n\text{ or }\SO(2n),\\2n+1&{\bf G}=\Sp_n,\\n&{\bf G}=\U(n).\end{cases}$
\item $L(1,\pi_i,r)=\infty$ for all $i$ where $r=\begin{cases}\wedge^2&{\bf G}=\SO(2n+1)\text{ or }\Mp_n,
\\\sym^2&{\bf G}=\Sp_n\text{ or }\SO(2n),\\\Asai^-&{\bf G}=\U(2n),\\\Asai^+&{\bf G}=\U(2n+1).\end{cases}$
\\(We could have considered instead the partial $L$-function since the local $L$-factors are holomorphic $s=1$.)
Here $\Asai^{\pm}$ are the so-called Asai representations (see e.g., \cite{MR2127169, MR2149370} for the precise definition).
In the case of $\Mp_n$ we also require that $L(\frac12,\pi_i)\ne0$ for all $i$.
\item $\prod_{i=1}^k\omega_{\pi_i}=1$ if ${\bf G}=\Sp_n$; $\prod_{i=1}^k\omega_{\pi_i}=\chi_D$ if ${\bf G}=\SO(2n)$.
(In all other cases we automatically have $\omega_{\pi_i}\rest_{\A^*}\equiv1$ for all $i$.)
\end{enumerate}
In all algebraic cases except $\SO(2n)$ the set $\imgset{G}$ corresponds exactly to $\tparams(G)$
and if $\phi$ is the corresponding parameter then $\cent_\phi\simeq(\Z/2\Z)^{k-1}$ (see \cite{Artendclass} and \cite{1206.0882}).
In the case of $\SO(2n)$ the situation is the following.
The set $\imgset{G}$ (resp., $\tparams(G)$) corresponds to the equivalence classes, under conjugation by $\Orth(2n,\C)$ (resp., $\SO(2n,\C)$),
of homomorphism  $\phi:\Langlands_F\rightarrow\Orth(2n,\C)$ with bounded image whose composition with the determinant corresponds
to $\chi_D$ under class field theory.
Thus, there is a surjective map $\tparams(G)\rightarrow\imgset{G}$ whose fiber over $\{\pi_1,\dots,\pi_k\}$ is a singleton unless $n_i$ is even
for all $i$ in which case the fiber consists of two elements.
We have $\cent_\phi\simeq(\Z/2\Z)^l$ where $l=k-1$ if $n_i$ is even for all $i$ and $l=k-2$ otherwise.

We denote by $\Cusp_{\psi_N}G$ the set of irreducible constituents of $L^2_{\cusp,\psi_N}(G(F)\bs G(\A))$;
in the case of $\Mp_n$ we also require that the $\psi$-theta lift to $\SO(2n-1)$ vanishes
where $\psi$ is related to $\psi_N$ as in \cite[\S11]{MR2848523}.
(In the case $n=1$ this excludes the so-called exceptional representations.)

Let $\stdG{G}:\,^LG\rightarrow\,^L(\Res_{E/F}\GL_m)$ be the $L$-homomorphism such that
\[
\stdG{G}\rest_{\widehat G}=\begin{cases}\Sp_n(\C)\hookrightarrow\GL_{2n}(\C)&{\bf G}=\SO(2n+1),\Mp_n\\
\SO(2n+1,\C)\hookrightarrow\GL_{2n+1}(\C)&{\bf G}=\Sp_n,\\
\SO(2n,\C)\hookrightarrow\GL_{2n}(\C)&{\bf G}=\SO(2n),\\
\GL(n,\C)\xrightarrow{g\mapsto(g,J_n^{-1}\,^tg^{-1}J_n)}\GL(n,\C)\times\GL(n,\C)&{\bf G}=\U(n),\end{cases}
\]
for a suitable hermitian form $J_n$ (cf. \cite[\S1]{MR2767514} for more details).
By \cite{MR2767514} and \cite{MR2848523}
for any $\psi_N$-generic (non-exceptional in the metaplectic case)
irreducible cuspidal representation $\pi$ of $G(\A)$
the functorial transfer of $\pi$ to $\GL_N(\A_E)$ under $\stdG{G}$
exists and is of the form $\pi_1\boxplus\dots\boxplus\pi_k$ (isobaric sum) where $\{\pi_1,\dots,\pi_k\}\in\imgset{G}$.
(In the metaplectic case, the functorial transfer is defined with respect to a character $\psi$ which is compatible with $\psi_N$:
it is also compatible with the theta correspondence.)
Note that $L^S(s,\pi,\Ad)=L^S(s,\pi_1\boxplus\dots\boxplus\pi_k,\tilde r)$ where 
$\tilde r=\begin{cases}\sym^2&{\bf G}=\SO(2n+1)\text{ or }\Mp_n,
\\\wedge^2&{\bf G}=\Sp_n\text{ or }\SO(2n),\\\Asai^+&{\bf G}=\U(2n),\\\Asai^-&{\bf G}=\U(2n+1).\end{cases}$
\\Hence, the property \eqref{eq: Ad propr} holds for all $\psi_N$-generic $\pi$ (non-exceptional in the metaplectic case),
so that $\mainconst{\pi}$ is well defined.

In addition, the descent method provides for any $\{\pi_1,\dots,\pi_k\}\in\imgset{G}$, a subrepresentation
$\sigma=\sigma(\{\pi_1,\dots,\pi_k\})$ of $L^2_{\cusp,\psi_N}(G(F)\bs G(\A))$ such that
\begin{enumerate}
\item For any irreducible constituent $\sigma'$ of $\sigma$ we have $\stdG{G}(\sigma')=\pi_1\boxplus\dots\boxplus\pi_k$.
\item No $\psi_N$-generic cuspidal representation whose functorial lift is $\pi_1\boxplus\dots\boxplus\pi_k$ is
orthogonal to $\sigma$ in $L^2(G(F)\bs G(\A))$.
\end{enumerate}
(See \cite[Theorem 11.2]{MR2075885}.)
In the case of $\SO(2n+1)$ it was proved that $\sigma$ is irreducible (see \cite{MR1846354}, which is based on \cite{MR1983781}).
It follows from the results of \cite{MR2330445} that $\sigma$ is irreducible in the case ${\bf G}=\Mp_n$ as well.
It is expected that also in the remaining cases either $\sigma$ is irreducible and $\theta$-invariant or that (in the $\SO(2n)$ case)
$\sigma=\tau\oplus\theta(\tau)$ where $\tau$ is irreducible and $\theta(\tau)\not\simeq\tau$.


In the algebraic cases this construction is closely related to the (hypothetical) representations $\pi^{\psi_N}(\phi)$ of \S\ref{sec: pipsi}.
More precisely, in all cases other than $\SO(2n)$, one expects multiplicity one for $G$ and therefore if
$\{\pi_1,\dots,\pi_k\}\in\imgset{G}$, $\phi\in\tparams(G)$ is the corresponding parameter,
and $\sigma=\sigma^{\psi_N}(\{\pi_1,\dots,\pi_k\})$ then $\sigma=\pi^{\psi_N}(\phi)$.
In the case of $\SO(2n)$ the situation is the following.
If not all $n_i$'s are even then there is a unique $\phi\in\tparams(G)$
above $\{\pi_1,\dots,\pi_k\}$ and as before $\sigma=\pi^{\psi_N}(\phi)$
(which is expected to be irreducible and $\theta$-invariant, cf.~\cite{MR1248702}).
Suppose now that all $n_i$'s are even. Let $\{\phi_1,\phi_2\}$ be the two parameters in $\tparams(G)$ above $\{\pi_1,\dots,\pi_k\}$.
Recall that we expect that either $\sigma$ is irreducible or $\sigma=\tau\oplus\theta(\tau)$ with $\tau$ irreducible and $\theta(\tau)\not\simeq\tau$.
In the latter case $\{\pi^{\psi_N}(\phi_1),\pi^{\psi_N}(\phi_2)\}=\{\tau,\theta(\tau)\}$.
On the other hand, if $\sigma$ is irreducible then $\sigma$ is not one of the $\pi^{\psi_N}(\phi_i)$'s
since they define equivalent representations on which $\whit^{\psi_N}$ is non-zero.
(Another way to say it: $\sigma$ is $\theta$-invariant while $\theta(\pi^{\psi_N}(\phi_1))=\pi^{\psi_N}(\phi_2)$.)
Instead, $\pi^{\psi_N}(\phi_1)\oplus\pi^{\psi_N}(\phi_2)$ is the isotypic component of $\sigma$ in $L^2_{\cusp}(G(F)\bs G(\A))$
and the space of $\sigma$ is
\[
\{\varphi_1+\varphi_2:\varphi_i\in\pi^{\psi_N}(\phi_i), i=1,2, \whit^{\psi_N}(\cdot,\varphi_1)\equiv\whit^{\psi_N}(\cdot,\varphi_2)\}.
\]
It follows that $\mainconst{\sigma}=\frac14(\mainconst{\pi^{\psi_N}(\phi_1)}+\mainconst{\pi^{\psi_N}(\phi_2)})$.

Altogether we can formulate the following variant of Conjecture \ref{conj: main} for classical groups
(which also covers the metaplectic group).
\begin{conjecture} \label{conj: classical}
Suppose that $\pi$ is an irreducible constituent of $\sigma^{\psi_N}(\{\pi_1,\dots,\pi_k\})$
and let $s$ be the size of the stabilizer of $\pi$ under $\{1,\theta\}$. (In particular, $s=1$ unless $G=\SO(2n)$.)
Then $\mainconst{\pi}=2^{k-1}/s$.
\end{conjecture}


Conjecture \ref{conj: classical} is the combination of Conjectures \ref{conj: globalclassical} and \ref{conj: metplectic global}
stated in the introduction of the paper.
In our work in progress we will reduce these conjectures
to local conjectures and prove the latter in the metaplectic case for $p$-adic fields.
(The determination of the global constant in Conjecture \ref{conj: metplectic global} is based on this work.)





\begin{remark}
A similar picture holds for the groups $G=\Gspin(n)$ \cite{MR2219256, 1101.3467, MR2505178, 1110.6788}.
If we fix the central character $\chi$ then the set $\Pi^G(\GL_m)$ (with $m=n$ if $n$ is even and $m=n-1$ if $n$ is odd)
is defined analogously except that
the $L$-function condition is changed to $L(1,\sym^2\pi_i\otimes\chi^{-1})=\infty$
or $L(1,\wedge^2\pi_i\otimes\chi^{-1})=\infty$ depending on the parity of $n$.
Otherwise, Conjecture \ref{conj: classical} is unchanged.
\end{remark}

\section{Some examples} \label{sec: examples}
We end the paper by proving some low rank cases of Conjecture \ref{conj: classical}.
These cases are closely related to the case of the general linear group.
Roughly speaking, we may recast Conjecture \ref{conj: classical} in each case as a relation between the size of a certain group of self-twists of
a cuspidal representations $\pi$ of $\GL_n$ and
the structure of the isobaric decomposition of a certain functorial transfer of $\pi$.
While these relations are only valid for small $n$ and are purely fortuitous, they match up nicely
with known results in the literature about the cuspidality of certain functorial transfers.
(We mention \cite{MR2899809, MR2522032, MR2369494, MR2094113, MR2052020, MR1874921, MR1792292, MR1923967, MR1890650, MR1937203, MR2767509}
for a partial list of results in this direction.)

In the following we will have the situation of \S\ref{sec: restder}.
The group $\bf G$ will be either a classical group $\bf G'$ or a cover thereof by an induced torus
(so that by the analysis of \S\ref{sec: projG} we can replace $\bf G'$ by $\bf G$ without any loss of information).
For any irreducible cuspidal representation $\tilde\pi$ of $\widetilde G(\A)$ we denote by
$X(\tilde\pi)$ the group of characters of $\widetilde G(\A)$, trivial on $G(\A)\widetilde G(F)$,
such that $\tilde\pi\otimes\omega=\tilde\pi$.
In the cases at hand, because of an exceptional isomorphism, the group $\widetilde G$ will be the product of restriction of scalars
of general linear groups. In particular,
\begin{enumerate}
\item $\widetilde G$ has multiplicity one.
\item $\mainconst{\tilde\pi}=1$ for any $\tilde\pi$.
\item \label{num: propconst} $\mainconst{\pi}=\abs{X(\tilde\pi)}$ for the $\psi_N$-generic constituent $\pi$ of $V_{\tilde\pi}'$
(Lemma \ref{lem: restG}).
\end{enumerate}
We will also have a homomorphism $\stdG{\widetilde G}:\,^L\widetilde G\rightarrow\,^L(\Res_{E/F}\GL_m)$,
which factors through an embedding $\stdG{G}$ of $^LG$, where either $E=F$ or $E$ is quadratic extension of $F$.
The restriction of $\stdG{G}$ to $^LG'$ coincides with the embedding $\stdG{G'}$ defined in \S\ref{sec: classical groups}.
If $\stdG{\widetilde G}(\tilde\pi)=\pi_1\boxplus\dots\boxplus\pi_k$ with $\pi_i$ cuspidal representations of $\GL_{n_i}(\A_E)$,
$n_1+\dots+n_k=m$ then we write $\cent_{\tilde\pi}=(\Z/2\Z)^l$ where $l=k-1$ except for the case where $G'=\SO(2n)$
and not all $n_i$'s are even in which $l=k-2$.
Thus, by property \eqref{num: propconst} above, Conjecture \ref{conj: classical} boils down to the equality $\abs{\cent_{\tilde\pi}}=\abs{X(\tilde\pi)}$.
Note that if we have Langlands's group $\Langlands_F$ in our disposal and let $\tilde\phi:\Langlands_F\rightarrow\,^L\tilde G$ be the
Langlands parameter of $\tilde\pi$ then $\pi=\pi^{\psi_N}(\phi)$ where $\phi\in\tparams(G)$ is the composition of $\tilde\phi$
with the projection $^L\widetilde G\rightarrow\,^LG$. That's why the case of $\SO(2n)$ is consistent with Conjecture \ref{conj: classical},
where we considered representations which are not necessary of the form $\pi^{\psi_N}(\phi)$ -- see the discussion before Conjecture \ref{conj: classical}.

\begin{remark} \label{rem: multone}
Introduce two equivalence relations on the set of irreducible cuspidal representations of $\widetilde G(\A)$:
$\pi_1\sim\pi_2$ (resp., $\pi_1\sim_w\pi_2$) if there exists a character $\omega$ of $\tilde G(F)G(\A)\bs \tilde G(\A)$ (resp.,
$G(\A)\bs\tilde G(\A)$) such that $\pi_2\simeq\pi_1\otimes\omega$.
Then by \cite[Theorem 4.13 and Lemma 5.3]{MR2918491}, $G$ has multiplicity one if and only if the two equivalence relations coincide.
\end{remark}

We will use the following notation.
If $E$ is a quadratic extension of $F$ let $\omega_{E/F}$ be the corresponding quadratic character of $\A_F^*$
and (unlike in \S\ref{sec: classical groups}) $\theta$ will denote the non-trivial Galois involution.
If $\mu$ is a Hecke character of $\A_E^*$ then we will write
$\AI_{E/F}(\mu)$ for the corresponding dihedral representation of $\GL_2$.
It is cuspidal unless $\theta(\mu)=\mu$, or equivalently, $\mu$ factors through the norm map.
If $\chi$ is a Hecke character of $\A_F^*$ then we write $\chi_E$ for the Hecke character of $\A_E^*$
given by composing $\chi$ with the norm map.
We also write $\BC_{E/F}(\pi)$ for the base change from $\GL_m(F)$ to $\GL_m(E)$.

\subsection{$\U(1)$, $\SO(2)$, $\SO(3)$, $\Sp_1$}

The simplest case is $G'=\U(1)$, for which we can take $\widetilde G=G=\Res_{E/F}\GL_1$ and $\stdG{\widetilde G}=\id$.
Of course in this case $X(\tilde\pi)$ is always trivial and $\stdG{\widetilde G}(\tilde\pi)$ is cuspidal.

Similarly for $G'=\SO(2)$ we take $\widetilde G=G=\Res_{E/F}\GL_1$ where $E$ is the quadratic \'etale
algebra defined by the discriminant and $\stdG{\widetilde G}$ the two-dimensional representation defining automorphic induction.
In this case $X(\tilde\pi)$ trivial; $\AI_{E/F}\tilde\pi$ is cuspidal if and only if $\tilde\pi$ is not
$\theta$-invariant but in any case $\cent_{\tilde\pi}$ is trivial.

For $G'=\SO(3)$ (split) we take $\widetilde G=G=\GL_2$ and $\stdG{\widetilde G}=\id$.
Then $X(\tilde\pi)$ is always trivial and $\stdG{\widetilde G}(\tilde\pi)$ is always cuspidal.

Consider the case $G'=\Sp_1=\SL_2$. We take $G=G'$, $\widetilde G=\GL_2$ and $\stdG{\widetilde G}$ to be the adjoint representation.
Let $\tilde\pi\in\Cusp\widetilde G$. Here $X(\tilde\pi)$ is as in Corollary \ref{cor: SLm}.
There are three possibilities (cf.~\cite{MR540902}):
\begin{enumerate}
\item $X(\tilde\pi)=1$,
\item $\abs{X(\tilde\pi)}=2$,
\item $\abs{X(\tilde\pi)}=4$.
\end{enumerate}
Let $\Pi$ be the (adjoint) lifting of $\tilde\pi$ to $\GL_3$ \cite{MR533066}.
In the first case $\tilde\pi$ is not dihedral and $\Pi$ is cuspidal.
In the second case, let $E$ be the quadratic extension of $F$ defined by the non-trivial
element $\omega$ of $X(\tilde\pi)$ and let $\theta$ be the Galois involution of $E/F$.
Then $\tilde\pi=\AI_{E/F}(\mu)$ for some Hecke character $\mu$ of $\A_E^*$, $\theta(\mu)/\mu$ is not quadratic
(i.e., not $\theta$-invariant) and $\Pi=\omega\boxplus\AI_{E/F}(\theta(\mu)/\mu)$.
Finally, in the third case the group $X(\tilde\pi)=\{1,\omega_1,\omega_2,\omega_3\}$ defines a biquadratic extension $K$ of $F$ and
$\Pi=\omega_1\boxplus\omega_2\boxplus\omega_3$.
In all cases we have $\abs{X(\tilde\pi)}=\abs{\cent_{\tilde\pi}}$.

Recall that $\SL_2$ has multiplicity one \cite{MR1792292}.
In other words if $\tilde\pi_1,\tilde\pi_2\in\Cusp\GL_2$ and there exists a character $\omega$ of $\A^*$ such that
$\tilde\pi_2=\tilde\pi_1\otimes\omega$ then we can choose such $\omega$ to be trivial on $F^*$ (i.e., a Hecke character).

\subsection{$\U(2)$}
Consider now $G'=\U(2)$. The group $\GU(2)$ of unitary similitudes is the quotient of $\widetilde G=\GL_2\times\Res_{E/F}\GL_1$
by $\GL_1$ embedded diagonally.
In this identification the similitude factor is $(g,x)\mapsto\det g\Nm_{E/F}(x)^{-1}$.
Therefore we take $G$ to be the kernel of $\det g\Nm_{E/F}(x)^{-1}$
and $\stdG{\widetilde G}:\,^L\widetilde G\rightarrow\,^L\Res_{E/F}\GL_2$ so that the functorial transfer of $(\pi,\mu)$
is $\BC_{E/F}(\pi)\otimes\mu$.

Note that $X((\pi,\mu))$ is the group of Hecke characters $\chi$ of $\A^*$ such that $\pi=\pi\otimes\chi$
and $\mu=\mu\chi_E^{-1}$. Hence, $X((\pi,\mu))=1$ unless $\pi\otimes\omega_{E/F}=\pi$, i.e., unless $\pi$ is dihedral with respect to $E/F$,
in which case $X((\pi,\mu))=\{1,\omega_{E/F}\}$. The condition for the non-cuspidality of $\BC_{E/F}(\pi)$ is also
that $\pi$ is dihedral with respect to $E/F$. Hence indeed $\abs{\cent_{(\pi,\mu)}}=\abs{X((\pi,\mu))}$.

We also note that group $G$ (and hence $G'$) has multiplicity one.
This is easy to see from Remark \ref{rem: multone}.
(This is of course not a new result, but we will use the argument later.)
Indeed, suppose that $\pi_1, \pi_2\in\Cusp\GL_2$,
$\mu_1$, $\mu_2$ are two Hecke characters of $\A_E^*$ and there exists a character $\chi$ of $\A_F^*$ such that
$(\pi_1\otimes\chi,\mu_1\chi_E^{-1})\simeq(\pi_2,\mu_2)$.
We need to show that we can take such $\chi$ which is trivial on $F^*$.
Note that our assumption implies that $\mu_1/\mu_2$ is Galois invariant and hence can be written
as $\chi'_E$ for some Hecke character $\chi'$ of $\A_F^*$. Therefore, $\chi_v\circ\Nm_{E_v/F_v}=\chi_v'\circ\Nm_{E_v/F_v}$ for all $v$.
Thus, for all $v$, either $(\pi_1)_v\otimes\chi_v'=(\pi_2)_v$ or (in the case where $v$ is inert) $(\pi_1)_v\otimes\chi_v'=(\pi_2)_v\otimes\omega_{E_v/F_v}$.
It follows that $\BC_{E/F}(\pi_1\otimes\chi')=\BC_{E/F}(\pi_2)$ and therefore either $\pi_2=\pi_1\otimes\chi'$ or $\pi_2=\pi_1\otimes\chi'\omega_{E/F}$.
The conclusion follows.

\subsection{Split $\SO(4)$}
Consider now the case where $G'$ is the split $\SO(2,2)$.
Then $G'$ is the quotient of group
\[
G=\{(g_1,g_2)\in\GL_2\times\GL_2:\det g_1=\det g_2\}
\]
by the center of $\GL_2$ embedded diagonally.
We take $\widetilde G=\GL_2\times\GL_2$ and $\stdG{\widetilde G}$ the tensor product to $\GL_4$.
We denote the functorial transfer of $(\pi_1,\pi_2)$ by $\pi_1\boxtimes\pi_2$.
Its existence was first proved by Ramakrishnan in \cite{MR1792292}.

We have
\[
X((\pi_1,\pi_2))=X'(\pi_1)\cap X'(\pi_2)
\]
where $X'(\pi)$ is the group of Hecke characters $\omega$ such that $\pi\otimes\omega=\pi$.

Let us see that $\abs{\cent_{(\pi_1,\pi_2)}}=\abs{X((\pi_1,\pi_2))}$.
The transfer $\pi_1\boxtimes\pi_2$ is cuspidal unless one of the following conditions holds (\cite{MR1792292}, cf.~also \cite{MR1936579}
and \cite[Appendix]{MR2899809}):
\begin{enumerate}
\item $\pi_1$ is not dihedral (i.e., $X'(\pi_1)=1$) and $\pi_2=\pi_1\otimes\chi$ for some Hecke character $\chi$, or
\item There exists a quadratic extension $E/F$ and Hecke characters $\mu_1$, $\mu_2$ of $\A_E^*$ such that
$\pi_i=\AI_{E/F}(\mu_i)$, $i=1,2$.
\end{enumerate}

In the first case, $X((\pi_1,\pi_2))$ is trivial while $\pi_1\boxtimes\pi_2$ is of type $(3,1)$,
namely it has the form $\sym^2(\pi_1)\otimes\chi\boxplus\omega_\pi\chi$, and $\cent_{(\pi_1,\pi_2)}=1$.

The second case is analyzed in the following lemma.
\begin{lemma} \label{lem: splitSO4}
Suppose that $\pi_i=\AI_{E/F}(\mu_i)$, $i=1,2$.
Then there are two possibilities. If
\begin{equation} \label{eq: excpt}
\theta(\mu_1)/\mu_1=\theta(\mu_2)/\mu_2\text{ and they are quadratic},
\end{equation}
and $K$ is the quadratic extension of $E$ defined by $\theta(\mu_1)/\mu_1$
then $K$ is a biquadratic extension of $F$, $\abs{\cent_{(\pi_1,\pi_2)}}=4$
and $X((\pi_1,\pi_2))$ is the group of Hecke characters of $\A_F^*$
which are trivial on the norms from $K$.
If \eqref{eq: excpt} is not satisfied then $\abs{\cent_{(\pi_1,\pi_2)}}=2$ and $X((\pi_1,\pi_2))=\{1,\omega_{E/F}\}$.
\end{lemma}

\begin{proof}
We have $\pi_1\boxtimes\pi_2=\AI_{E/F}(\mu_1\mu_2)\boxplus\AI_{E/F}(\mu_1\theta(\mu_2))$.
Therefore $\abs{\cent_{(\pi_1,\pi_2)}}=2$ unless neither $\AI_{E/F}(\mu_1\mu_2)$ nor $\AI_{E/F}(\mu_1\theta(\mu_2))$ is cuspidal,
in which case $\abs{\cent_{(\pi_1,\pi_2)}}=4$.
Note that \eqref{eq: excpt} is equivalent to saying that
both $\mu_1\mu_2$ and $\mu_1\theta(\mu_2)$ factor through the norm (i.e., are Galois invariant),
which is in turn equivalent to the condition that neither $\AI_{E/F}(\mu_1\mu_2)$ nor $\AI_{E/F}(\mu_1\theta(\mu_2))$ is cuspidal.

On the other hand, $X((\pi_1,\pi_2))=\{1,\omega_{E/F}\}$ unless \eqref{eq: excpt} is satisfied, in which case $K$ is biquadratic over $F$.
\end{proof}

We mention that the group $G$ (and hence $\SO(4)$) has multiplicity one.
By remark \ref{rem: multone} we need to show that if $\pi_1,\pi_1',\pi_2,\pi_2'\in\Cusp\GL_2$ such that there exists a character $\omega$ of
$\A^*$ such that $\pi_1'\simeq\pi_1\otimes\omega$ and $\pi_2'\simeq\pi_2\otimes\omega^{-1}$, then we can choose such a $\omega$ which is trivial on $F^*$.
We follow the argument of \cite{MR1792292}. The condition implies that $\pi_1\boxtimes\pi_2=\pi_1'\boxtimes\pi_2'$.
By \cite{MR1792292} we have $\pi_1=\pi_1'\otimes\chi_1$ and $\pi_2=\pi_2'\otimes\chi_2$ for some Hecke characters
$\chi_1$ and $\chi_2$. Thus we may assume without loss of generality that $\pi_1'=\pi_1$ and $\pi_2'=\pi_2\otimes\chi$
for some Hecke character $\chi$ (necessarily with $\chi^2=1$).
We can assume that $\chi\not\equiv1$ since otherwise there is nothing to prove.
Similarly, we can assume that $\pi_i\otimes\chi\ne\pi_i$, $i=1,2$.
Consider $\Pi:=\pi_1\boxtimes\pi_2=\pi_1\boxtimes(\pi_2\otimes\chi)$.
Now, $L^S(s,\Pi\otimes\dual\Pi)$ has a pole at $s=1$; it has at least a double pole if moreover $\pi_2$ is a twist of $\pi_1$.
On the other hand, $L^S(s,\Pi\otimes\dual\Pi)$ factors as
\[
L^S(s,\chi)L^S(s,\overline{\Ad}(\pi_1)\otimes\chi)L^S(s,\overline{\Ad}(\pi_2)\otimes\chi)
L^S(s,\overline{\Ad}(\pi_1)\otimes\overline{\Ad}(\pi_2)\otimes\chi).
\]
Here we denote by $\overline{\Ad}$ the Gelbart-Jacquet lifting from $\GL_2$ to $\GL_3$ \cite{MR533066}.
By our assumption on $\pi_i$ and $\chi$ only the last factor can have a pole at $s=1$.
If $\pi_1$ and $\pi_2$ are not both dihedral this pole, necessarily simple, can occur only if $\pi_2$ is a twist of $\pi_1$.
Therefore $\pi_1$ and $\pi_2$ are both dihedral. This case can be analyzed as in \cite[\S4.1]{MR1792292}.
We omit the details.
Thus, we get multiplicity one for $G$.

\subsection{Non-split $\SO(4)$}
Now we turn to the quasi-split group $G'=\SO(3,1)$, pertaining to a quadratic extension $E$ of $F$.
It can be realized as the quotient of
\[
G=\{g\in\GL_2(E):\det g\in F^*\}
\]
by the center of $\GL_2(F)$. We take $\widetilde G=\Res_{E/F}\GL_2$ and $\stdG{\widetilde G}$ to be the four-dimensional twisted tensor homomorphism.
The functorial lift of $\tilde\pi$ to $\GL_4$ is the Asai transfer $\Asai_{E/F}(\tilde\pi)$ \cite{MR2000968}.
$X(\tilde\pi)$ consists of the Hecke characters $\omega$ of $\A_E^*$, trivial $\A_F^*$, such that $\tilde\pi\otimes\omega=\tilde\pi$.

Let us see that $\abs{X(\tilde\pi)}=\abs{\cent_{\tilde\pi}}$.
We follow \cite[Appendix]{MR2899809}. The transfer $\Asai(\tilde\pi)$ is cuspidal unless one of the following conditions holds:
\begin{enumerate}
\item $\tilde\pi$ is not dihedral and $\theta(\tilde\pi)=\dual{\tilde\pi}\otimes\chi$ for some (necessarily $\theta$-invariant) Hecke character $\chi$ of $\A_E^*$.
\item There exists a quadratic extension $K/E$ which is biquadratic over $F$ and a Hecke character $\mu$ of $K$ such that $\tilde\pi=\AI_{K/E}(\mu)$.
\end{enumerate}

In the first case, $\Asai(\tilde\pi)$ is of type $(3,1)$: it is of the form $\pi_1\boxplus\chi'$ where $\BC_{E/F}(\pi_1)=\Ad(\tilde\pi)\otimes\chi$
(which is cuspidal) and $\chi'_E=\chi$. On the other hand $X(\tilde\pi)=1$.

The second case is analyzed in the following lemma which is the non-split analogue of Lemma \ref{lem: splitSO4}.
\begin{lemma} \label{lem: SO31}
Let $K$ be a quadratic extension of $E$ which is biquadratic over $F$, $\mu$ a Hecke character $\A_K^*$
and $\tilde\pi=\AI_{K/E}(\mu)$. Let $\sigma$ be the non-trivial element of $\Gal(K/E)$, $E_1$, $E_2$ the intermediate fields between
$F$ and $K$ other than $E$, $\chi=\sigma(\mu)/\mu$. If
\begin{equation} \label{eq: spclchi}
\chi\rest_{\A_{E_i}^*}\equiv1, \ \ i=1,2,
\end{equation}
and $L$ is the quadratic extension of $K$ defined by $\chi$ then $L$ is a $(\Z/2\Z)^3$-extension of $F$,
$\abs{\cent_{\tilde\pi}}=4$ and $X(\tilde\pi)$ is equal to the group of Hecke characters of $\A_E^*$ which
are trivial on the norms of $L$.
If \eqref{eq: spclchi} is not satisfied then $\abs{\cent_{\tilde\pi}}=2$ and $X(\tilde\pi)=\{1,\omega_{K/E}\}$.
\end{lemma}

\begin{proof}
Let $\sigma_i$ be the non-trivial element of $\Gal(K/E_i)$, $i=1,2$, so that $\sigma_2=\sigma\sigma_1$.
Then $\Asai(\tilde\pi)=\pi_1\boxplus\pi_2$ where $\BC_{E/F}(\pi_i)=\AI_{K/E}(\mu\sigma_i(\mu))$.
Therefore, $\abs{\cent_{\tilde\pi}}=2$ if either $\pi_1$ or $\pi_2$ is cuspidal and $\abs{\cent_{\tilde\pi}}=4$ otherwise.

We first show that \eqref{eq: spclchi} is equivalent to the condition that neither $\pi_1$ nor $\pi_2$ is cuspidal.
Suppose that neither $\pi_1$ nor $\pi_2$ is cuspidal.
Then in particular $\AI_{K/E}(\mu\sigma_i(\mu))=\BC_{E/F}(\pi_i)$ is not cuspidal, $i=1,2$.
Thus, $\mu\sigma_i(\mu)=\sigma(\mu\sigma_i(\mu))$ or equivalently, $\chi$ is $\sigma_{3-i}$-invariant, $i=1,2$.
Therefore, $\chi$ is also $\sigma$-invariant, hence quadratic.

We may write $\mu\sigma_i(\mu)=(\mu_i)_K$ for a Hecke character $\mu_i$ of $\A_E^*$, and then
$\BC_{E/F}(\pi_i)=\AI_{K/E}(\mu\sigma_i(\mu))=\mu_i\boxplus\mu_i\omega_{K/E}$.
Since $\pi_i$ is not cuspidal this means that $\mu_i$ is $\theta$-invariant,
and therefore $\mu_i=(\nu_i)_E$ for a Hecke character $\nu_i$ of $\A_F^*$.
Thus, $(\mu\rest_{\A_{E_i}^*})_K=\mu\sigma_i(\mu)=\nu_i\circ\Nm_{K/F}=((\nu_i)_{E_i})_K$.
It follows that $\mu\rest_{\A_{E_i}^*}$ is equal to either $(\nu_i)_{E_i}$ or $(\nu_i)_{E_i}\omega_{K/E_i}$.
Upon multiplying $\nu_i$ by $\omega_{E/F}$ if necessary we may assume that
$\mu\rest_{\A_{E_i}^*}=(\nu_i)_{E_i}$. (Note that $(\omega_{E/F})_{E_i}=\omega_{K/E_i}$.)
Thus, $\mu\rest_{\A_{E_i}^*}$ is $\Gal(E_i/F)$-invariant, or equivalently
$\chi\rest_{\A_{E_i}^*}\equiv1$.

Conversely, suppose that \eqref{eq: spclchi} is satisfied. First note that this implies that $\chi\circ\Nm_{K/E_i}\equiv1$
and therefore $\sigma_i(\chi)=\chi^{-1}=\sigma(\chi)$, $i=1,2$. Thus, $\chi$ is $\Gal(K/F)$-invariant
and in particular, $\chi$ is quadratic.
As before, it follows from \eqref{eq: spclchi} that $\mu\rest_{\A_{E_i}^*}$ is $\Gal(E_i/F)$-invariant,
$i=1,2$.
Therefore we can write $\mu\rest_{\A_{E_i}^*}=(\nu_i)_{E_i}$ for a Hecke character $\nu_i$ of $\A_F^*$.
Let $\mu_i=(\nu_i)_E$. Then $\mu\sigma_i(\mu)=(\mu_i)_K$ and therefore
$\BC_{E/F}(\pi_i)=\AI_{K/E}(\mu\sigma_i(\mu))=\mu_i\boxplus\mu_i\omega_{K/E}$.
Since $\mu_i$ is $\theta$-invariant, this implies that $\pi_i$ is not cuspidal.

Continue to assume that \eqref{eq: spclchi} holds.
Let $L$ be the quadratic extension of $K$ defined by $\chi$.
Then by \cite[Appendix, Lemma E]{MR2899809}, $L/E_i$ is biquadratic since $\chi\rest_{\A_{E_i}^*}\equiv 1$.
Moreover, $L/F$ is Galois. Indeed, let $L'$ be the normal closure of $L/F$ and let
$\tau\in\Gal(L'/F)$. If $\tau\in\Gal(L'/E)$ then $\tau(L)=L$ since $L/E$ is Galois.
Otherwise $\tau$ induces $\theta$ on $E$, so that $\tau$ induces $\sigma_i$ (for $i=1$ or $2$)
on $K$. The extension $\tau(L)/K$ is determined by the character $\sigma_i(\chi)$.
Since $\sigma_i(\chi)=\chi$ we necessarily have $\tau(L)=L$. Hence $L/F$ is Galois.
Note that any subfield of $L$, other than $K$, contains at most one of the fields $E$, $E_1$, $E_2$.
Therefore, $L$ is a quadratic extension of at least seven different subfields containing $F$
(namely, $K$ and two for each of $E$, $E_1$, $E_2$).
Hence, we necessarily have $\Gal(L/F)=(\Z/2\Z)^3$.
Let $K_1$, $K_2$ be the intermediate fields between $E$ and $L$, other than $K$. Then $K_i/F$ is biquadratic
and therefore $\omega_{K_i/E}\rest_{\A_F^*}\equiv1$.
It follows that $X(\pi)=\{1,\omega_{K/E},\omega_{K_1/E},\omega_{K_2/E}\}$.


Note that if $\tilde\pi=\AI_{K/E}(\mu)$ then $X(\tilde\pi)$ contains $\omega_{K/E}$.
Finally, suppose that $X(\tilde\pi)$ has order bigger than 2. Then $\chi=\sigma(\mu)/\mu$ is a quadratic character of $\A_K^*$
and if $L$ is the corresponding quadratic extension of $K$ then $L/E$ is biquadratic and all the Hecke characters of $\A_E^*$
which are trivial on norms of $L$ are trivial on $\A_F^*$. Let $K_1$, $K_2$ be as above. Then $\omega_{K_i/E}\rest_{\A_F^*}\equiv1$
which implies that $K_i/F$ is biquadratic, $i=1,2$. It follows that $L=K_1K_2$ is Galois over $F$ and $\Gal(L/F)=(\Z/2\Z)^3$.
Thus, $L/E_i$, $i=1,2$ are biquadratic, and hence the restriction of $\chi=\omega_{L/K}$ to $\A_{E_i}^*$ is trivial, namely \eqref{eq: spclchi} holds.
\end{proof}

Once again we show that $G$ (and hence $\SO(3,1)$) has multiplicity one.
By Remark \ref{rem: multone} we need to show that if $\pi_1, \pi_2\in\Cusp\GL_2(\A_E)$
and there exists a character $\omega$ of $\A_E^*$, trivial on $\A_F^*$, such that $\pi_1=\pi_2\otimes\omega$ then we can choose such a character
which is also trivial on $E^*$. The condition implies that $\Asai(\pi_1)=\Asai(\pi_2)$.
Using \cite{MR1792292} we may assume that $\pi_2=\pi_1\otimes\chi$ for some Hecke character $\chi$ of $\A_E^*$.
For simplicity we write $\pi=\pi_1$.
We obtain $\Asai(\pi)=\Asai(\pi)\otimes\chi\rest_{\A_F^*}$.
In particular,
\[
\pi\boxtimes\theta(\pi)=\BC_{E/F}(\Asai(\pi))=\BC_{E/F}(\Asai(\pi)\otimes\chi\rest_{\A_F^*})=\pi\boxtimes\theta(\pi)\otimes\chi\theta(\chi).
\]
As in the previous subsection, we analyze $L^S(s,\pi\boxtimes\theta(\pi)\otimes\dual\pi\boxtimes\theta(\dual\pi)\otimes\chi\theta(\chi))$
which has a pole at $s=1$, and at least a double pole if moreover $\theta(\pi)$ is a twist of $\pi$.
If $\theta(\chi)\chi\not\equiv1$ then we infer that $\pi$ is dihedral.
On the other hand if $\chi\rest_{\A_F^*}=\omega_{E/F}$ then $\pi\boxtimes\theta(\pi)=\BC_{E/F}(\Asai(\pi))$ is neither cuspidal
nor of type $(3,1)$ and by \cite[Appendix, Theorem B]{MR2899809} we deduce once again that $\pi$ is dihedral. Thus if $\pi$ is not dihedral we
get $\chi\rest_{\A_F^*}=1$, namely $\pi\sim\pi\otimes\chi$.
Finally, the case where $\pi$ is dihedral
can also be analyzed using the method of \cite[Appendix]{MR2899809}. We omit the details.

\subsection{}
Finally we consider the example of the split $\SO(6)$.
We identify $\SO(6)$ with the quotient of
\[
G=\{(g,\lambda)\in\GL_4\times\GL_1:\lambda^2\det g=1\}
\]
by the image of $(z,z^{-2}):\GL_1\rightarrow G$. We take $\widetilde G=\GL_4\times\GL_1$ and $\stdG{\widetilde G}$ to be the exterior square times $x\mapsto x^{-1}$
so that $\stdG{\widetilde G}((\pi,\mu))=\wedge^2(\pi)\otimes\mu^{-1}$, an isobaric representation of $\GL_6$
whose existence was proved by Kim \cite{MR1937203}. We have
\[
X((\pi,\mu))=\{\chi\text{ quadratic Hecke character}:\pi\otimes\chi=\pi\}.
\]
Since $\mu$ is irrelevant, we will simply write $X(\pi)$.
We also write $\Pi=\wedge^2\pi$.

We want to verify the equality $\abs{\cent_\pi}=\abs{X(\pi)}$.
There are several cases to consider.
We start with the simplest.

\begin{lemma}
The following conditions are equivalent.
\begin{enumerate}
\item $\cent_\pi$ is non-trivial.
\item There exists a quadratic extension $E/F$ such that $\pi$ is the automorphic induction
of $\varrho\in\Cusp\GL_2(\A_E)$.
\item $\pi\otimes\omega=\pi$ for some non-trivial character $\omega$.
\item $X(\pi)\ne1$.
\end{enumerate}
\end{lemma}

\begin{proof}
For the equivalence of the last three conditions -- see \cite{MR2767509} (which is based on \cite{MR1007299}).

Suppose that $\pi$ is not an automorphic induction. Then by \cite[Proposition 4.2]{MR2767509} we cannot
have a $\GL_2$ in the isobaric decomposition of $\Pi$.
We also claim that we cannot have two $\GL_1$'s in the isobaric decomposition.
Indeed, if $\chi$ is a Hecke character which occurs in the isobaric decomposition of $\Pi$ then necessarily
$\pi=\pi^\vee\otimes\chi\mu$. Hence, if we had two Hecke characters $\chi_1$ and $\chi_2$ in the isobaric decomposition
of $\Pi$ then $\pi\otimes\chi_1\chi_2^{-1}=\pi$ contradicting our assumption that $\pi$ is not an automorphic induction. Thus,
the only options for the isobaric decomposition of $\Pi$ is $6$, $(5,1)$, or $(3,3)$.
In all cases $\cent_\pi$ is trivial.

Conversely, suppose that $\pi=\AI_{E/F}\varrho$ where $\varrho\in\Cusp\GL_2(\A_E)$.
Then
\begin{equation} \label{eq: inducedcase}
\wedge^2\pi=\Asai(\varrho)\otimes\omega_{E/F}\boxplus\AI(\omega_\varrho)
\end{equation}
\cite[\S3]{MR2076595}. It follows that $\cent_\pi$ is non-trivial.
\end{proof}

It remain to consider the case where $\pi=\AI_{E/F}\varrho$ where $\varrho\in\Cusp\GL_2(\A_E)$.
From now on we assume that this is the case.

\begin{lemma}
Assume that $\varrho$ is not dihedral with respect to a biquadratic extension of $F$ (containing $E$). Then either
$\abs{\cent_\pi}=2$ and $X(\pi)=\{1,\omega_{E/F}\}$ or
the following equivalent conditions are satisfied.
\begin{enumerate}
\item $\abs{\cent_\pi}=4$.
\item $\theta(\varrho)=\varrho\otimes\chi_E$ for some quadratic character $\chi$ of $\A_F^*$ (necessarily
different from $\omega_{E/F}$ since $\pi$ is cuspidal).
\item $X(\pi)=\{1,\omega_{E/F},\chi,\chi\omega_{E/F}\}$ for a quadratic character $\chi$ of $\A_F^*$
different from $\omega_{E/F}$.
\end{enumerate}
Moreover, in these cases $\varrho$ is not dihedral.
\end{lemma}

\begin{proof}
Recall \eqref{eq: inducedcase}.
If $\AI(\omega_\varrho)$ is cuspidal, i.e., if $\theta(\omega_\varrho)\ne\omega_\varrho$,
then $\abs{\cent_\pi}=2$ since $\Asai(\varrho)$ is either cuspidal or of type $(3,1)$.
If $\AI(\omega_\varrho)$ is not cuspidal, i.e., if $\theta(\omega_\varrho)=\omega_\varrho$ then
$\abs{\cent_\pi}=2$ unless $\theta(\varrho)=\varrho\otimes\nu$ for some Hecke character $\nu$ of $\A_E^*$
in which case $\abs{\cent_\pi}=4$ and $\varrho$ is not dihedral. In the latter case $\omega_\varrho=\theta(\omega_\varrho)=\omega_\varrho\nu^2$
so that $\nu$ is quadratic. ($\nu$ is non-trivial since $\pi$ is cuspidal.)
Moreover, since $\varrho$ is not dihedral we also have $\nu\theta(\nu)=1$ so that $\theta(\nu)=\nu$.
Thus, $\nu=\chi_E$ for some Hecke character $\chi$ of $\A_F^*$.
Finally, it follows from the relation $\theta(\varrho)=\varrho\otimes\nu$ that $\chi$ is quadratic,
since by \cite{MR1611951} we cannot have $\chi^2=\omega_{E/F}$.

For the last condition suppose that $\chi$ is a quadratic Hecke character of $\A_F^*$ different from $\omega_{E/F}$.
Then the condition
$\AI(\varrho)\otimes\chi=\AI(\varrho)$ is equivalent to $\varrho\otimes\chi_E=\varrho$ or $\theta(\varrho)$.
However, by the assumption on $\varrho$ we cannot have $\varrho\otimes\chi_E=\varrho$ since $\chi_E$ defines a biquadratic extension of $F$.
Thus it follows from the equivalence of the first two parts which was proved above that $\abs{X(\pi)}=2$ if and only if
$\abs{\cent_\pi}=2$ and in this case $\abs{X(\pi)}=4$. The lemma follows.
\end{proof}

Finally, we consider the case where $\varrho=\AI_{K/E}(\mu)$ and $K/F$ is biquadratic.

\begin{lemma}
Assume first that $\mu\rest_{\A_E^*}$ is not $\theta$-invariant.
Let $\sigma$ be the non-trivial element of $\Gal(K/E)$, $E_1$, $E_2$ the intermediate fields between
$F$ and $K$ other than $E$, and $\chi=\sigma(\mu)/\mu$. If
\begin{equation} \label{eq: spclchi2}
\chi\rest_{\A_{E_i}^*}\equiv1, \ \ i=1,2,
\end{equation}
and $L$ is the quadratic extension of $K$ defined by $\chi$ then
\begin{enumerate}
\item $L$ is a $(\Z/2\Z)^3$-extension of $F$,
\item $\abs{\cent_\pi}=8$,
\item $X(\pi)$ is equal to the group of Hecke characters of $\A_F^*$ which
are trivial on the norms of $L$.
\end{enumerate}
If \eqref{eq: spclchi2} is not satisfied then $\abs{\cent_\pi}=4$ and
$X(\pi)$ is equal to the group of Hecke characters of $\A_F^*$ which are trivial on the norms of $K$.

Now assume that $\mu\rest_{\A_E^*}$ is $\theta$-invariant but \eqref{eq: spclchi2} is not satisfied.
Then $\abs{\cent_\pi}=4$ and
$X(\pi)$ is equal to the group of Hecke characters of $\A_F^*$ which are trivial on the norms of $K$
unless we can write $\mu\sigma_i(\mu)=\xi_K$ for $i=1$ or $2$ and some Hecke character $\xi$ of $\A_F^*$
in which case $\abs{\cent_\pi}=8$ and $\abs{X(\pi)}=8$.

Finally, if both $\mu\rest_{\A_E^*}$ is $\theta$-invariant and \eqref{eq: spclchi2} is satisfied
then $\abs{\cent_\pi}=16$ and $\abs{X(\pi)}=16$.
\end{lemma}

\begin{proof}
First note that the $\theta$-invariance of $\mu\rest_{\A_E^*}$ is equivalent to the $\theta$-invariance of $\omega_{\varrho}$
since $\omega_{\varrho}=\mu\rest_{\A_E^*}\omega_{K/E}$.
Thus, if $\mu\rest_{\A_E^*}$ is non-$\theta$-invariant then by \eqref{eq: inducedcase}, $\abs{\cent_\pi}$ is either $8$ or $4$ depending on
whether or not $\Asai(\varrho)$ is of type $(1,1,1,1)$. By Lemma \ref{lem: SO31} this is equivalent to the condition \eqref{eq: spclchi2}.

Recall that $\Asai(\varrho)=\pi_1\boxplus\pi_2$ where $\BC_{E/F}(\pi_i)=\AI_{K/E}(\mu\sigma_i(\mu))$.

Thus, if $\mu\rest_{\A_E^*}$ is $\theta$-invariant then $\abs{\cent_\pi}=2^{2+\epsilon_1+\epsilon_2}$
where $\epsilon_i=0$ if $\pi_i$ is cuspidal and $\epsilon_i=1$ otherwise.
Let us analyze the condition $\AI_{E/F}(\varrho)=\AI_{E/F}(\varrho)\otimes\omega$.
This is equivalent to either $\varrho=\varrho\otimes\omega_E$ or $\theta(\varrho)=\varrho\otimes\omega_E$.
Note that a quadratic character $\mu$ of $\A_E^*$ is of the form $\chi_E$ for a quadratic Hecke character $\chi$ of $\A_F^*$ if and only
if $\mu$ is trivial on $\A_F^*$.
In particular, the group $\{\omega\text{ quadratic}:\varrho\otimes\omega_E=\varrho\}$ is precisely the preimage under $\omega\mapsto\omega_E$ of
$\{\omega\text{ on }\A_E^*/\A_F^*:\varrho\otimes\omega=\varrho\}$.

Once again, the existence of a quadratic Hecke character $\omega$ of $\A_F^*$ such that $\theta(\varrho)=\varrho\otimes\omega_E$
is equivalent to the existence of a quadratic Hecke character $\omega$ on $\A_E^*$ trivial on $\A_F^*$ such that
$\theta(\varrho)=\varrho\otimes\omega$.
In turn, the condition $\theta(\varrho)=\varrho\otimes\omega$ can be written as $\AI_{K/E}(\sigma_1(\mu))=\AI_{K/E}(\mu\omega_K)$
or equivalently as $\sigma_i(\mu)=\mu\omega_K$ for $i=1$ or $2$.

Note that $\pi_i$ is not cuspidal if and only if there exists a $\theta$-invariant Hecke character $\nu$ of $\A_E^*$ such that
$\mu\sigma_i(\mu)=\nu_K$, i.e. if and only if $\mu\sigma_i(\mu)$ factors through $\Nm_{K/F}$.

The lemma will follow from the following claim:
there exists a Hecke character $\omega$ of $\A_E^*$ such that $\omega^2=1$,
$\omega\rest_{\A_F^*}\equiv1$ and $\sigma_{3-i}(\mu)/\mu=\omega_K$ if and only if
$\mu\rest_{\A_E^*}$ is $\theta$-invariant and $\pi_i$ is not cuspidal.

Indeed, suppose that such $\omega$ exists. Then
\[
\frac{\sigma_{3-i}(\mu)}{\mu}\rest_{\A_E^*}=\omega_K\rest_{\A_E^*}=\omega^2=1.
\]
so that $\theta(\mu\rest_{\A_E^*})=\mu\rest_{\A_E^*}$.
Moreover,
\[
\mu\sigma_i(\mu)=\omega_K^{-1}\sigma_i(\mu)\sigma_{3-i}(\mu)=
\omega_K^{-1}(\sigma_i(\mu)\rest_{\A_E^*})_K=\omega_K^{-1}(\theta(\mu\rest_{\A_E^*}))_K=
(\omega^{-1}\mu\rest_{\A_E^*})_K
\]
and both $\omega$ and $\mu\rest_{\A_E^*}$ are $\theta$-invariant.

In the converse direction, suppose that $\mu\sigma_i(\mu)=\nu_K$ where $\nu$ is a $\theta$-invariant
Hecke character of $\A_E^*$ and that $\mu\rest_{\A_E^*}$ is $\theta$-invariant. Then we write $\nu=\lambda_E$
for some Hecke character $\lambda$ on $\A_F^*$, so that $\mu\sigma_i(\mu)=\lambda_K$.
As before we have
\[
\sigma_{3-i}(\mu)/\mu=\nu_K^{-1}\sigma_i(\mu)\sigma_{3-i}(\mu)=\omega_K
\]
where $\omega=\nu^{-1}\mu\rest_{\A_E^*}$. Restricting the relation $\nu_K=\mu\sigma_i(\mu)$
to $\A_E^*$ we get $\nu^2=\mu\rest_{\A_E^*}\theta(\mu\rest_{\A_E^*})=\mu^2\rest_{\A_E^*}$
so that $\omega^2=1$. Also, the relation
\[
(\mu\rest_{\A_{E_i}^*})_K=\mu\sigma_i(\mu)=\lambda_K=(\lambda_{E_i})_K
\]
implies that
\[
\mu\rest_{\A_{E_i}^*}=\lambda_{E_i}\text{ or }\lambda_{E_i}\omega_{K/E_i}.
\]
In both cases,
\[
\mu\rest_{\A_F^*}=\lambda_{E_i}\rest_{\A_F^*}=\lambda^2=\lambda_E\rest_{\A_F^*}=\nu\rest_{\A_F^*}.
\]
We conclude that $\omega\rest_{\A_F^*}\equiv1$ as required.
\end{proof}

\begin{remark}
We do not know whether $\SO(6)$ has multiplicity one.
This is closely related to the question of whether the group $\SO(6,\C)=\SL_4(\C)/\{\pm1\}$ is acceptable in the language of \cite{MR1303498}.
Unfortunately this case was left open in [loc. cit.].
It would be interesting to settle this.

\end{remark}

\appendix

\section{Characters of $G$ over local and global fields\\
by Jean-Pierre Labesse and Erez Lapid} \label{sec: appendix}

Let $G$ be a connected reductive group over either a local or a global field $F$. 
Let $\widehat G$ be the complex dual group of $G$ and $Z(\widehat G)$ its center.
We denote by $Z(\widehat G)_u$ the maximal compact subgroup of $Z(\widehat G)$.

In the local case we have a map
\begin{equation} \label{eq: localmap}
H^1(W_F,Z(\widehat G))\rightarrow\Hom(G(F),\C^*).
\end{equation}
In the global case we write
\[
H^1_{\loc}(W_F,Z(\widehat G))=H^1(W_F,Z(\widehat G))/\Ker[H^1(W_F,Z(\widehat G))\rightarrow\prod_v H^1(W_{F_v},Z(\widehat G))]
\]
and by abuse of notation
\[
\Hom(G(F)\bs G(\A),\C^*)=\Ker[\Hom(G(\A),\C^*)\rightarrow\Hom(G(F),\C^*)].
\]
Then there is a map
\begin{equation} \label{eq: globalmap}
H^1_{\loc}(W_F,Z(\widehat G))\rightarrow\Hom(G(F)\bs G(\A),\C^*).
\end{equation}

The following standard lemma is probably well known. Since we were unable to find a proof in the literature, we include it here.
We thank R. Kottwitz and B. Lemaire for useful discussions.

\begin{lemma} \label{lem: characters}
Suppose that $F$ is a local field.
Then the map \eqref{eq: localmap} is injective if either $F$ is non-archimedean or $G$ is quasi-split.
The map \eqref{eq: localmap} is surjective if $G_{\SC}(F)$ is perfect, i.e.
if either $F=\R$ or $G_{\SC}$ does not contain a simple factor of the form $\Res_{E/F}(\SL_1(D))$
for some finite-dimensional non-commutative division algebra $D$ over a finite separable extension $E$ of $F$.
In particular, \eqref{eq: localmap} is an isomorphism if $G$ is quasi-split.
Moreover, under this isomorphism $H^1(W_F,Z(\widehat G)_u)$ maps to $\Hom(G(F),\unitcirc)$.

Suppose now that $F$ is a global field.
Then the map \eqref{eq: globalmap} is injective provided (in the number field case) that $G$ is quasi-split over all real places.
It is an isomorphism, if in addition $G_{\SC}(F_v)$ is perfect for all non-archimedean places.
Thus, \eqref{eq: globalmap} is an isomorphism if $G$ is quasi-split.
Moreover, under this isomorphism $H^1_{\loc}(W_F,Z(\widehat G)_u)$ maps to $\Hom(G(F)\bs G(\A),\unitcirc)$.
\end{lemma}

\begin{remark}
For convenience we recall the well-known criterion for the perfectness of $G(F)$ where $G$
is a simply connected group over a local field $F$.
Since $G$ is a product of $F$-simple groups we may assume that $G$ is $F$-simple.
If $G$ is quasi-split then the perfectness of $G(F)$ (and in fact, its simplicity modulo
the center, for any field $F$ with more than $3$ elements) follows from \cite{MR787664} since in this case \cite[Theorem A]{MR787664} is very easy and
it reduces the statement to the $F$-rank one case which is elementary -- cf.~\cite{Cassunip}.
(See also \cite[\S1.1.2]{MR521771} and the references therein.)
So suppose that $G$ is arbitrary.
If $F=\R$ the perfectness of $G(F)$ (and again, its simplicity modulo the center) goes back to Cartan -- cf. \cite[Proposition 3.6 and Proposition 7.6]{MR1278263}.
Suppose that $F$ is a non-archimedean local field.
If $G$ is anisotropic then $G=\Res_{E/F}(\SL_1(D))$
where $D$ is a finite-dimensional division algebra over a finite separable extension $E$ of $F$ (see \cite{MR0230838}).
It is well-known that $G(F)=\SL_1(D)$ is not perfect if $D\ne E$ (e.g.~\cite[\S1]{MR939479}).
On the other hand if $G$ is isotropic then the perfectness of $G(F)$ (and moreover the fact that $G(F)$ modulo its center is simple)
is an immediate consequence of the solution of the Kneser-Tits problem for local non-archimedean fields (e.g. \cite[Remark 1.7 and \S2]{MR787664}).
See \cite[Ch.~7]{MR1278263} for more details, at least in characteristic $0$.
\end{remark}

\begin{proof}
We first reduce to the case where $G^{\der}$ is simply connected.
We recall that there exists a $z$-extension, i.e., a central extension of $G$ whose kernel is an induced torus.
This is well known in characteristic $0$ (e.g., \cite[p.~297-9]{MR654325} and \cite[p.~228-9]{MR540901}) but the proof in fact works over any field.
One may start with an arbitrary central extension $G'$ of $G$ by a group of multiplicative type such that the derived group of $G'$ is simply connected.
For instance we can take $G'$ to be the product of $G_{\SC}$ and the largest central torus in $G$, with the obvious map to $G$.
We can embed the (scheme-theoretic) kernel $Z_1$ in an induced torus $Z_2$.
For instance, if $K/F$ is a finite Galois extension such that the Galois action on $X^*(Z_1)$ factors through
$\Gal(K/F)$ then we take $X^*(Z_2)$ to be a free $\Z[\Gal(K/F)]$ module which surjects to $X^*(Z_1)$.
The pushout of $G'\rightarrow G$ with respect to $Z_1\hookrightarrow Z_2$ will be a $z$-extension of $G$.
We thank Robert Kottwitz for this explanation.

Let $\widetilde G\xrightarrow{p} G$ be a $z$-extension of $G$ and let $T$ be the kernel of $p$.
In particular, $\widetilde G^{\der}=G_{\SC}$.
Then we have short exact sequences
\[
1\rightarrow T(F)\rightarrow\widetilde G(F)\rightarrow G(F)\rightarrow 1
\]
and
\[
1\rightarrow Z(\widehat{G})\rightarrow Z(\widehat{\widetilde G})\rightarrow\widehat T\rightarrow1.
\]
If $F$ is local then we get exact sequences
\[
1\rightarrow\Hom(G(F),\C^*)\rightarrow\Hom(\widetilde G(F),\C^*)\rightarrow\Hom(T(F),\C^*)
\]
and
\[
Z(\widehat{\widetilde G})^\Gamma\rightarrow\widehat T^{\Gamma}\rightarrow H^1(W_F,Z(\widehat G))\rightarrow
H^1(W_F,Z(\widehat{\widetilde G}))\rightarrow H^1(W_F,\widehat T).
\]
As in the proof of Lemma \ref{lem: injindtorus}, since $T$ is an induced torus, the first map is surjective and therefore
\[
H^1(W_F,Z(\widehat G))=\Ker[H^1(W_F,Z(\widehat{\widetilde G}))\rightarrow H^1(W_F,\widehat T)].
\]
Similarly, in the global case
\[
1\rightarrow\Hom(G(F)\bs G(\A),\C^*)\rightarrow\Hom(\widetilde G(F)\bs\widetilde G(\A),\C^*)\rightarrow\Hom(T(F)\bs T(\A),\C^*)
\]
is exact. On the other hand by a similar reasoning
\[
H^1_{\loc}(W_F,Z(\widehat G))=\Ker[H^1_{\loc}(W_F,Z(\widehat{\widetilde G}))\rightarrow H^1_{\loc}(W_F,\widehat T)].
\]
Thus the statements for $G$ follow from the statements for $\widetilde G$ and Langlands reciprocity for $T$
(which is elementary in this case since $T$ is induced) \cite{MR1610871}.

Suppose now that $G^{\der}$ is simply connected, i.e., $G^{\der}=G_{\SC}$ and let $T=G/G_{\SC}$.
Then we have an exact sequence
\[
1\rightarrow G_{\SC}(F)\rightarrow G(F)\rightarrow T(F)
\]
Suppose that $F$ is local.
Then the map $\Hom(T(F),\C^*)\rightarrow\Hom(G(F),\C^*)$ is surjective provided that $G_{\SC}(F)$ is perfect.
We claim that the map $G(F)\rightarrow T(F)$ is onto (or equivalently, the map $H^1(F,G_{\SC})\rightarrow H^1(F,G)$
has a trivial kernel) if either $F$ is non-archimedean or $G$ is quasi-split.
The first case follows from Kneser and Bruhat-Tits \cite{MR0424962}.
The second case is also well-known (e.g., \cite[Lemma 32.6]{Gillemcm}). For the convenience of the reader we reproduce the
short argument with the kind permission of Philippe Gille.
Let $S$ be a maximal $F$-split torus of $G_{\SC}$ and let $\widetilde S$ be a maximal $F$-split torus of $G$
containing $S$. Then $C_{G_{\SC}}(S)$ (resp., $C_G(\widetilde S)$) is a maximal torus of $G_{\SC}$ (resp., $G$) and therefore
we have a short exact sequence of tori
\[
1\rightarrow C_{G_{\SC}}(S)\rightarrow C_G(\widetilde S)\rightarrow T\rightarrow1.
\]
Since $C_{G_{\SC}}(S)$ is an induced torus, $H^1(F,C_{G_{\SC}}(S))=1$ and therefore
$C_G(\widetilde S)(F)\rightarrow T(F)$ is onto. A fortiori $G(F)\rightarrow T(F)$ is onto.


Thus, if either $F$ is non-archimedean or $G$ is quasi-split we obtain a short exact sequence
\[
1\rightarrow G_{\SC}(F)\rightarrow G(F)\rightarrow T(F)\rightarrow 1
\]
which gives that the map
\[
\Hom(T(F),\C^*)\rightarrow\Hom(G(F),\C^*)
\]
is injective. On the other hand, since $G_{\SC}$ is simply connected, $\widehat T=Z(\widehat G)$
and by Langlands $\Hom(T(F),\C^*)\simeq H^1(W_F,Z(\widehat T))$
and $\Hom(T(F),\unitcirc)\simeq H^1(W_F,Z(\widehat T)_u)$ \cite{MR1610871}.

Suppose that $F$ is global and (in the number field case) $G$ is quasi-split at all real places. Then the map
$G(F_v)\rightarrow T(F_v)$ is surjective for all $v$. Thus, the map
\begin{equation} \label{eq: pullbackglobal}
\Hom(T(F)\bs T(\A),\C^*)\rightarrow\Hom(G(F)\bs G(\A),\C^*)
\end{equation}
is injective.
Also, $H^1(F,G_{\SC})\rightarrow H^1(F,G)$ has trivial kernel and hence $G(F)\rightarrow T(F)$ is surjective.
This follows from the local case and the injectivity of
\[
H^1(F,G_{\SC})\rightarrow\prod_vH^1(F_v,G_{\SC})
\]
(i.e, the Hasse principle for simply connected groups -- cf.~\cite[Ch.~6]{MR1278263}).
Of course if $G$ is quasi-split over $F$ then the surjectivity was proved above in an elementary way.
Thus, the map \eqref{eq: pullbackglobal} is surjective if $G(F_v)_{\SC}$ is perfect at all places.
Once again the lemma now follows from Langlands reciprocity for $T$.
\end{proof}


We can bypass the use of $z$-extensions and
make the argument slightly more uniform and functorial if we use the language of crossed modules, cf.~\cite[\S1]{MR1695940}.
We will freely use the notation of [ibid.]. Since this reference is written for zero characteristics
some remarks are in order: see Appendix B.

Let $G$ be a connected reductive group over some field $F$.
Recall that  $G_{\ab}$ denotes the crossed module defined by the small complex
$[G_{\SC}\to G]$.
Let $T_{\SC}$ be any maximal torus in $G_{\SC}$ and let $T$ be the centralizer
of its image in $G$.  Since the map
$$[T_{\SC}\to T]\to[G_{\SC}\to G]$$
is a quasi-isomorphism (for Galois action)
we have a commutative diagram with exact rows
\[
\begin{CD}
T_{\SC}(F)@>>> T(F)@>>> H^0_{}(F,[T_{\SC}\to T])@>>> H^1_{}(F,T_{\SC})\\
@VVV     @VVV   @VVV  @VVV\\
G_{\SC}(F)@>>> G(F)@>>> H^0_{}(F,G_{\ab})@>>> H^1_{}(F,G_{\SC})
\end{CD}
\]
and an isomorphism
\[
H^0_{}(F,[T_{\SC}\to T])\to H^0_{}(F,G_{\ab})\,\,.
\]
Denote by  $G(F)^{\ab}$  the quotient of $G(F)$ by its group of commutators.
Since $H^0_{}(F,G_{\ab})$ is an abelian group
one gets a homomorphism \[
G(F)^{\ab}\to H^0_{}(F,G_{\ab})
\]
which is injective when $G_{\SC}(F)$ is perfect.

Assume $G$ is quasisplit. Then if we choose $T_{\SC}$ to be maximally split
(i.e. $T_{\SC}=C_{G_{\SC}}(S)$)
such a torus is induced and hence $H^1(F,T_{\SC})$ is trivial. This and the diagram above  imply that the map
\[
H^0(F,G_{\ab})\to H^1(F,G_{\SC})
\]
is also trivial and we get a surjective map
\[
G(F)^{\ab}\to H^0(F,G_{\ab})\,\,.
\]
On the other hand, if $G$ is arbitrary but $F$ is $p$-adic
then, thanks to Kneser's theorem, $H^1_{}(F,G_{\SC})$ is trivial and one has  again a
surjective map
\[
G(F)^{\ab}\to H^0_{}(F,G_{\ab})\,\,.
\]

Now assume that $F$ is global and $G$ is quasi-split at archimedean places. Then, by the remarks above,
one has a surjective map
\[
G(\mathbb A_F)^{\ab}\to H^0(\mathbb A_F,G_{\ab})
\] which is bijective, if in addition $G_{\SC}(F_v)$ is perfect for all non-archimedean places.
In view of the commutative diagram,
\[
\begin{CD}
 G(F)@>>>G(\A_F)@.\\
 @VVV @VVV @. \\
 H^0(F,G_{\ab})@>>>  H^0(\A_F,G_{\ab})@>>> H^0(\A_F/F,G_{\ab})\\
\end{CD}
\]
with exact rows,
we get a surjective map
\[
G(F)\bs G(\mathbb A_F)\to H^0_{\loc}(\mathbb A_F/F,G_{\ab})
\]
where
\[
H^0_{\loc}(\mathbb A_F/F,G_{\ab})=\Img [H^0(\mathbb A_F,G_{\ab})\to H^0(\mathbb A_F/F,G_{\ab})].
\]
Now, to construct the maps
 \eqref{eq: localmap} and \eqref{eq: globalmap}
  and prove the assertions of Lemma  \ref{lem: characters}
 we need the following
variant of Langlands duality for tori.
\begin{lemma} The group
$H^1(W_F,Z(\widehat G))$
is the group of quasi-characters of
$H^0(F,G_{\ab})$
when $F$ is local  and  $H^1_{\loc}(W_F,Z(\widehat G))$
is the group of quasi-characters of
$H^0_{\loc}(\mathbb A_F/F,G_{\ab})$ when $F$ is global.
\end{lemma}
\begin{proof}Up to quasi-isomorphism, the crossed module $G_{\ab}$ is insensitive to inner automorphism
and the same is true for $Z(\widehat G)$, hence to establish the lemma we may and will assume $G$ quasi-split
and we may and will choose $T_{\SC}$ induced. Let $U=T_{\SC}$. We have quasi-isomorphisms
$G_{\ab}\simeq[U\to T]$ and $ [Z(\widehat G)\to1]\simeq [\widehat T\to \widehat U]$.
Then, since $U$ is induced
\[H^0(F,G_{\ab})=\Coker [U(F)\to T(F)]
\] when $F$ is local and
\[H^0_{\loc}(\A_F/F,G_{\ab})=\Coker [U(\A_F)/U(F)\to T(\A_F)/T(F)] \]
when $F$ is global. On the other hand
\[H^1_{\loc}(W_F,Z(\widehat G))=\Ker[H^1_{\loc}(W_F,\widehat T)\to H^1_{\loc}(W_F,\widehat U)]\]
and the statement  follows from Langlands duality for $T$ and $U$.
\end{proof}

\begin{remark}
The injectivity of $H^1(W_F,Z(\widehat G))\rightarrow\Hom(G(F),\C^*)$ does not hold for $F=\R$ and general $G$.
For instance, if $G$ is the multiplicative group of the quaternions
then the derived group of $G$  is simply connected
and the quotient \[G(\R)/G_{\SC}(\R)=\mathbb H^\times/SU(2)\]
is
$\R_{>0}^\times$ while $H^1(W_F,Z(\widehat G))=H^1(W_F,\C^*)$ is isomorphic to the group of
characters of $\R^\times$.

On the other hand, in the local case the perfectness of $G_{\SC}(F)$ is not a necessary condition for the surjectivity of
$H^1(W_F,Z(\widehat G))\rightarrow\Hom(G(F),\C^*)$. Indeed if $G$ is the multiplicative group $D^*$ of a division
algebra $D$ over a non-archimedean local field then its derived group consists of the elements of reduced norm one
(see \cite[\S1.4.3]{MR1278263}, which is based on \cite{MR0379444}).
Therefore the characters of $D^*$ factor through the reduced norm.
\end{remark}

\section{On abelianized cohomology for reductive groups\\
by Jean-Pierre Labesse and Bertrand Lemaire}

There are two possible definitions for the abelianized cohomology of connected reductive groups.
The first one, due to Borovoi \cite{MR1401491} and still used by him and his coworkers (e.g.
\cite{1303.6586}),
is based on the use of the complex of groups of multiplicative type $[Z_{\SC}\to Z]$
where $Z_{\SC}$ and $Z$ are the  (scheme-theoretic)
centers of $G_{\SC}$ and $G$ respectively: by definition
\[H^i_{\ab}(F,G):=H^i_{\flf}(\Spec F,[Z_{\SC}\to Z])
\]
where $H^i_{\flf}$ is the flat cohomology\footnote{Flat cohomology can be computed using small or big flat sites; but
both sites yield the same cohomology group; the same applies to \'etale cohomology
 (cf. \cite{MR559531} Remark III.3.2.(a) p.111).}.
For  a complex of such group schemes the use of flat instead of Galois cohomology is essential in positive characteristics.

The second definition, given and used in \cite[\S1]{MR1695940} and also introduced in \cite{MR1155229}
relies on the crossed module \[G_{\ab}:=[G_{\SC}\to G]\]
and one considers the abelian groups defined by the Galois cohomology (in any characteristic)
of the crossed module $G_{\ab}$:
\[H^i(F,G_{\ab})=H^i(F,[G_{\SC}\to G])\] defined by an elementary cocycle construction
in degree $i\le 1$. It is then immediate that one has  canonical morphisms
\[H^i(F,G)\to H^i(F,G_{\ab})\]
in degree $i\le 1$ while the existence of a natural map
\[H^i(F,G)\to H^i_{\ab}(F,G)\] is not obvious: it
relies on the use of $z$-extensions and some of the following remarks.

Now, if $T_{\SC}$ is a maximal torus in $G_{\SC}$ and $T$ the centralizer of its image in $G$
there is a quasi-isomorphism for Galois actions (\'etale topology)
\[[T_{\SC}\to T]\to [G_{\SC}\to G]
\] and hence
\[H^i(F,G_{\ab})\simeq H^i(F,[T_{\SC}\to T])=H^i_{\et}(\Spec F,[T_{\SC}\to T])\,\,.\]
This allows to compute $H^i(F,G_{\ab})$ in many cases.
On the other hand one has a quasi-isomorphism 
in flat topology
\[[Z_{\SC}\to Z]\to [T_{\SC}\to T]\,\,.
\]
Now, since \'etale and flat cohomology coincide for (complexes of) tori
(cf. \cite{MR0244271} Thm. 11.7  p.~180 or \cite{MR559531} Thm. III.3.9 p.114)
we get the following sequence of isomorphisms:
\[H^i_{\flf}(\Spec F,[Z_{\SC}\to Z])
\simeq H^i_{\flf}(\Spec F,[T_{\SC}\to T])\simeq H^i_{\et}(\Spec F,[T_{\SC}\to T])
\]
and hence a canonical isomorphism in degree $i\le 1$
\[H^i_{\ab}(F,G)\simeq H^i(F,G_{\ab})\,\,.
\]

We emphasize that avoiding the small complex $[Z_{\SC}\to Z]$ allows to work
with Galois cohomology all along, thus bypassing highly technical issues like flat cohomology and comparison theorems.
For example, the paper \cite{MR697075} by Kottwitz can be simplified by eliminating all references
to \'etale or flat cohomology. Let $G_{\Gad}$ be the adjoint group of some connected reductive group $G$.
The basic construction in [ibid.] is a natural map in Galois cohomology:
\[f:H^1(F,G_{\Gad})\to H^2(F,\mathbb G_m)\,\,.
\] It can be obtained as follows: one has a series of natural morphisms in Galois cohomology
\[H^1(F,G_{\Gad})\to H^1(F,[G_{\SC}\to G_{\Gad}])\simeq H^1(F,[T_{\SC}\to T_{\Gad}])\to H^2(F,T_{\SC})\]
and a map induced by the half sum of positive roots (with respect to a fixed basis)
\[H^2(F,T_{\SC})\to H^2(F,\mathbb G_m)\,\,.\]
The map  $f$ is the compositum of the above maps. Now consider the diagram
\[
\begin{CD}
 H^1(F,G_{\Gad})@>>> H^2(F,T_{\SC})@>>> H^2(F,T_{\Gad}) \\
 @. @VVV @VVV  \\
@.  H^2(F,\mathbb G_m)@>>> H^2(F,\mathbb G_m)\\
\end{CD}
\]
where  the last vertical map is
induced by the sum of positive root  while
the map on the second line is $x\mapsto x^2$.
The diagram is commutative and
the compositum of maps in the first line is zero.
This shows that
$f$ is independent of the choice of the root basis
and elements in its image are of order 2.

\end{document}